\documentclass{jfl}

\usepackage[arrow, matrix, curve]{xy}

\usepackage{amssymb}

\usepackage{tikz}

\numberwithin{equation}{section}
\numberwithin{figure}{section}

\newcommand{\mc}[1]{\ensuremath{\mathcal{#1}}}

\newcommand{\calL}{\mathcal L}

\newcommand{\calS}{\mathcal S}

\newcommand{\bN}{\mathbb N}
\newcommand{\bQ}{\mathbb Q}
\newcommand{\bR}{\mathbb R}
\newcommand{\bZ}{\mathbb Z}

\newcommand{\acl}{\mathrm {acl}}

\newcommand{\eq}{\mathrm {eq}}

\newcommand{\Th}{\mathrm {Th}}
\newcommand{\tp}{\mathrm {tp}}

\newcommand{\ph}{\varphi}

\newcommand\vc{\operatorname{vc}}
\newcommand\VC{\operatorname{VC}}

\newcommand\ACF{\operatorname{ACF}}

\newcommand\MRk{\operatorname{MR}}
\newcommand\URk{\operatorname{U}}

\newcommand\breadth{\operatorname{breadth}}
\newcommand\width{\operatorname{width}}
\newcommand\height{\operatorname{height}}
\newcommand\Gdim{\operatorname{Gdim}}
\newcommand\PP{\operatorname{PP}}

\newcommand\restrict{\!\upharpoonright\!}

\newcommand{\abs}[1]{\lvert#1\rvert}

\DeclareMathAlphabet{\mathbf}{OML}{cmm}{b}{it}

\theoremstyle{definition}

\theoremstyle{remark}
\newtheorem*{remark*}{Remark}
\newtheorem*{remarks*}{Remarks}

\newtheorem*{problem*}{Problem}
\newtheorem*{example*}{Example}
\newtheorem*{examples*}{Examples}
\newtheorem*{claim*}{Claim}

\newtheorem{examples}[theorem]{Examples}

\begin{document}

\title[VC Density in some NIP Theories, II]{Vapnik-Chervonenkis Density in some Theories without the Independence Property, II}

\dedicatory{For Anand Pillay, on his 60th birthday.}

\thanks{Work on this paper was begun while some of the authors were participating in the thematic program on O-minimal Structures and Real Analytic Geometry at the Fields Institute in Toronto (Spring 2009), and attending the 
Durham Symposium on
New Directions in the Model Theory of Fields (July 2009), organized by the London Mathematical Society  and funded by EPSRC~grant EP/F068751/1. The support of these institutions is gratefully acknowledged.
Aschen\-bren\-ner was partly supported by NSF grant DMS-0556197. He would also like to express his gratitude to Andreas Baudisch and Humboldt-Universit\"at Berlin for their hospitality during Fall~2010. Haskell's research was supported by  NSERC grant~238875, Macpherson acknowledges support by EPSRC grant EP/F009712/1, and Starchenko was partly supported by NSF grant DMS-0701364.
We also thank the referees for the careful reading of the manuscript, in particular, for detecting an error in the original formulation and proof of Proposition~\ref{prop:dp-minimal abelian}.}

\author[Aschenbrenner]{Matthias Aschenbrenner}
\address{Department of Mathematics \\
University of California, Los Angeles \\ 
Box 951555 \\
Los Angeles, CA 90095-1555, U.S.A.}
\email{matthias@math.ucla.edu}

\author[Dolich]{Alf Dolich}
\address{Department of Mathematics and Computer Science \\
Kingsborough Community College \\
2001 Oriental Blvd. \\
Brooklyn NY 11235, U.S.A.}
\email{alfredo.dolich@kbcc.cuny.edu}

\author[Haskell]{Deirdre Haskell}
\address{Department of Mathematics and Statistics\\
	  McMaster University \\
	  1280 Main St W \\
         Hamilton ON L8S 4K1, Canada} 
\email{haskell@math.mcmaster.ca}

\author[Macpherson]{Dugald Macpherson}
\address{School of Mathematics\\
University of Leeds\\
Leeds LS2 9JT, U.K.}
\email{h.d.macpherson@leeds.ac.uk}

\author[Starchenko]{Sergei Starchenko}
\address{Department of Mathematics \\
University of Notre Dame\\
255 Hurley Building\\
Notre Dame, IN 46556-4618, U.S.A.}
\email{starchenko.1@nd.edu}

\secondarydata{03C60, 03G10}

\begin{abstract}
We study the Vapnik-Chervonenkis~(VC)~density of definable families in certain stable first-order theories. In particular we obtain uniform bounds on VC~density of definable families in finite $\URk$-rank theories without the finite cover property, and we characterize those abelian groups for which there exist uniform bounds on the VC~density of  definable families.
\end{abstract}

\maketitle

\tableofcontents
\newpage

\section{Introduction}

\noindent
The Vapnik-Chervonenkis (VC) density and its  cousin, the VC~dimension, are numerical parameters associated to any set system (i.e., a family of subsets of a given base set).
VC~classes, that is, set systems whose  VC~dimension is finite, have been investigated since the early seventies (beginning with \cite{sauer,S2,vc}). Since then, the concept of VC~dimension has found numerous applications in 
statistics, 
combinatorics, 
learning theory, and 
computational geometry. 
In model theory, the connection between VC~classes and the theories without the independence property (NIP theories) introduced by Shelah \cite{S1} was first made explicit by Laskowski \cite{l}, who showed that a complete first-order theory is NIP if and only if every formula defines, in each model, a VC~class. 
An important class of NIP theories are the ones which are stable.
The class of NIP theories, however, also encompasses, among other examples, all weakly o-minimal theories and many interesting theories of expansions of valued fields.
In recent years, not least due to the efforts of Anand Pillay (see, e.g., \cite{HPP,HP,HPS}), there has been some progress in extending the reach of the highly developed methods of stability theory into this wider realm of NIP theories.

In this paper we establish uniform bounds, in terms of the number of parameter variables, on the VC~density of formulas in certain stable theories, continuing our study from \cite{ADHMS}, where we mainly focussed on unstable examples. Such uniform bounds may be interesting since the VC~density is often an important parameter for the complexity of a family of sets, and is related to discrepancy  and the size of $\varepsilon$-approximations~\cite{MWW}, and to entropic dimension~\cite{Assouad}. Moreover, uniform bounds on VC~density are intimately connected with a strengthening of the NIP concept, called dp-minimality, which has recently received attention through the work of Shelah \cite{S-strongly dependent}, Onshuus-Usvyatsov \cite{OU}, Dolich-Goodrick-Lippel \cite{dl}, and others.
Before we state some of the main results of this paper, we briefly recall the relevant terminology.
In the rest of this introduction, $\mathcal L$ is a first-order language,  $T$ is a complete  $\mathcal L$-theory without finite models, and $\mathbf M$ is a model of~$T$ (with universe~$M$). Given a tuple $x=(x_1,\dots,x_m)$ of pairwise distinct variables we denote by $\abs{x}:=m$ the length of $x$.

\subsection{Shatter function, VC dimension, and VC density.}
Let $\varphi(x;y)$ be a \emph{partitioned $\mathcal L$-formula}, that is, an $\mathcal L$-formula with the free variables of $\varphi$ contained among the entries of the tuples $x=(x_1,\dots,x_m)$ and $y=(y_1,\dots,y_n)$ of pairwise distinct variables (also assumed to be disjoint). We call the $x_i$ the \emph{object variables} and the~$y_j$ the \emph{parameter variables} of $\varphi(x;y)$. 
We obtain a set system
$$\calS_\ph = \big\{ \ph(M^{m};b) : b\in M^{n}\big\}$$ 
on $M^{m}$, where $\ph(M^{m};b)$ denotes the subset of $M^{m}$ defined in $\mathbf M$ by the formula~$\varphi(x;b)$.
Given $A'\subseteq A\subseteq M^m$, we say that $A'$ is \emph{cut out from $A$ by $\varphi$} if $A'=S\cap A$ for some $S\in\calS_\ph$, and
we say that \emph{$A$ is shattered by $\varphi$} if every subset of $A$ is cut out from $A$ by~$\ph$.

For each non-negative integer $t$ let $\pi_\varphi(t)$ be the maximal number of subsets cut out from a $t$-element subset of $M^m$. (So $\pi_\ph(t)=2^t$ if and only if there is a $t$-element subset of $M^m$ shattered by $\ph$.)
The function $t\mapsto\pi_{\ph}(t)$ is called the \emph{shatter function} of $\ph$; it is routine to verify (see, e.g.,~\cite[Section~3]{ADHMS}) that $\pi_\ph$ is an invariant of the elementary theory $T$ of $\mathbf M$. 
The \emph{VC dimension} of $\ph$ (in $T$) is the largest $d=\VC(\ph)$ (if it exists) such that $\pi_\ph(d)=2^d$; if there is no such~$d$ (i.e., if $\pi_\ph(t)=2^t$ for each $t$), then we also set $\VC(\ph)=\infty$.
(We suppress the dependence on $T$ in the notation for the shatter function of $\varphi$ and concepts, like the VC~dimension of $\ph$, derived from it.)

Now if $\ph$ has finite VC~dimension $d$, then by a fundamental combinatorial fact (proved independently in~\cite{sauer}, \cite{S2} and~\cite{vc}), the shatter function of $\ph$ is bounded above by a polynomial in $t$ of degree $d$; in fact,  $\pi_\ph(t)$ is bounded by ${t\choose \leq d}:={t\choose 0}+{t\choose 1}+\cdots+{t\choose d}$ (the number of subsets of $[t] := \{1,\dots,t\}$ having at most $d$ elements).  One defines the \emph{VC~density $\vc(\ph)$} (in $T$) of a formula $\ph$ with finite VC~dimension in $T$ as the infimum of all real numbers~$r\geq 0$ such that $\pi_{\ph}(t)/t^r$ is bounded for all positive $t$; if $\VC(\ph)=\infty$, then we put~$\vc(\ph)=\infty$.

\subsection{Independence dimension and VC duality.}
The partitioned $\calL$-formula $\ph(x;y)$ is said to have the \emph{independence property}  for~$\mathbf M$  if for every $t\in\bN$ there are $b_1,\ldots,b_t\in M^n$ such that for every $S \subseteq [t]$ there is $a_S\in M^m$ such that for all $i\in [t]$, $\mathbf M\models \ph(a_S;b_i) \Longleftrightarrow i\in S$. The structure $\mathbf M$ is said to have the independence property if some $\calL$-formula has the independence property for $\mathbf M$, and to be \emph{NIP} (or to be \emph{dependent}) otherwise. Clearly the independence property for $\varphi$ only depends on $T=\Th(\mathbf M)$, so it makes sense to speak of a partitioned $\mathcal L$-formula having the independence property for the complete $\mathcal L$-theory $T$, and of~$T$ being NIP. The \emph{dual} of the partitioned $\mathcal L$-formula $\varphi(x;y)$ is the partitioned $\mathcal L$-formula $\varphi^*(y;x)$ which is syntactically the same formula as $\varphi$, except that the roles of object and parameter variables are interchanged. Thus $\varphi$ has the independence property for $T$ if and only if $\varphi^*$ has finite VC~dimension in $T$. (It is also true, by another fundamental observation about set systems, that $\varphi$ has the independence property for $T$ if and only if $\varphi$ has finite VC~dimension in $T$; cf., e.g.,~\cite[Section~3]{ADHMS}.)

Given $B\subseteq M^{n}$, a complete $\varphi(x;B)$-type is a maximal consistent subset of 
$$\big\{\varphi(x;b):b\in B\big\}\cup \big\{\neg\varphi(x;b):b\in B\big\}.$$ 
We denote the set of complete $\varphi(x;B)$-types by $S^\varphi(B)$.
The shatter function $\pi_{\varphi}^* := \pi_{\varphi^*}$ of the dual formula of $\varphi$ counts the number of complete $\varphi(x;B)$-types over finite parameter sets $B$: the map
$$B'\mapsto \big\{ \varphi(x;b) : b\in B' \big\} \cup \big\{ \neg\varphi(x;b) : b\in B\setminus B' \big\}$$ 
defines a bijection from the collection of subsets of $B$ cut out by $\varphi^*$ onto
$S^\varphi(B)$. Hence 
$$\pi^*_{\varphi}(t)=\max\big\{\abs{S^\varphi(B)}:B\subseteq M^n,\ \abs{B}=t\big\}\qquad\text{for each $t$.}$$
This interpretation of the shatter function allows us to transform the problem of estimating $\vc(\varphi^*)$, which is one of bounding the asymptotic growth of $\pi_{\varphi^*}$, into the problem of counting the number of $\varphi(x;B)$-types over finite parameter sets~$B$. For the latter task, ready-made model-theoretic tools (like the local ranks employed in Section~\ref{sec:other} below) are available. As already in \cite{ADHMS}, we also use a variant of the familiar notion of definable type; we combine it with additional combinatorial techniques, centered around the notion of breadth of a set system, and  apply it, in this paper, to modules.
(See Sections~\ref{sec:examples of VCm theories, 2}--\ref{sec:abelian}.)

\subsection{Finite-rank theories.}
Our first main theorem gives a uniform bound on VC~density of formulas for superstable theories of finite rank which do not have the finite cover property.
From \cite[Definition~II.4.1]{Shelah-book} recall that $T$ is said to have the \emph{finite cover property} if there is an $\mathcal L$-formula $\varphi(x;y)$ with the finite cover property in $T$, i.e.,
so that for arbitrarily large $t \in \bN$ there are $b_1, \dots, b_t\in M^{n}$ so that 
$\{ \ph(x;b_i) : i\in [t] \}$
is inconsistent but $\{ \ph(x;b_i) : i\in w \}$ is consistent for any proper subset $w$ of  $[t]$.
If $T$ has the finite cover property, then there is actually a formula $\varphi(x;y)$ with the finite cover property in $T$ where $\abs{x}=1$  \cite[Theorem~II.4.4]{Shelah-book}. Every theory which does not have the finite cover property is stable \cite[Theorem~II.4.2]{Shelah-book}. In Section~\ref{sec:other} we show:

\begin{theorem}\label{thm:finite MR}
Suppose  $T$ does not have the finite cover property and  $T$ has finite $\URk$-rank $d$.  Then every partitioned $\mathcal L$-formula $\varphi(x;y)$ has VC~density at most~$d\abs{y}$ in $T$ \textup{(}in fact, $\pi_{\varphi}(t)=O(t^{d\abs{y}})$\textup{)}.
\end{theorem}

So e.g., for $p$ a prime or $0$, every  formula $\varphi(x;y)$ in the language of rings has VC~density at most~$\abs{y}$ in the theory of algebraically closed fields of characteristic~$p$. Other examples of theories to which Theorem~\ref{thm:finite MR} applies (including all expansions of groups having finite Morley rank) are provided in Section~\ref{sec:applications}. The proof of this theorem, given in Section~\ref{sec:other}, uses the local ranks $R^m(-,\Delta,\aleph_0)$ of Shelah \cite[Chapter~II]{Shelah-book}, here denoted by $R_\Delta(-)$. The absence of the finite cover property enters the picture as a necessary condition for the definability of the ranks~$R_\Delta$; see~Fact~\ref{nfcp} below.

\subsection{Abelian groups.}
We have an improvement over Theorem~\ref{thm:finite MR} for $\aleph_0$-cate\-go\-ri\-cal abelian groups. For an abelian group $A$, always written additively, let
$$U(p,i;A):=\abs{(p^iA)[p]/(p^{i+1}A)[p]}\qquad\text{ ($p$ prime)}$$ 
be its Ulm invariants; here $(p^iA)[p]$ denotes the subgroup
$$(p^iA)[p]=\{a\in A:\text{$pa=0$ and $p^ib=a$ for some $b\in A$}\}$$
of $A$. 
For each prime~$p$ let $U_{\geq\aleph_0}(p;A)$ denote the  set of $i\geq 0$ such that $U(p,i;A)$ is infinite. If $A$ has finite exponent, then each $U_{\geq\aleph_0}(p;A)$ is finite, with $U_{\geq\aleph_0}(p;A)=\emptyset$ for all but finitely many~$p$. With this notation we have:

\begin{theorem}\label{thm:aleph0-cat abelian}
Let $A$ be an infinite abelian group of finite exponent, construed as a first-order structure in the language $\mathcal L=\{0,{+}\}$ as usual. Let
$n_p:=\abs{U_{\geq\aleph_0}(p;A)}$ and $m_p:=\max U_{\geq\aleph_0}(p;A)$ if $U_{\geq\aleph_0}(p;A)\neq\emptyset$ and $m_p:=0$ otherwise.
Set 
$$d:=\sum_p \min\big\{n_p,\lfloor m_p/2\rfloor+1\big\}.$$ 
Then every $\mathcal L$-formula $\varphi(x;y)$ has VC~density at most $d\,\abs{y}$ in $A$.
\end{theorem}

Since, in the context of the previous theorem, the Morley rank of the abelian group~$A$ is $\MRk(A) = \sum_p \left(n_p+\sum_{i\in U_{\geq\aleph_0}(p;A)} i\right) \geq d$, the bound obtained here is more precise than the one in the  general Theorem~\ref{thm:finite MR}. Moreover, the bound in Theorem~\ref{thm:aleph0-cat abelian} can be refined even further; since the resulting bound is somewhat technical to describe, we refer to the proof of this theorem in Section~\ref{sec:examples of VCm theories, 2} below for the improved result. (However, even the crude bound stated in the theorem is optimal in some cases, under the additional condition that for all primes $p$ and all $i$ we have $U(p,i;A)\geq\aleph_0$ whenever~$U(p,i;A)\neq 1$.) 
The proof of Theorem~\ref{thm:aleph0-cat abelian} is based on our definable-type technique from \cite{ADHMS} and a detailed
analysis of the ordered set of join-irreducibles of the lattice of positive-primitive definable subgroups of $A$.

\medskip
\noindent
No particularly good bounds on density can be expected for $1$-based stable theories in general, even in the case of abelian groups (see Lemma~\ref{lem:lower bound, abelian gps}). However,  in \cite{ADHMS} we already showed that in Presburger~Arithmetic, every partitioned formula with~$n$ parameter variables has VC~density at most~$n$.
Theorem~\ref{thm:aleph0-cat abelian} is the cornerstone of a characterization of those  abelian groups admitting a uniform bound on VC~density of formulas in terms of the number of parameter variables:

\begin{theorem}\label{thm:abelian}
Let $A$ be an infinite abelian group. Then the following are equivalent:
\begin{enumerate}
\item there is some $d$ such that each $\mathcal L$-formula $\varphi(x;y)$ with $\abs{y}=1$ has VC~density at most $d$ in $A$;
\item there is some $d$ such that each $\mathcal L$-formula $\varphi(x;y)$ has VC~density at most $d\,\abs{y}$ in $A$;
\item there are only finitely many  $p$ such that $A[p]$ or $A/pA$ is infinite, and for all $p$ the set
$U_{\geq\aleph_0}(p;A)$ is finite.
\end{enumerate}
\end{theorem}

Implicit in the arguments comprising the proof of this theorem is an explicit description of all dp-minimal abelian groups.
For example, by Theorem~\ref{thm:aleph0-cat abelian}, every $\mathcal L$-formula $\varphi(x;y)$ with $\abs{y}=1$ has VC~density at most $1$ in $$A=\bZ(p)^{(\aleph_0)}\oplus\bZ(p^2)^{(\aleph_0)},$$ 
hence this group $A$ is dp-minimal by \cite[Proposition~3.2]{dl} (see also Section~\ref{sec:dp-min} below).
We refer to Proposition~\ref{prop:dp-minimal abelian} for the complete list of all dp-minimal abelian groups.

\subsection{Organization of the paper.}
We begin with a preliminary Section~\ref{sec:prelims}, where we recall some basic definitions and results from Part~I of our paper. The reader familiar with \cite{ADHMS} may choose to skip this section upon first reading. In Section~\ref{sec:other} we prove Theorem~\ref{thm:finite MR} (in a slightly more general form, cf.~Theorem~\ref{mainst}). 
In Section~\ref{sec:examples of VCm theories, 2} we introduce the breadth of a module (defined via its lattice of positive-primitive definable subgroups) and relate it to VC~density, and after these foundations have been laid, we prove Theorems~\ref{thm:aleph0-cat abelian} and \ref{thm:abelian} in the final Section~\ref{sec:abelian}.

\subsection{Notations and conventions.}
Throughout this paper, $d$, $k$, $m$ and $n$ range over the set $\bN:=\{0,1,2,\dots\}$ of natural numbers. We set $[n]:=\{1,\dots,n\}$. Given a set $X$,
we write $2^X$ for the power set of $X$.

\section{Preliminaries}\label{sec:prelims}

\noindent
In this section we recapitulate some basic concepts and results proved in the predecessor \cite{ADHMS} of this paper.
We let $\mathbf M$ be an infinite $\mathcal L$-structure for some first-order language~$\mathcal L$ and $T=\Th(\mathbf M)$. 

\subsection{Dual VC~density of finite sets of formulas.}
Let $\Delta=\Delta(x;y)$ be a finite set of partitioned $\mathcal L$-formulas $\varphi=\varphi(x;y)$ in the tuple of object variables $x$ and tuple of parameter variables $y$.
We write $\neg\Delta:=\{\neg\varphi:\varphi\in\Delta\}$.
Let $B\subseteq M^{\abs{y}}$ be finite.  We set 
$$\Delta(x;B) := \big\{ \varphi(x;b) : \varphi\in\Delta,\ b\in B \big\},$$
and we call a consistent subset of $\Delta(x;B)\cup\neg\Delta(x;B)$ a $\Delta(x;B)$-type. 
The set of realizations in~$\mathbf M$ of a  $\Delta(x;B)$-type $p$ is denoted by $p^{\mathbf M}$. 
Given $a\in M^{\abs{x}}$ we write $\operatorname{tp}^\Delta(a/B)$ for the $\Delta(x;B)$-type realized by $a$, and
$$S^\Delta(B) = \big\{ \operatorname{tp}^\Delta(a/B) : a\in M^{\abs{x}}\big\}.$$
If $\Delta=\{\varphi\}$ is a singleton, we also write $S^\varphi(B)$ instead of $S^\Delta(B)$. 
We set
$$\pi^*_\Delta(t) :=\max\big\{ \abs{S^\Delta(B)}:B\subseteq M^{\abs{y}},\ \abs{B}=t\big\}\qquad\text{for each $t\in\bN$.}$$
If we pass from $\mathbf M$ to an elementarily equivalent $\mathcal L$-structure, then $\pi^*_\Delta$ does not change; this justifies our notation, which suppresses $\mathbf M$. (But $\pi^*_\Delta$ does depend on $T=\Th(\mathbf M)$, which is also suppressed in our notation.)
If no $\varphi\in\Delta$ has the independence property, then there exists a real number~$r$ with $0\leq r\leq \sum_{\varphi\in\Delta} \vc^*(\varphi)$ and
$$\abs{S^\Delta(B)} = O(\abs{B}^r) \qquad\text{for all finite $B\subseteq M^{\abs{y}}$.}$$
(See \cite[Lemma~3.15]{ADHMS}.)
Hence in this case one can define the \emph{dual VC density of~$\Delta$} as 
$$\vc^*(\Delta)=\inf\big\{r\geq 0:\pi^*_\Delta(t)=O(t^r)\big\}.$$ 
If $\Delta=\{\varphi\}$ is a singleton, then
$\vc^*(\Delta)$ agrees with $\vc^*(\varphi)$ as defined in the introduction.

\subsection{The VC~density function of $T$.}
Suppose $T$ is  NIP. We define the \emph{VC density of~$T$} to be the function
$\vc=\vc^T\colon\bN\to\bR^{\geq 0}\cup\{\infty\}$
given by
\begin{align*}
\vc(n)	& :=\sup\big\{ \vc(\varphi): \text{$\varphi(x;y)$ is an
$\mathcal L$-formula with $\abs{y}=n$} \big\} \\
		& \hskip0.25em =\sup\big\{ \vc^*(\varphi): \text{$\varphi(x;y)$ is an
$\mathcal L$-formula with $\abs{x}=n$} \big\}.
\end{align*}
It is easy to see that $\vc(m)\geq m$ for each $m$; in \cite[Section~3]{ADHMS} we showed that in fact $\vc(m+1)\geq \vc(m)+1$ for each $m$. Moreover, we proved (cf.~\cite[Lemma~3.6]{ADHMS}):

\begin{lemma}\label{lem:expansions by constants}
Let $\mathcal L'=\mathcal L\cup\{c_i:i\in I\}$ where the $c_i$ are new constant symbols, and let $T'\supseteq T$ be a complete $\mathcal L'$-theory. Then $\vc^T=\vc^{T'}$.
\end{lemma}

The VC~density function of $T$ also bounds the dual VC~density of finite sets of formulas; in fact, we have (cf.~\cite[Corollary~3.19]{ADHMS}):

\begin{lemma}\label{lem:encoding finite sets of formulas}
Let $\Phi$ be a set of $\mathcal L$-formulas with the tuple of object variables $x$ and varying parameter variables such that every $\mathcal L$-formula $\varphi(x;y)$ is equivalent in $T$ to a Boolean combination of formulas in $\Phi$. Then 
$$\vc^T(m) = \sup\big\{ \vc^*(\Delta):  \text{$\Delta\subseteq\Phi$ finite }\big\}\qquad\text{where $m=\abs{x}$.}$$
\end{lemma}

\subsection{Breadth.}\label{sec:breadth}
Let $\mathcal B$ be a set system on a set $X$. The smallest integer $d$ (if it exists)
with the property that for every non-empty intersection  $\bigcap_{i\in [n]} B_i$ of $n>d$ sets from $\mathcal B$ 
there is some $I\subseteq [n]$ with $\abs{I}=d$ and $\bigcap_{i\in [n]} B_i=\bigcap_{i\in I} B_i$, is called the \emph{breadth} of $\mathcal B$; if there is no such $d$ we also say that $\mathcal B$ has breadth~$\infty$. See Section~2.4 of \cite{ADHMS} for more on this concept; the following fact was also already noted there:

\begin{lemma}\label{lem:breadth of cosets}
Let $G$ be a group and let $\mathcal H$ be a collection of subgroups of $G$ with breadth~$d$.  Let $\mathcal B=\{gH:g\in G,H\in\mathcal H\}$ be the set of all \textup{(}left\textup{)} cosets of subgroups from $\mathcal H$. Then $\mathcal B$ also has breadth $d$. 
\end{lemma}
\begin{proof}
Let $H_1,\dots,H_n\in\mathcal H$ and $g_1,\dots,g_n\in G$ such that $\bigcap_{i\in [n]} g_iH_i\neq\emptyset$. 
This intersection is a coset of $H:=\bigcap_{i\in [n]} H_i$. Choose $I\subseteq [n]$ of size $d$ such that $H=\bigcap_{i\in I} H_i$. Then both $\bigcap_{i\in [n]} g_iH_i$ and $\bigcap_{i\in I} g_iH_i$ are cosets of $H$, with the former contained in the latter, and hence they are equal.
\end{proof}

\subsection{The $\VC{}d$ property.}
In the following we let $\Delta=\Delta(x;y)$ be a finite set of partitioned $\mathcal L$-formulas in the object variables $x$ and parameter variables $y$.
The $\VC{}d$~property is a convenient condition on a theory which allows the 
counting of the number of $\Delta(x;B)$-types over finite parameter sets $B$. The definition of the $\VC{}d$~property rests on a ``uniform'' variant of the notion of definable type, originating in~\cite{Guingona}:

\begin{definition}
We say that $\Delta$ has \emph{uniform definability of types over finite sets}\/ (abbreviated as \emph{UDTFS}\/) in $\mathbf M$ with $d$ parameters if there are finitely many families 
$$\mathcal D_i=\big(\operatorname{d}_{\varphi,i}(y;y_1,\ldots,y_d)\big)_{\ph\in\Delta}\qquad (i\in I)$$ 
of $\mathcal L$-formulas (with $\abs{y_j}=\abs{y}$ for $j=1,\dots,d$) such that for every finite  set $B\subseteq M^{\abs{y}}$
and $q\in S^\Delta(B)$ there are $b_1,\ldots,b_d \in B$ and some $i\in I$ such that $\mathcal D_i(y;b_1,\dots,b_d)$
defines $q$.  We call the family $\mathcal D=(\mathcal D_i)_{i\in I}$ a \emph{uniform definition of $\Delta(x;B)$-types over finite sets}\/ in $\mathbf M$ with $d$~parameters.
\end{definition}

An easy way to check for UDTFS is via breadth (cf.~\cite[Lemma~5.2]{ADHMS}):

\begin{lemma}\label{lem:breadth and UDTFS}
Suppose the set system 
$$\mathcal S_\Delta=\big\{\varphi(M^{\abs{x}};b):b\in M^{\abs{y}}\big\}$$ 
has breadth~$d$. Then $\Delta$ has UDTFS with $d$ parameters.
\end{lemma}

We say that $\mathbf M$ has the \emph{$\VC{}d$~property} if  any $\Delta$ with $\abs{x}=1$ has uniform definability of types over finite sets with $d$~parameters. 
Clearly if $\mathbf M$ has the $\VC{}d$~property, then so does every other model of $T$, and we say that $T$ has the $\VC{}d$~property if some model of $T$ does.
The point of the $\VC{}d$~property is that it guarantees that every finite subset of partitioned $\mathcal L$-formulas has UDTFS in $\mathbf M$ {\it with a uniform bound on the number of parameters}\/:

\begin{theorem}[{see \cite[Theorem~5.7]{ADHMS}}]\label{VCdensity, 2}
Suppose that $\mathbf M$ has the $\VC{}d$ property. Then every $\Delta$ has UDTFS in~$\mathbf M$ with  $d\,\abs{x}$~parameters. 
\end{theorem}

If $\Delta$ has UDTFS in $\mathbf M$ with $d$~parameters, then clearly
$\abs{S^\Delta(B)}=O(\abs{B}^d)$  for every finite $B\subseteq M^{\abs{y}}$. Hence if  $\mathbf M$ has the $\VC{}d$ property, then for each $\Delta$ we have 
$\pi^*_\Delta(t)=O(t^{d\abs{x}})$ for each $t$ and thus $\vc^*(\Delta)\leq d\,\abs{x}$.

\medskip

The following fact (cf.~\cite[Corollary~5.6]{ADHMS}), when combined with a quantifier-simplifica\-tion result, is handy for verifying the $\VC{}d$ property; it will later be applied in the case of modules.

\begin{lemma}\label{lem:qe and VCd}
Let $\Phi$ be a family of partitioned $\mathcal L$-formulas in the single object variable~$x$ such that 
\begin{enumerate}
\item every partitioned $\mathcal L$-formula in the object variable $x$ is equivalent in $T$ to a Boolean combination of formulas from $\Phi$, and
\item every finite set of  $\mathcal L$-formulas from $\Phi$ has UDTFS in $T$ with $d$ parameters.
\end{enumerate}
Then $T$ has the $\VC{}d$~property.
\end{lemma}

\subsection{Dp-minimality.}\label{sec:dp-min}
By a \emph{bigraph} we mean a triple $G=(X,Y,\Phi)$ where $X$, $Y$ are sets and $\Phi\subseteq X\times Y$. We call the elements of the disjoint union $X\cup Y$ the \emph{vertices} of the bigraph $G$ and the elements of
$\Phi$ the \emph{edges} of $G$. 
A bigraph $G'=(X',Y',\Phi')$ is a \emph{sub-bigraph} of the bigraph 
$G=(X,Y,\Phi)$ if $X'\subseteq X$, $Y'\subseteq Y$, and $\Phi'\subseteq \Phi$. We say that a bigraph $G$ contains a given bigraph $G'$ (as a sub-bigraph) if $G'$ is isomorphic to a sub-bigraph of $G$.

Let $\varphi(x;y)$ be a  partitioned $\mathcal L$-formula. Given subsets $A$, $A'$ of $M^{\abs{x}}$ we define a bi\-graph $G_{A,A',\varphi}=(A,A',\Phi)$ with edge set 
$$\Phi=\big\{(a,a')\in A\times A':\{a,a'\}\in\mathcal S_\varphi\cap (A\cup A')\bigr\}.$$ 
Note that a pair $(a,a')\in A\times A'$ is connected by an edge of $G_{A,A',\varphi}$ iff 
$$\big\{\varphi^*(y;a),\varphi^*(y;a')\big\}\cup\big\{ \neg\varphi^*(y;a_0) : a_0\in A\cup A',\ a_0\neq a,a'\big\}$$
is realized in $\mathbf M$.

It is easy to see that a bigraph $G=(X,Y,\Phi)$ with a finite number $n$ of vertices can have at most $\frac{1}{4}n^2$ edges. A fundamental fact about bigraphs is the theorem of K\H{o}v\'ari, S\'os and Tur\'an \cite{KST}: 
{\it given a positive integer $r$, there exists a real number $C=C(r)$ such that every  bigraph $G$ with $n$ vertices which does not contain $K_{r,r}$ as a sub-bigraph 
has at most $C\,n^{2-1/r}$ edges.}
Here $K_{r,r}$ denotes the (complete) bigraph $K_{r,r}=([r],[r],[r]\times[r])$.
In particular, from this theorem we obtain:

\begin{lemma}\label{lem:dp-min}
Let $\varphi(x;y)$ be a partitioned $\mathcal L$-formula. The following are equivalent:
\begin{enumerate}
\item for each $r>0$ there are finite subsets $A$, $A'$ of $M^{\abs{x}}$ such that $G_{A,A',\varphi}$ contains the bigraph $K_{r,r}$;
\item for each $r>0$ there are finite subsets $A$, $A'$ of $M^{\abs{x}}$ such that $G_{A,A',\varphi}$ has more than $Cn^{2-1/r}$ edges, where $C=C(r)$ is as above and $n=\abs{A}+\abs{A'}$.
\end{enumerate}
\end{lemma}

The $\mathcal L$-structure $\mathbf M$ is said to be \emph{dp-minimal} if there is no $\mathcal L$-formula  $\varphi(x; y)$ with $\abs{y} = 1$ satisfying one of the two equivalent conditions in the previous lemma. 
In~\cite{ADHMS} we called $\mathbf M$ \emph{vc-minimal} if $\vc^*(\varphi)<2$ for every $\mathcal L$-formula $\varphi(x;y)$ with $\abs{x}=1$. (So if $\vc(1)<2$, then $\mathbf M$ is vc-minimal.) It is easy to see (cf.~\cite[Corollary~5.13]{ADHMS}) that if $\mathbf M$ is vc-minimal, then $\mathbf M$ is dp-minimal. (We do not know whether, conversely, every dp-minimal structure is vc-minimal.) On the other hand, if $\mathbf M$ is dp-minimal, then for each $\mathcal L$-formula $\varphi(x;y)$ with $\abs{y}=1$ there is an upper bound on the size of finite subsets $A\subseteq M^{\abs{x}}$ such that $A\cap\mathcal S_\varphi$ contains all two-element subsets of $A$. Thus, although reducts of dp-minimal structures are obviously also dp-minimal, a structure interpretable in a dp-minimal structure is not in general itself dp-minimal:

\begin{example*}
Suppose $\mathcal L$ is the empty language, and let $\mathcal L'=\{\pi_1,\pi_2\}$ where $\pi_i$ is a unary function symbol, for $i=1,2$. Let $\mathbf M'$ be the $\mathcal L'$-structure with underlying set $M'=M\times M$, where $\pi_i^{\mathbf M'}$ is given by $(x_1,x_2)\mapsto (x_i,x_i)$, for $i=1,2$. Then clearly~$\mathbf M'$ is interpretable without parameters in $M$. Consider the partitioned $\mathcal L'$-formula $\varphi(x;y) := x=\pi_1(y)\vee x=\pi_2(y)$, where $\abs{x}=\abs{y}=1$. Then $\mathcal S_{\varphi}^{\mathbf M'}={D\choose 1}\cup {D\choose 2}$, where $D=\{(a,a):a\in M\}$,  hence $\mathbf M'$ is not dp-minimal.
\end{example*}

One defines  a complete $\mathcal L$-theory (without finite models) to be dp-minimal if one of its models is dp-minimal (equivalently, if all of its models are). The definition of dp-minimality given here is not the original one as introduced in \cite{OU}, which used the notion of ICT pattern from \cite{S-strongly dependent}:
an \emph{ICT pattern}\/ in $\mathbf M$ consists of a pair $\alpha(x;y)$, $\beta(x;y)$ of partitioned $\mathcal L$-formulas, where $\abs{x}=1$,  and sequences $(a_i)_{i\in\bN}$, $(b_j)_{j\in\bN}$ in $M^{\abs{y}}$ such that for all $i$ and $j$ the set 
$$\big\{\alpha(x;a_i),\beta(x;b_j)\big\}\cup\big\{ \neg\alpha(x;a_k) : k\neq i\big\} \cup \big\{ \neg\beta(x;b_l):l\neq j\big\}$$
is consistent with $\mathbf M$. One can show (see, e.g., \cite[Fact~2.13]{dl} or \cite[Proposition~5.12]{ADHMS}) that
$\mathbf M$ is dp-minimal iff  there is no ICT pattern in an elementary extension of $\mathbf M$.

\medskip

The extent of the failure of the dp-minimality property for a theory can be quantified via its dp-rank.

\begin{definition}  
One says that \emph{$T$ has  dp-rank at least $n$} if there are partitioned $\mathcal L$-formulas $\phi_1(x; y),\dots,\phi_n(x;y)$,  where $\abs{x}=1$, and parameters $a^i_j\in M^{\abs{y}}$, where $i\in [n]$ and $j \in \bN$, so that for any function $\eta\colon [n] \to \bN$
the set of $\mathcal L(M)$-formulas
\[
\big\{ \phi_i(x; a^i_{\eta(i)}) : i\in [n] \big\} \cup \big\{ \neg\phi_i(x; a^i_j) : i\in [n],\ j\neq\eta(i) \big\}\]
is consistent (with $\mathbf M$).
\end{definition}

So $T$ is not dp-minimal iff $T$ has dp-rank at least $2$.
An easy generalization of \cite[Corollary~5.13]{ADHMS} shows that if $T$ has dp-rank at least $n$, then $\vc(1)\geq n$.
The following lemma gives a convenient way to check for dp-rank at least $n$:

\begin{lemma}  \label{lem:dp-rank}
Suppose there exists an infinite definable set $X \subseteq M$ and a definable one-to-one function $X^n \to M$. Then the dp-rank of $T=\Th(\mathbf M)$ is at least $n$.
\end{lemma}

\begin{proof}  
Let $f\colon X^n\to M$ be definable and injective, where $X\subseteq M$ is definable.
Without loss of generality we may assume that $X$ is $\emptyset$-definable by $\psi(x)$ and $f$ is
also $\emptyset$-definable.  Pick pairwise distinct $a_i \in X$ for $i \in \bN$.
Let $\phi_i(x; y)$ be the formula
\[\exists z_1 \dots z_n\left(\bigwedge_{j\in [n]}\psi(z_j) \wedge  f(z_1,\dots,z_n)=x \wedge z_i=y \right).\]
Now as in the preceding definition setting $a^i_j:=a_j$ for $i\in [n]$ and $j\in\bN$, the
formulas $\phi_i$ and parameters $a^i_j$ witness that $T$ has dp-rank at least $n$.
\end{proof}

\section{Linear VC Density for Finite Rank Theories}\label{sec:other}

\noindent
Throughout this section we assume that $T$ is a complete stable theory without finite models, in a first-order language $\mathcal L$. We fix a monster model $\mathbf M$ of~$T$. 
``Small'' means ``of cardinality smaller than $\abs{M}$,'' and ``$\models\dots$'' abbreviates ``$\mathbf M\models\dots$.'' 
We also let $\Delta$ be a set (possibly infinite) of partitioned $\mathcal L$-formulas in the tuple of object variables $x=(x_1,\dots,x_m)$ (and varying tuples of parameter variables). 

Given $B\subseteq M$, by a $\Delta$-formula over $B$ we mean a formula of the form $\varphi(x;b)$ or $\neg\varphi(x;b)$ where $\varphi(x;y)\in\Delta$ and $b\in B^{\abs{y}}$, and
by a $\Delta$-type over $B$ (in $\mathbf M$) we mean a consistent set of $\Delta$-formulas over $B$. 
A $\Delta$-type $p=p(x)$ over $B$ is said to be complete if for all $\Delta$-formulas $\varphi$, either $\varphi$ or $\neg\varphi$ is in $p$. We denote by $S_\Delta(B)$ the set of complete $\Delta$-types over $B$ in $\mathbf M$, equipped with the topology with a subbasis consisting of the sets of the form $\{p\in S_\Delta(B):\varphi\in p\}$ where $\varphi$ is a $\Delta$-formula over $B$. This makes $S_\Delta(B)$ into a compact totally disconnected Hausdorff space.  As usual $S_x(B)$ (or $S_m(B)$, if the particular choice of an $m$-tuple $x$ of object variables is unimportant) denotes $S_\Delta(B)$ where $\Delta$ is the set of \emph{all} partitioned $\mathcal L$-formulas in the object variables $x$. (Note the difference between the notions $S_\Delta(B)$ and $S^\Delta(B)$ as introduced in Section~\ref{sec:prelims}  for finite sets of $\mathcal L$-formulas $\Delta=\Delta(x;y)$ in the parameter variables $y$: in the former the parameter set $B$ is understood to be a subset of $M$, in the latter it is a subset of $M^{\abs{y}}$.)

We will write $R_{\Delta}(-)$ for the rank denoted by $R^m(-,\Delta, \aleph_0)$
in \cite[Definition~II.1.1]{Shelah-book}. 
Given a set $\Theta=\Theta(x)$ of $\mathcal L(M)$-formulas, denote by $[\Theta]$ the closed subset of $S_\Delta(M)$ consisting of all $p(x)\in S_\Delta(M)$ which are consistent with $\Theta(x)$. Then
$R_\Delta(\Theta)$ is the Cantor-Bendixson rank of $[\Theta]$ (i.e., the maximum Cantor-Bendixson rank of an element of $[\Theta]$), with the convention that  $R_\Delta(\Theta)=-1$ if $\Theta$ is inconsistent, and if $[\Theta]$ does not have a Cantor-Bendixson rank, then $R_\Delta(\Theta):=\infty$. We call $R_\Delta(\Theta)$ the $\Delta$-rank of $\Theta$.
If $\Theta=\{\theta\}$ is a singleton we also write $R_\Delta(\theta)$ instead of $R_\Delta(\Theta)$. We also let  
$R_\Delta(T):=R_\Delta(x=x)$ where $x=x$ is shorthand for the $\mathcal L$-formula $x_1=x_1\wedge\cdots\wedge x_m=x_m$.
Since we assume that $T$ is stable, for {\it finite $\Delta$,}\/ the rank $R_\Delta(-)$ always takes values smaller than~$\omega$  \cite[Lemma~II.2.1 and Theorem~II.2.2]{Shelah-book}. 

\medskip
\noindent
Our goal in this section is to prove:

\begin{theorem}\label{mainst} 
Suppose $T$ does not have the finite cover property and $\Delta$ is finite, and let $\varrho=R_{\Delta}(T)$.
Then $\pi^*_\Delta(t)=O(t^\varrho)$ and hence  $\vc^*(\Delta)\leq \varrho$.
\end{theorem}

See Corollary~\ref{cor:mainst, finite U-rank} below on how this theorem implies Theorem~\ref{thm:finite MR} from the introduction.
Before we give the proof of Theorem~\ref{mainst} we recall a few facts about $R_\Delta$  used in the proof.
We will occasionally work in ${\mathbf M}^{\eq}$.  Nonetheless, all sets or elements should
be assumed to lie in the home sort unless explicitly stated otherwise.

\subsection{Basic properties of $\Delta$-rank.}
Let, as above, $\Delta$ be a set of  partitioned $\mathcal L$-formulas in the object variables $x=(x_1,\dots,x_m)$, and let $\Theta(x)$ be a set of $\mathcal L(M)$-formulas.

If $R_\Delta(\Theta)<\infty$, then there is a largest integer $d$ such that there are pairwise distinct $p_1,\dots,p_d\in [\Theta]$ of maximal Cantor-Bendixson rank $R_\Delta(\Theta)$; we call $d$ the $\Delta$-degree of $\Theta$, denoted by $\deg_\Delta(\Theta)$. (In \cite[Definition~II.1.1]{Shelah-book} this is denoted by $\operatorname{Mlt}(\Theta,\Delta,\aleph_0)$.) 
For an $\mathcal L(M)$-formula $\theta$ we write $\deg_\Delta(\theta):=\deg_\Delta(\{\theta\})$.
For every type $\Theta(x)$, there is a $p\in S_x(M)$ with $p\supseteq\Theta$ and $R_\Delta(p)=R_\Delta(\Theta)$; there are exactly $\deg_\Delta(\Theta)$ many types in $S_\Delta(M)$ of the form $p\restrict \Delta$ where $p\in S_x(M)$ with $p\supseteq\Theta$ and $R_\Delta(p)=R_\Delta(\Theta)$.
If $\Theta$ is closed under finite conjunctions, then there is a formula $\theta\in\Theta$ such that $R_\Delta(\Theta)=R_\Delta(\theta)$ and $\deg_\Delta(\Theta)=\deg_\Delta(\theta)$.

If $a\in M^{m}$ and $B\subseteq M^n$, then we set $R_\Delta(a/B):=R_\Delta(\tp(a/B'))$ where $B'$ is the smallest subset of $M$ such that $B\subseteq (B')^n$, and similarly for $\deg_\Delta$. (Note that $\abs{B'}\leq n\,\abs{B}$.)
If $\Theta(x)$ is a type over a small set $B\subseteq M$, then there exists a realization $a\in M^{\abs{x}}$ of $\Theta$ with $R_\Delta(a/B)=R_\Delta(\Theta)$; we say that such a realization of $\Theta$ is \emph{$\Delta$-generic.}

If $\Delta'$ is another set of  partitioned $\mathcal L$-formulas in the object variables $x$ with $\Delta\subseteq\Delta'$, then 
$R_\Delta(\Theta)\leq R_{\Delta'}(\Theta)$, and  if equality holds, then $\deg_\Delta(\Theta)\leq\deg_{\Delta'}(\Theta)$ \cite[Lemma~II.1.3]{Shelah-book}. If $\Delta$ is the set of {\it all}\/ partitioned $\mathcal L$-formulas in the tuple of object variables $x$, then $R_\Delta$ is Morley rank; hence $R_\Delta(\Theta)\leq\MRk(\Theta)$.
Moreover, if $\Theta'$ is another set of $\mathcal L(M)$-formulas with $\Theta^{\mathbf M}\subseteq (\Theta')^{\mathbf M}$, then  $R_\Delta(\Theta)\leq R_\Delta(\Theta')$, and if equality holds, then $\deg_\Delta(\Theta)\leq\deg_\Delta(\Theta')$  \cite[Theorem~II.1.1]{Shelah-book}. It is also easy to see that for all $\mathcal L(M)$-formulas $\theta_1(x)$ and $\theta_2(x)$ we have
\begin{equation}\label{eq:Delta-rank union}
R_\Delta(\theta_1\vee\theta_2)=\max\{R_\Delta(\theta_1),R_\Delta(\theta_2)\}.
\end{equation}
Suppose now $A\subseteq B\subseteq M$, $p\in S_x(B)$. Then $R_\Delta(p\restrict A)\geq R_\Delta(p)$ for all $\Delta$, and $p$ does not fork over $A$ iff $R_\Delta(p\restrict A)=R_\Delta(p)$ for all finite $\Delta$ \cite[Theorem~III.4.1]{Shelah-book}. For example, if $B=\acl(A)$, then $p$ does not fork over $A$ (since any two extensions of a type in $S_x(A)$ to types in $S_x(\acl(A))$ are conjugate under $\operatorname{Aut}(\mathbf M|A)$ and hence have the same $\Delta$-rank). 
A complete type is stationary (i.e., has a unique non-forking extension to any larger parameter set) iff it has $\Delta$-degree $1$ for all finite~$\Delta$ \cite[Theorem~III.4.2]{Shelah-book}, and as a consequence of the Finite Equivalence Relation Theorem \cite[Theorem~III.2.8]{Shelah-book}, every complete type over a parameter set which is algebraically closed in $\mathbf M^\eq$ is stationary; hence we obtain:

\begin{lemma}\label{lem:degree 1}
Suppose $\Delta$ is finite, and let $\theta(x)$ be an $\mathcal L(A)$-formula, $A\subseteq M$, and set $\alpha=R_\Delta(\theta)$, $d=\deg_\Delta(\theta)$. There exist $\mathcal L^\eq(\acl^\eq(A))$-formulas $\theta_1(x),\dots,\theta_d(x)$ such that $R_\Delta(\theta_i)=\alpha$, $\deg_\Delta(\theta_i)=1$ for every $i\in [d]$, and
$\theta(M^{m})$ is the disjoint union of the sets $\theta_i(M^{m})$, $i\in [d]$.
\end{lemma}

We finish the preliminaries by recording three important consequences of a theory not having the finite cover property.  

\begin{fact}\cite[Theorem~II.4.4]{Shelah-book}\label{nfcp} 
Suppose $T$ does not have the finite cover property and $\Delta$ is finite. Then:
\begin{enumerate}
\item There is $k \in \bN$ so that if $p$ is any $\Delta$-type there is $q \subseteq p$ with $\abs{q} \leq k$ so that
$R_{\Delta}(p)=R_{\Delta}(q)$. (Finite witness of rank.)
\item For any $\mathcal L$-formula $\theta(x;y)$ and any $\alpha \in \bN$ there is an $\mathcal L$-formula $\psi(y)$ so that for any $b\in M^{\abs{y}}$, $R_{\Delta}(\theta(x;b))=\alpha$ iff $\models\psi(b)$. (Definability of rank.)
\item For any $\mathcal L$-formula $\theta(x;y)$ there is some $D \in \bN$ so that for any $b\in M^{\abs{y}}$, $\deg_\Delta(\theta(x;b))\leq D$. (Boundedness of multiplicity.)
\end{enumerate}
\end{fact}

\subsection{Proof of Theorem~\ref{mainst}.}
The proof will be through a series of lemmas.  We fix a finite set $\Delta=\Delta(x;y)$ of partitioned $\mathcal L$-formulas, and write $\varrho=R_{\Delta}(T)$.

\begin{lemma}\label{one}  
Let $\psi(x)= \bigwedge_{i=1}^k \phi_i(x;b_i)$ be a conjunction of instances of formulas $\phi_i(x;y_i)\in\Delta$, $\alpha:=R_\Delta(\psi)$, $\mu:=\min_i R_\Delta (\phi_i(x;b_i))$, and $a\in M^{m}$ with $\models\psi(a)$ and $R_{\Delta}(a/b_1, \dots, b_k)=\alpha$.
There is an $\mathcal L^\eq$-formula $\theta(x;z)$ and $c \in \acl^{\eq}(b_{i_1}, \dots, b_{i_{r}})$, where $i_1, \dots, i_{r} \in [k]$ with $r\in[\mu-\alpha]$, such that $\models \theta(a;c)$ and  $R_{\Delta}(\theta(x;c)) = \alpha$.
\end{lemma}
\begin{proof}
Let $\gamma_1(x;c_1)$, where $c_1\in\acl^\eq(\emptyset)$, have $\Delta$-rank~$\mu_1:=\varrho$ and $\Delta$-degree~$1$ so that $\models \gamma_1(a;c_1)$, by Lemma~\ref{lem:degree 1}. 
We have $\mu_1\geq\alpha$; if $\mu_1=\alpha$ we are done, so suppose $\mu_1>\alpha$. 
Then $R_\Delta(\gamma_1(x;c_1)\wedge \phi_i(x;b_i))<\mu_1$ for some $i$: otherwise $R_\Delta(\gamma_1(x;c_1)\wedge\neg\phi_i(x;b_i))<\mu_1$ for every $i$, since $\deg_\Delta(\gamma_1(x;c_1))=1$, hence $R_\Delta\left(\gamma_1(x;c_1)\wedge\neg\left(\bigwedge_{i} \phi_i(x;b_i)\right)\right)<\mu_1$ and therefore $R_\Delta(\gamma_1(x;c_1))<\mu_1$, by \eqref{eq:Delta-rank union}, a contradiction.  Without loss of generality we may assume that  $\mu_2:=R_\Delta(\gamma_1(x;c_1)\wedge \phi_1(x;b_1))<\mu_1$.
We have $\mu_2\geq \alpha$, and if $\mu_2=\alpha$ we are done. Otherwise let $\gamma_2(x;c_2)$, where $c_2\in\acl^{\eq}(b_1)$, have $\Delta$-rank $\mu_2$ and $\Delta$-degree $1$ so that 
$\models \gamma_2(a;c_2)$ and $\models\forall x (\gamma_2(x;c_2) \rightarrow \gamma_1(x;c_1)\wedge \phi_1(x;b_1))$, according to Lemma~\ref{lem:degree 1}.  As before we see that for some $i>1$ we have  $R_\Delta(\gamma_2(x;c_2)\wedge \phi_i(x;b_i))<\mu_2$.

We can continue in this fashion to find a sequence $\gamma_1(x;z),\dots,\gamma_{r}(x;z)$ of $\mathcal L^\eq$-formulas, a sequence of indices $i_1,\dots,i_{r}\in [k]$, and a sequence $c_1,\dots,c_{r}$ where $c_j\in\acl^\eq(b_{i_1},\dots,b_{i_{j-1}})$, of maximal length $r\geq 1$, subject to the following properties:  for every $j\in [r]$ we have
$\models\gamma_j(a;c_j)$, and setting $\delta_0(x):=(x=x)$ and
$\delta_j(x;y_{i_{j}},z):=\gamma_{j}(x;z)\wedge \phi_{i_{j}}(x;y_{i_{j}})$ for $j\in [r]$,
we have:
$$\models \forall x (\gamma_{j}(x;c_{j}) \rightarrow \delta_{j-1}(x;b_{i_{j-1}},c_{j-1}))$$ 
and 
$$\alpha\leq R_\Delta(\delta_{j}(x;b_{i_{j}},c_{j})) < R_\Delta(\delta_{j-1}(x;b_{i_{j-1}},c_{j-1})).$$
We have $r\leq\mu-\alpha$, and by the above $R_\Delta(\delta_r(x;b_{i_{r}},c_{r}))=\alpha$. Hence 
$$\theta(x;y_{i_r},z):=\delta_r(x;y_{i_r},z),\qquad c:=(b_{i_r},c_{r})$$ 
have the required properties.
\end{proof}

\begin{lemma}\label{two}  
Suppose $T$ does not have the finite cover property.
Let 
$$\psi(x; y_1,\dots,y_k)=\bigwedge_{i=1}^k \phi_i(x;y_i)\quad\text{ where $\phi_i(x;y_i)\in\Delta$ for $i\in [k]$,}$$ 
and let $\alpha\in [\varrho]$.  There is a family  
$$\big\{\theta_{\mathbf ij}(x;y_{i_1},\dots,y_{i_{\varrho-\alpha}}): \mathbf i=(i_1,\dots,i_{\varrho-\alpha})\in [k]^{\varrho-\alpha},\ j\in [N]\big\}$$ of $\mathcal L$-formulas, where $N\in\bN$, with the following properties: if $a\in M^{m}$ and $b_i\in M^{\abs{y_i}}$ \textup{(}$i\in [k]$\textup{)} are such that 
$$\models \psi(a; b_1,\dots,b_k)\quad\text{ and }\quad\alpha=R_{\Delta}(a/b_1,\dots,b_k)=R_{\Delta}(\psi(x; b_1,\dots,b_k)),$$  
then for some $\mathbf i\in [k]^{\varrho-\alpha}$ and $j\in [N]$ we have  
$$\models  \theta_{\mathbf ij}(a; b_{i_1},\dots, b_{i_{\varrho-\alpha}})\quad\text{and}\quad
R_{\Delta}(\theta_{\mathbf ij}(x; b_{i_1}, \dots, b_{i_{\varrho-\alpha}}))=\alpha.$$
\end{lemma}
\begin{proof}
For each $a\in M^{m}$ and  $b_i\in M^{\abs{y_i}}$ ($i\in [k]$) satisfying the hypothesis, by Lemma~\ref{one} there are $i_1,\dots,i_{\varrho-\alpha}\in [k]$ such that
$$R_\Delta(a/\acl^\eq(b_{i_1},\dots,b_{i_{\varrho-\alpha}}))=\alpha$$
and hence that 
$$R_\Delta(a/b_{i_1},\dots,b_{i_{\varrho-\alpha}})=\alpha,$$
so for some $\mathcal L$-formula $\theta(x;y_{i_1},\dots,y_{i_{\varrho-\alpha}})$ such that $\models \theta(a;b_{i_1},\dots,b_{i_{\varrho-\alpha}})$ we have
$$R_\Delta(\theta(x;b_{i_1},\dots,b_{i_{\varrho-\alpha}}))=\alpha.$$
Thus the set of formulas $\Gamma(z;y_1,\dots,y_k)$, where $\abs{z}=m$, consisting of
\begin{itemize}
\item $\psi(z;y_1,\dots,y_k)$;
\item $R_\Delta(\psi(x;y_1,\dots,y_k))=\alpha$;
\item $R_\Delta(\theta(x;y_1,\dots,y_k))<\alpha \to \neg\theta(z;y_1,\dots,y_k)$, where $\theta$ ranges over all partitioned $\mathcal L$-formulas in $(x;y_1,\dots,y_k)$; and
\item $\neg(\theta(z;y_{i_1},\dots,y_{i_{\varrho-\alpha}})\wedge R_\Delta(\theta(x;y_{i_1},\dots,y_{i_{\varrho-\alpha}}))=\alpha)$, where  $i_1,\dots,i_{\varrho-\alpha}$ range over $[k]$ and $\theta$ over all partitioned $\mathcal L$-formulas in $(z;y_{i_1},\dots,y_{i_{\varrho-\alpha}})$,
\end{itemize}
is inconsistent. Note that this is a first-order type by Fact~\ref{nfcp},~(2). The lemma now follows by compactness.
\end{proof}

\begin{lemma}\label{three} 
Suppose $T$ does not have the finite cover property.
There is a family $\{\theta_j(x; y_1, \dots, y_{\varrho})\}_{j\in [N]}$ of $\mathcal L$-formulas, where $N\in\bN$ and $\abs{y_i}=\abs{y}$ for $i\in [\varrho]$, with the following property:
if $B \subseteq M^{\abs{y}}$, $p \in S^{\Delta}(B)$, and $a\in M^{m}$ is a $\Delta$-generic realization of $p$,
 then there are $b_1,\dots,b_\varrho\in B$ and $j\in [N]$ such that $\models\theta_j(a; b_1,\dots,b_\varrho)$ and \[R_{\Delta}(\theta_j(x; b_1,\dots,b_\varrho))
 = R_{\Delta}(p).\]
\end{lemma}
\begin{proof}
We may and shall assume that $\Delta$ is closed under negation, i.e., with every $\varphi\in\Delta$ the set $\Delta$ also contains a formula equivalent to $\neg\varphi$ (in $T$).
By Fact~\ref{nfcp},~(1) take $k\in\bN$ so that if $p$ is any $\Delta$-type, then
$R_{\Delta}(p)=R_{\Delta}(q)$ for some $q \subseteq p$ with $\abs{q} \leq k$. 

Apply Lemma~\ref{two} to all possible conjunctions of $k$ formulas from $\Delta$  (and then combine the resulting families $\{\theta_{\mathbf i j}\}$ to a single family)
to obtain a family $\{\theta_j(x; y_1, \dots, y_{\varrho})\}_{j\in [N]}$ of $\mathcal L$-formulas, where $N\in\bN$ and $\abs{y_i}=\abs{y}$ for $i\in [\varrho]$, with the following property: if $\psi(x)=\bigwedge_{i=1}^k \phi_i(x;b_i)$ is a conjunction of instances of formulas from $\Delta$ and $a\in M^{m}$ with $\models\psi(a)$ and $R_\Delta(a/b_1,\dots,b_k)=R_\Delta(\psi(x))$, then for some  $j\in [N]$ we have 
$$\models \theta_j(a;b_{i_1},\dots,b_{i_\varrho})\quad\text{ and }\quad R_\Delta(\theta_j(x;b_{i_1},\dots,b_{i_\varrho}))=R_\Delta(a/b_1,\dots,b_k).$$
Now let $B \subseteq M^{\abs{y}}$, $p \in S^{\Delta}(B)$, and $a\in M^{m}$ be a realization of $p$ with $R_\Delta(a/B)=R_\Delta(p)$. Choose $\phi_i(x;y)\in\Delta$ and $b_i\in B$, where $i\in [k]$, such that $R_\Delta(\bigwedge_{i=1}^k \phi_i(x;b_i))=R_\Delta(p)$. Then there are  $i\in [k]^\varrho$ and $j\in [N]$ such that $\models\theta_j(a;b_{i_1},\dots,b_{i_\varrho})$ and $R_\Delta(\theta_j(x;b_{i_1},\dots,b_{i_\varrho}))=R_\Delta(p)$.
\end{proof}

\begin{proof}[Proof of Theorem~\ref{mainst}.]
We wish to count the number of $\Delta(x;B)$-types over finite parameter sets $B\subseteq M^{\abs{y}}$.
Fix $\mathcal L$-formulas $\theta_i(x;y_1,\dots,y_\varrho)$ with $i\in [N]$ as in Lemma~\ref{three}.  Let $D \in \bN$
bound the $\Delta$-degree of any instance of any $\theta_i$.  Note that $D$ exists by Fact~\ref{nfcp},~(3).  Fix a finite set $B \subseteq M^{\abs{y}}$.  Let $p \in S^{\Delta}(B)$ and let $a\in M^{m}$ be a $\Delta$-generic realization of
$p$.
There are $b_1,\dots,b_\varrho \in B$ and $i\in [N]$
so that $\models \theta_i(a; b_1,\dots,b_\varrho)$ and $R_{\Delta}(p)=R_{\Delta}(\theta_i(x; b_1,\dots,b_\varrho))$.
Since $\deg_\Delta(\theta_i(x;b_1,\dots,b_\varrho))\leq D$ there are global $\Delta$-types $q_1, \dots,
q_D \in S_{\Delta}(M)$ (not necessarily distinct) such that for any $a'$ with $\models \theta_i(a';b_1,\dots,b_\varrho)$ and
$R_{\Delta}(a'/B)=R_{\Delta}(\theta_i(x;b_1,\dots,b_\varrho))$ we have $\tp^{\Delta}(a'/B) \subseteq q_i$ some $i\in [D]$.
In particular $p \subseteq q_i$ for some $i\in [D]$. So $\abs{S^\Delta(B)}\leq ND\abs{B}^\varrho$, and our result follows.
\end{proof}

\subsection{$\Delta$-rank, $\URk$-rank, and Morley rank.}
We now  want to apply Theorem~\ref{mainst} to obtain a uniform bound on the VC~densities $\vc^*(\Delta)$ of finite sets of $\mathcal L$-formulas~$\Delta=\Delta(x;y)$ in terms of $m=\abs{x}$ and the $\URk$-rank $\URk(T)$ of $T$, provided the latter is finite (i.e., less than~$\omega$). For this, we first note the following relationship between the ranks $R_\Delta$ and the Lascar $\URk$-rank of a complete type:

\begin{lemma}
Let $p\in S_x(A)$ where $A\subseteq M$. Then:
\begin{enumerate}
\item $R_\Delta(p)\leq\URk(p)$ for all finite sets $\Delta=\Delta(x;y)$ of partitioned $\mathcal L$-formulas;
\item if $\URk(p)<\omega$, then there is some finite set $\Delta=\Delta(x;y)$ of partitioned $\mathcal L$-formulas such that $R_\Delta(p)=\URk(p)$.
\end{enumerate}
\end{lemma}
\begin{proof}
For (1) recall that the $\Delta$-ranks detect forking: $q\in S_x(B)$ with $A\subseteq B\subseteq M$ does not fork over $A$ iff $R_\Delta(q\restrict A)=R_\Delta(q)$ for every finite set $\Delta=\Delta(x;y)$ of partitioned $\mathcal L$-formulas. Given such $\Delta$, a straightforward transfinite induction now proves the implication $R_\Delta(p)\geq\alpha\Rightarrow\URk(p)\geq\alpha$, for all ordinals $\alpha$, as required.

For (2), we use that by \cite[Lemma~III.1.2,~(2)]{Shelah-book}, if a type $q\in S_x(B)$ with $A\subseteq B\subseteq M$ forks over $A$, then there is a finite $\Delta_0$ such that for every finite $\Delta$ containing $\Delta_0$ one has $R_\Delta(q)<R_\Delta(q\restrict A)$. So from a forking sequence 
$p=p_0 \subseteq p_1 \subseteq \cdots \subseteq p_n$ of types
of length $n$ one obtains a finite $\Delta$ such that $R_\Delta(p_0)>R_\Delta(p_1)>\cdots>R_\Delta(p_n)\geq 0$ and so $R_\Delta(p)=R_\Delta(p_0)\geq n$.
\end{proof}

A special case of Lascar's Inequality (see, e.g., \cite[Theorem~1.5]{Saffe}) states that if $\URk(T)<\omega$ and 
$a_1\in M^{m_1}$ and $a_2\in M^{m_2}$ are independent (over $\emptyset$), then $\URk(a_1,a_2)=\URk(a_1)+\URk(a_2)$. 
(Here $\URk(a):=\URk(\tp(a/\emptyset))$ for $a\in M^m$.) Thus in this case all types in $S_m(T)=S_m(\emptyset)$ have
$\URk$-rank at most $m\URk(T)$. 
Hence from Theorem~\ref{mainst} and the previous lemma we obtain the following linear bound on the VC~density function of $T$ (that is, Theorem~\ref{thm:finite MR} from the introduction):

\begin{corollary}\label{cor:mainst, finite U-rank}
Suppose $T$ has finite $\URk$-rank $d=\URk(T)$ and does not have the finite cover property. Then if $\Delta$ is finite we have $\pi_\Delta^*(t)=O(t^{d\,\abs{x}})$, and hence $\vc^T(m)\leq dm$ for each~$m$.
\end{corollary}

Since $\URk(T)\leq\MRk(T)$, the previous corollary applies in particular if $T$ has finite Morley rank and does not have the finite cover property. 
The following folklore result (see, e.g., \cite[Proposition~B.1]{PP}, \cite{Poizat-finite MR}, or \cite[Theorem~4.5]{Saffe}) can be used to verify that $T$ has finite Morley rank:

\begin{proposition} \label{prop:U=Mrk}
Suppose there is a family $\{D_i\}_{i\in I}$ of strongly minimal definable sets  such that every non-algebraic type in $T$ is non-orthogonal to some $D_i$. Then $\URk(p)=\MRk(p)<\omega$ for each $p\in S_m(B)$, $B\subseteq M$.
\end{proposition}

\begin{example} \mbox{} \label{ex:U=Mrk}
Suppose $T$ is totally transcendental and unidimensional (i.e., any two non-algebraic $1$-types are non-orthogonal). Then $T$ satisfies the hypothesis of the previous proposition with a single $D_i$. (The fact that $\MRk(T)$ is finite for $\aleph_1$-categorical countable $T$ was first shown in \cite{Baldwin}.)
\end{example}

\subsection{Applications.}\label{sec:applications}
Together with the remarks in the previous subsection, we can now draw some immediate consequences of Theorem~\ref{mainst}. 

\subsubsection{Totally transcendental $\aleph_0$-categorial theories.}
First we discuss the   $\aleph_0$-cate\-gorical case. It is well-known that a stable $\aleph_0$-categorical theory does not have the finite cover property. (This is a consequence of \cite[Theorem~II.4.4]{Shelah-book}, see \cite[Lemma~2.4]{Saffe}).) 
If $T$ is countable, $\aleph_0$-categorical, and $\omega$-stable, then $T$ has finite Morley rank, in fact,  $\MRk(T)\leq \abs{S_2(\emptyset)}$; see \cite[Corollary~4.3 and Theorem~5.1]{CHL}. In this case moreover $\URk(p)=\MRk(p)$ for each  $m$-type $p$ over a finite parameter set, by \cite[Theorem~2.1]{Saffe}.
Hence by Corollary~\ref{cor:mainst, finite U-rank}:

\begin{corollary} Suppose $T$ is countable, totally transcendental, and $\aleph_0$-categori\-cal. 
Then $\vc^T(m)\leq m\URk(T)=m\MRk(T)\leq m\,\abs{S_2(\emptyset)}$ for every $m$.
\end{corollary}

\subsubsection{Dimensional theories.}
Recall that $T$ is called dimensional (also known as non-multidimensional) if the set of domination equivalence classes of non-algebraic stationary $1$-types is small.
Dimensional theories do not have the finite cover property (see \cite[Lem\-ma~IX.1.10]{Shelah-book}). 
Hence by Corollary~\ref{cor:mainst, finite U-rank} and Example~\ref{ex:U=Mrk}  above:

\begin{corollary}\label{cat}  Suppose $T$ is totally transcendental and unidimensional.  Then $\URk(T)=\MRk(T)<\omega$ and $\vc^T(m)\leq m\MRk(T)$ for every $m$.
\end{corollary}

The previous corollary applies if $T$ is countable and $\aleph_1$-categorical.
In particular, we see that if  $T$ is strongly minimal, then $\vc^T(m)=m$ for every~$m$. For example, if $T=\ACF_p$ is the theory of algebraically closed fields of characteristic~$p$ (where~$p$ is a prime or $p=0$), then $\vc^T(m)=m$ for every~$m$. 

\medskip
\noindent
Every  $\URk$-rank $1$ theory is dp-minimal \cite[Fact~3.2]{OU}. Our next application strengthens this remark:

\begin{corollary}
Suppose $\URk(T)=1$. Then $\vc^T(m)=m$ for every~$m$.
\end{corollary}

\begin{proof}
By Theorem~21 in \cite{bpw}, after expanding $\mathcal L$ by constant symbols if necessary, one may assume that $T$ is \emph{quasi strongly minimal} in the sense of \cite{bpw}, that is: in each model of $T$, the one-variable definable sets are precisely the Boolean combinations of $\emptyset$-definable or finite sets. 
(Quasi strongly minimal theories are the analogue, in stability theory, of quasi weakly o-minimal theories.)
In this case, by \cite[Theorem~20]{bpw} every non-algebraic $1$-type $p\in S_1(A)$ is determined by its restriction $p\restrict\emptyset$ to the empty parameter set. In particular, $T$ is dimensional, and hence does not have the finite cover property, thus Corollary~\ref{cor:mainst, finite U-rank} applies.
\end{proof}

For a simple example where this corollary applies (which can also be dealt with directly), consider the language $\mathcal L=\{P_i:i\in\bN\}$, where each $P_i$ is a unary predicate symbol, and the  $\mathcal L$-theory $T$ of an infinite set with the $P_i$  interpreted by independent predicates, i.e.: 
$$T\models\exists^{\geq n} x\left(\textstyle\bigwedge_{i\in I} P_i x\wedge\bigwedge_{j\in J}\neg P_j x\right)\qquad\text{ for all $n$ and finite disjoint $I,J\subseteq\bN$.}$$
(See also \cite[Proposition~3.7]{dl}.)

\subsubsection{Groups of finite Morley rank.}
Groups of finite Morley rank are finite-dimen\-sional and hence do not have the finite cover property \cite[Lemme~4, p.~457]{Lascar}. Hence by Corollary~\ref{cor:mainst, finite U-rank}:

\begin{corollary} \label{cor:finite MR groups}
Suppose $\mathcal L$ is an expansion of the language of groups and  $T$ is an expansion of the theory of groups with $\MRk(T)<\omega$.  Then $\vc^T(m)\leq m\MRk(T)$ for every~$m$.
\end{corollary}

This corollary applies, in particular, to the complete theory of a module of finite Morley rank. However, in this case we may use instead Corollary~\ref{cor:lower bounds on vc} and Lemma~\ref{lem:tt} below to obtain the same result.

\subsection{Examples and remarks.}
We first point out that the bound on VC~density given in Theorem~\ref{mainst} is  far from sharp:

\begin{example*}
Let $\mathcal L=\{E_n:n\geq 0\}$ be the language consisting of
countably many binary relations $E_n$.  Let $T$ be the theory whose axioms state that each $E_n$ is an equivalence relation
with infinitely many infinite classes and so that $E_{n+1}$ refines $E_n$ in such a way that each $E_n$-class is partitioned into infinitely many
$E_{n+1}$-classes.  The theory $T$ is stable and does not have the finite cover property.  Fix some $n$ and let $\Delta(x;y)=\{E_0(x;y), \dots, E_{n-1}(x;y)\}$.  It is
easy to see that $\vc^*(\Delta)\leq 1$  yet $R_{\Delta}(x=x) \geq n$.
\end{example*}

Similarly one may ask if not having the finite cover property is necessary for $T$ to have minimal density.  This is also seen to be false by a simple example.  

\begin{example*}
Let $\mc{L}$ be the language with a single binary relation $E$ and let $T$ be the theory stating that $E$ is an equivalence relation
with one class of size $n$ for each~$n$.  Then $T$ is $\omega$-stable and has the finite cover property.  But it is easy to see that $\vc^*(\Delta)\leq\abs{x}$ 
for any finite set  $\Delta=\Delta(x;y)$ of partitioned $\mathcal L$-formulas. 
\end{example*}

Laskowski \cite{chrisdef} has shown that for every $\mathcal L$-formula $\varphi(x;y)$ there is an $\mathcal L$-formula
$\operatorname{d}_\varphi(y; y_1,\dots,y_d)$ with $d=R^m(x=x,\varphi,2)$, where $m=\abs{x}$, so that if $B$ is a set of $\abs{y}$-tuples (possibly infinite)
and $p \in S_{\varphi}(B)$, then there are $b_1,\dots,b_d \in B$ so that $\operatorname{d}_\varphi(y; b_1, \dots, b_d)$ defines~$p$, and hence $\vc^*(\varphi)\leq d$. (See also \cite[Section~5.3]{ADHMS}.) 
We may ask if a similar situation
 holds in the case that $T$ does not have the finite cover property and we replace $R^m({-},\varphi,2)$ by $R_{\varphi}({-})$.  If this were the case, then our analysis of VC~density in the superstable finite-rank case would mirror much of our analysis in other cases given in \cite{ADHMS}, in that bounds on VC~density arise out
 of a bound on the number of parameters needed to define types.  We show by a very simple example that this is not the case.  Namely we
 show that for any $n>0$ there is a totally categorical theory $T_n$ with $\MRk(T_n)=1$ in which for a fixed formula $\varphi(x;y)$ with $\abs{x}=1$ the defining scheme for $\varphi$-types requires the naming of $n$~parameters. (This example also illustrates why VC~density is, arguably, a better measurement of complexity than the number of parameters in the defining formulas for types.)

\begin{example}\label{ex:n-sets}
 Let $\mathcal L=\{U,S,E\}$ be the language consisting of a unary predicate $U$ and a binary predicates $S$ and $E$. Fix $n>0$, and let $\mathbf M_n=(M_n,U,S,E)$ be the $\mathcal L$-structure with the following properties (see Figure~\ref{fig:equ class}):
\begin{enumerate}
\item $E$ is an equivalence relation with countably many classes, each of size $2n+{2n \choose n}$;
\item $U$ splits every $E$-class $C$ so that $\abs{C \cap U}=2n$ and
 $\abs{C \cap \neg U}={2n \choose n}$;
\item $S$ holds only between $E$-equivalent elements and
 only between elements of $U$ and $\neg U$; also, for each $E$-class $C$ and each $n$-element subset $X$ of  $C \cap U$  let there be a unique $b_X$ in $C \cap \neg U$ so
 that $S(b_X,y)$ holds if and only if $y \in X$, and let these be the only pairs in $S$. 
\end{enumerate}
The $\mathcal L$-theory $T_n=\Th(\mathbf M_n)$ is totally categorical with $\MRk(T_n)=1$, and Corollary~\ref{cat} applies.  For a fixed $E$-class $C$, if $Y \subseteq C \cap U$ has
 size less than $n$, then for any $a_1,a_2 \in (C\cap U) \setminus Y$ there is an automorphism of $\mathbf M_n$ fixing $Y$ and switching $a_1$ and $a_2$.
 Thus fixing an $E$-class $C$ and letting $A=C \cap U$, if $Z$ is an $A$-definable set definable with fewer than $n$ parameters from $A$, then
 $\abs{Z \cap A} \not= n$.  Hence for any $X \subseteq A$ the type $\tp_{S(x,y)}(b_X/A)$ is not definable using fewer than~$n$ parameters from $A$, yet the
 VC~density of $S(x,y)$ is~$1$ by Corollary~\ref{cat} and $R_{S(x,y)}(x=x)=1$.
\end{example}

\begin{figure}  
\vskip1em

\begin{tikzpicture}
\draw (0,0) rectangle (6,3);
\draw (3,3) -- (3,0);
\fill [black] (1,1) circle (2pt);
\fill [black] (1,2) circle (2pt);
\fill [black] (2,2) circle (2pt);
\fill [black] (2,1) circle (2pt);

\fill [black] (4,0.75) circle (2pt);
\fill [black] (5,0.75) circle (2pt);

\fill [black] (4,1.5) circle (2pt);
\fill [black] (5,1.5) circle (2pt);

\fill [black] (4,2.25) circle (2pt);
\fill [black] (5,2.25) circle (2pt);

\path (6.5,1) node {$C$};
\path (1.5, -0.5) node {$C\cap U$};
\path (4.5, -0.5) node {$C\cap \neg U$};

\path (4, 2.7) node (bX) {$b_X$};

\draw (-0.15,2.10) [rotate=-47] ellipse (40pt and 10pt);

\path (2.6, 1) node {$X$};

\draw (1,2) .. controls (2,2.5) .. (4, 2.25);

\draw (2,1) .. controls (3,1.5) .. (4, 2.25);

\end{tikzpicture}
\caption{An $E$-equivalence class $C$ in the case $n=2$}
 \label{fig:equ class}

\end{figure}
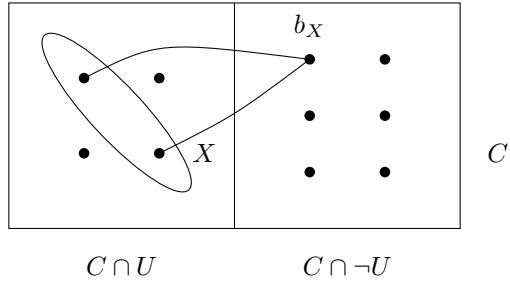

\section{Modules of Finite Breadth}\label{sec:examples of VCm theories, 2}

\noindent
Let $R$ be a ring (with $1$). In this paper, ``$R$-module'' always means ``left $R$-module.''  Let $\mathcal L_R$ be the language of $R$-modules and  $M$ be an $R$-module, construed as an $\mathcal L_R$-structure in the natural way. 
Recall that an $\mathcal L_R$-formula $\varphi(x)$  is said to be positive-primitive (p.p., for short) if $\varphi$ is equivalent (in the theory of $R$-modules) to a formula of the form $\exists y\, Ax=By$ where  $A\in R^{k\times \abs{x}}$, $B\in R^{k\times \abs{y}}$, for some~$k$ (where we employed the usual shorthand notation of combining a conjunction of linear equations into a single matrix equation). By the Baur-Monk Theorem, each $\mathcal L_R$-formula is equivalent, in $\Th(M)$,  to a Boolean combination of p.p.~$\mathcal L_R$-formulas.
(Our main references for the model theory of modules are \cite{Prest} and \cite[Appendix~A.1]{Hodges}.) 
In this section we first introduce a notion of breadth of~$M$, defined via its lattice of p.p.~definable subgroups, and, then show that if $M$ has finite breadth $d$, then $M$ has the $\VC{}d$ property.  
(Our usage of breadth is not obviously related to the concept of breadth of a module considered in \cite[Chapter~10]{Prest}.) We also show that uniform bounds on the VC~density (in $M$) of formulas with~$m$ parameter variables can be computed from a certain quotient of the lattice of p.p.~definable subgroups of~$M^m$.

\subsection{Dimension, width, height and breadth.}\label{sec:dimension ...}
We begin by recalling a few definitions and basic facts from the theory of (partially) ordered sets. In the following $(P,{\leq})$ denotes an ordered set. (We also simply denote $(P,{\leq})$ by $P$ if the ordering  $\leq$  is understood from the context.) 
Given another ordered set $(P',{\leq'})$, a map $f\colon P\to P'$ is said to be increasing (or
a morphism of ordered sets) if $a\leq b\Rightarrow f(a)\leq' f(b)$ for all $a,b\in P$ and strictly increasing if $a<b\Rightarrow f(a)<' f(b)$ for all $a,b\in P$. We call $f$
an embedding of ordered sets if $a\leq b\Longleftrightarrow f(a)\leq' f(b)$ for all $a,b\in P$.
Unless otherwise stated, proofs of the facts mentioned in the rest of this subsection can be found in any standard text on ordered sets (e.g., \cite{Harzheim}). 

\medskip
\noindent
The ordering $\leq$ can be extended to a linear ordering on $P$ (Szpilrajn/Mar\-czew\-ski). Moreover, if $\{{\leq}_i\}_{i\in I}$ is the family of all linear orderings on $P$ extending $\leq$, then ${\leq}=\bigcap_{i\in I} {\leq_i}$, i.e., for all $a,b\in P$ one has $a\leq b$ iff $a\leq_i b$ for every $i\in I$.
The \emph{dimension} of $P$ is the smallest non-zero cardinal $d=\dim(P)$  such that  there is a family $\{{\leq}_i\}_{i\in I}$ of $d=\abs{I}$ linear orderings  on $P$ with ${\leq}=\bigcap_{i\in I} {\leq_i}$. Equivalently, $\dim(P)$ is the smallest cardinal $d>0$ such that $P$ can be embedded into a direct product of $d$ chains. Here and below, given a non-empty family $\{(P_i,{\leq_i})\}_{i\in I}$ of ordered sets $(P_i,{\leq_i})$, the direct product $\prod_{i\in I} P_i$ is equipped with the ordering defined by $(a_i)\leq (b_i)$ iff $a_i\leq_i b_i$ for every $i\in I$.
We are mostly interested in finite-dimensional ordered sets. 
For all finite-dimensional ordered sets $P$ and $P'$,
$$\dim(P\times P') \leq \dim(P)+\dim(P'),$$
with equality if both $P$ and $P'$ have a greatest and a smallest element \cite[Chapter~2,(3.4) and (3.5)]{Trotter}.

\medskip
\noindent
The \emph{width} of $P$ is defined to be the supremum of the cardinalities of antichains in $P$, and denoted by $\width(P)$.  Clearly if $P$ is non-empty, then $P$ is a chain iff $\dim(P)=1$ iff $\width(P)=1$. Dilworth's Theorem states that if $\width(P)$ is finite, then it equals the smallest number of chains in $P$ whose union is $P$.
As a consequence of this theorem one obtains the inequality
\begin{equation}\label{eq:dilworth}
\dim(P)\leq \width(P),
\end{equation}
provided $P$ is of finite width. 
Dually, the \emph{height} of $P$ is defined to be the supremum of the cardinalities of a chain in $P$, denoted by $\height(P)$.

\medskip
\noindent
The \emph{breadth} of $P$, denoted by $\breadth(P)$, is the smallest integer $d\geq 0$ (if it exists) with the following property: for all $x_1,\dots,x_{d+1},y_1,\dots,y_{d+1}\in P$ such that $x_i\leq y_j$ for all $i\neq j$ in $[d+1]$ there exists $i\in [d+1]$ such that $x_i\leq y_i$. If there is no such $d$, then we set $\breadth(P)=\infty$. This definition of breadth for an arbitrary ordered set (which is ``self-dual'': $\breadth(P,{\leq})=\breadth(P,{\geq})$) is due to Wehrung. It is easy to easy to see that if $(L,{\wedge})$ is a (meet-) semilattice, then the breadth of $L$ may also be described as the smallest $d$ (if it exists) such that for all $x_1,\dots,x_n\in L$ with $n>d$ there are $i_1,\dots,i_d\in [n]$ with $x_1\wedge\cdots\wedge x_n = x_{i_1}\wedge\cdots\wedge x_{i_d}$. (Compare with Section~\ref{sec:breadth}.) Let  $P'$ be another ordered set; then
$$\breadth(P\times P')\leq\breadth(P)+\breadth(P'),$$
with equality if $P$ and $P'$ have  a greatest and a smallest element, and if $P'$ embeds into  $P$,  then $\breadth(P')\leq\breadth(P)$.
Together with \eqref{eq:dilworth} this yields the following fundamental inequality: if $P$ has finite width,  then
\begin{equation}\label{eq:dilworth, 2}
\breadth(P) \leq \dim(P) \leq \width(P).
\end{equation}
Note that in general there is no bound on $\breadth(P)$ in terms of $\height(P)$: the ``standard example $\operatorname{St}_d$'' of a $d$-dimensional ordered set (see, e.g., \cite[Chapter~1,~(5.1)]{Trotter}) actually has breadth $d$, yet has height $2$. However, if $(L,{\wedge})$ is a semilattice, then 
\begin{equation}\label{eq:breadth vs height}
\breadth(L)\leq\height(L),
\end{equation}
with a strict inequality if $L$ has a largest element. (See also \cite[Example~2.16]{ADHMS}.) 

\medskip
\noindent
We have $\breadth(L')\leq\breadth(L)$ if there exists a surjective morphism of semilattices $L\to L'$. A map $f\colon L\to L'$ between semilattices is a morphism of semilattices if $f(a\wedge b)=f(a)\wedge f(b)$ for all $a,b\in L$. (This condition is stronger than merely being a morphism of ordered sets, when $L$, $L'$ are viewed only as ordered sets.) An injective morphism of semilattices is called an embedding of semilattices. By \cite[Section~II.5, Exercise~6~(c)]{Birkhoff} we have $\breadth(L)\geq n$ iff $L$ embeds the semilattice~$(2^{[n]},{\cap})$.

\subsection{Breadth in modular and distributive lattices.}
In this subsection, $L$ denotes a lattice with smallest element $0$ and largest element $1$, and we assume $0\neq 1$. 
Recall that $L$ is said to be modular if the identity
$$(a \wedge c) \vee (b \wedge c) = \big((a \wedge c) \vee b\big) \wedge c$$
holds for all $a,b,c\in L$.
So far we have discussed upper bounds on the breadth of ordered sets and semilattices.
Lower bounds on the breadth of a modular lattice can be obtained via its (dual) Goldie dimension, and for distributive lattices, the breadth can be  computed as the width of an associated (often simpler) ordered~set.

\medskip
\noindent
A subset $A$ of $L\setminus \{0\}$ is said to be {\it join-independent}\/ if $(\bigvee A')\wedge a = 0$ for all finite subsets $A'$ of $A$ and all $a\in A\setminus A'$. The {\it Goldie dimension}\/ of $L$ is the largest $n=\Gdim(L)$ such that $L$ contains a join-independent subset of size $n$, if there is such an $n$; otherwise we set $\Gdim(L)=\infty$.
The Goldie dimension of the dual $L^*$ of $L$ is called the \emph{dual Goldie dimension} of $L$ and denoted by $\Gdim^*(L)$. (See~\cite{GP}.)

\begin{lemma}\label{lem:Goldie dim vs breadth}
Suppose $L$ is modular. Then
$$\breadth(L)\geq \max\big\{\Gdim(L),\Gdim^*(L)\big\}.$$
\end{lemma}

\begin{proof}
Since the lattices $L$ and $L^*$ have the same breadth, it suffices to prove that $\breadth(L)\geq\Gdim(L)$. Let $A\subseteq L\setminus\{0\}$ be a finite join-independent set. For each $a\in A$ put $\widehat{a}:=\bigvee (A\setminus\{a\})$. Then
for each $A'\subseteq A$ we have
$$\bigwedge_{a\in A'} \widehat{a} = \bigwedge_{a\in A'} \bigvee (A\setminus\{a\}) =  \bigvee \left(\bigcap_{a\in A'} A\setminus\{a\}\right) = \bigvee A\setminus A',$$
using \cite[Chapter~IV, Theorem~11]{Graetzer}. In particular $\bigwedge_{a\in A} \widehat{a}=0\neq \bigwedge_{a\in A'} \widehat{a}$ for all proper subsets $A'$ of $A$. This shows $\breadth(L)\geq\abs{A}$, so yields the claim.
\end{proof}

Distributive lattices form an important subclass of the class of modular lattices.
Recall that $L$ is said to be distributive~if
$$a\vee (b\wedge c) = (a\vee b) \wedge (a\vee c)\qquad\text{for all $a,b,c\in L$.}$$
Equivalently, a lattice  is distributive iff it is isomorphic to a sublattice of $2^X$ (with the usual meet and join), for some set $X$
(Birkhoff-Stone~Theorem). For finite distributive lattices, such an isomorphism can be described explicitly:
Recall that a non-zero element $a$ of $L$ is join-irreducible if whenever $a=x\vee y$ with $x,y\in L$, then $a=x$ or $a=y$, and let $J(L)$ be the set of join-irreducible elements of $L$, equipped with the ordering induced by~$L$.
Also, for an ordered set $P$ we denote by $I(P)$ the set of initial segments of $P$, ordered by inclusion, and for $a\in L$ let 
$$(a):=\{x\in J(L):x\leq a\}\in I(J(L)).$$ 
If $L$ is \emph{finite} and distributive, then the map $a\mapsto (a)\colon L\to I(J(L))$ is an isomorphism.

\begin{proposition}[Dilworth, cf.~{\cite[Chapter~2, (8.7)]{Trotter}}] \label{prop:dilworth}
Suppose $L$ is finite and distributive. Then
$$\breadth(L) = \dim(L) = \width(J(L)), \qquad \height(L)=\abs{J(L)}+1,$$
and $\Gdim(L)$ is the maximum number of pairwise disjoint initial segments of $J(L)$.
\end{proposition}

\begin{example*}
Suppose $L=2^{[n]}$. Then $L$ is distributive, and
$$\breadth(L)=\dim(L)=\Gdim(L)=n, \qquad \height(L)=n+1.$$
It is also well-known that
$\width(L)={n\choose \lfloor n/2\rfloor}$ (Sperner's Theorem).
\end{example*}

\subsection{The breadth of the lattice of p.p.~definable subgroups.}\label{sec:breadth and pp subgroups}
Let $\PP=\PP_R$ be the set of p.p.~$\mathcal L_R$-formulas in the single indeterminate $x$ modulo equivalence in the theory of $R$-modules. In the following we don't distinguish between a p.p.~$\mathcal L_R$-formula and its representative in $\PP$. The set $\PP$ has the structure of a modular lattice with meet and join given by
\begin{align*}
(\varphi\wedge\psi)(x)	&:= \varphi(x)\wedge\psi(x), \\
(\varphi\vee\psi)(x)		&:= \exists y\exists z (\varphi(y)\wedge \psi(z)\wedge x=y+z).
\end{align*}
having a smallest element (represented by $x=0$) and a largest element (represented by $x=x$).
Given an  $R$-module $M$, we also equip the collection of p.p.~definable subgroups
$$\PP_m(M)=\big\{\varphi(M):\text{$\varphi(x)$ p.p.~$\mathcal L_R$-formula with $\abs{x}=m$}\big\}$$ 
of $M^m$ with the structure of a modular lattice via
$$M_1\wedge M_2 = M_1\cap M_2, \quad M_1\vee M_2 = M_1+M_2\qquad\text{for $M_1,M_2\in\PP_m(M)$,}$$
having smallest element $\{0\}$ and greatest element $M^m$.
For all $m_1$, $m_2$ we have an embedding of lattices
\begin{equation}\label{eq:embedding of PPm's}
\PP_{m_1}(M)\times\PP_{m_2}(M) \to \PP_{m_1+m_2}(M)
\end{equation}
given by 
$$(H_1,H_2)\mapsto \big(H_1\times \{0\}^{m_2}\big)\oplus \big(\{0\}^{m_1}\times H_2\big).$$
If $N$ is a submodule of $M$ and $N$ is p.p.~definable in $M$, then clearly $\PP_m(N)$ is a sublattice of $\PP_m(M)$, and denoting the natural surjection $M^m\to (M/N)^m$ by~$\pi$, we have an injective lattice morphism 
$$H\mapsto\pi^{-1}(H)\colon \PP_m(M/N)\to\PP_m(M).$$
If $N$ is a pure submodule of $M$ (e.g., if $N$ is a direct summand of $M$), then for every p.p.~definable subgroup $H$ of $M^m$, the subgroup $H\cap N^m$ is p.p.~definable (by the same p.p.~formula), and the map 
\begin{equation}\label{eq:pure incl}
H\mapsto H\cap N^m\colon\PP_m(M)\to\PP_m(N)
\end{equation}
is a surjective morphism of lattices. 
If $M$ is an $R$-module and $I$ is a non-empty index set,  then after choosing some arbitrary $i_0\in I$
and identifying $M$ with the $i_0$-component of $N=M^{(I)}$ in the natural way, 
$M$ becomes a direct summand of $N=M^{(I)}$, and the resulting map \eqref{eq:pure incl} is an isomorphism of lattices.

\medskip

The following is also easy to verify (cf.~\cite[Lemma~2.10]{Prest}):

\begin{lemma}\label{lem:pp definable for direct sum}
Suppose $\{M_i\}_{i\in I}$ is a family of $R$-modules, where $I\neq\emptyset$. Then $M=\bigoplus_{i\in I} M_i$ is pure in $\prod_{i\in I} M_i$, and the map 
$$H\mapsto H\cap M^m\colon\PP_m\left(\textstyle\prod_{i\in I} M_i\right)\to\PP_m(M)$$ 
is bijective. Moreover, for each $H\in\PP_m(M)$ we have $H=\bigoplus_{i\in I} H\cap M_i$, and the map
\begin{equation}\label{eq:lattice embedding}
H\mapsto (H\cap M_i)_{i\in I}\colon\PP_m(M) \to \textstyle\prod_{i\in I} \PP_m(M_i)
\end{equation}
is an embedding of lattices. If $I$ is finite and each $M_i$ is p.p.~definable in $M$, then~\eqref{eq:lattice embedding} is onto.
\end{lemma}

In the rest of this subsection we focus on $\PP(M)=\PP_1(M)$. 
The map $\varphi\mapsto\varphi(M)\colon \PP\to\PP(M)$ is a surjective morphism of lattices.
Note that there is an $R$-module $M^*$ such that this map $\PP\to\PP(M^*)$ is an isomorphism of lattices (e.g., take $M^*=$ the direct sum of a set of representatives for all isomorphism types of finitely presented $R$-modules, cf.~\cite[Corollary~8.17]{Prest}). In particular
$$\breadth(\PP_R)=\max_M\breadth(\PP(M)).$$
We say that $M$ has breadth $d$ if $\PP(M)$ has breadth $d$, and similarly for width and height. (So $\PP_R$ has finite breadth iff there is a uniform [finite]  bound on the breadth of all $R$-modules.)
We collect some basic properties of breadth (immediate from the preceding discussions):

\begin{lemma}\label{lem:basic properties of breadth}
Let $M$, $N$ be $R$-modules. Then:
\begin{enumerate}
\item If $N$ is a p.p.~definable submodule of $M$, then  
$$\breadth(N),\breadth(M/N)\leq\breadth(M);$$
\item if $N$ is a pure submodule of $M$, then $\breadth(N)\leq\breadth(M)$;
\item if $I$ is a non-empty index set, then $\breadth(M^{(I)})=\breadth(M)$;
\item $\breadth(M\oplus N)\leq\breadth(M)+\breadth(N)$, with equality if
both $M$ and~$N$ are p.p.~definable in $M\oplus N$.
\end{enumerate}
\end{lemma}

It can be checked that the Goldie dimension $\Gdim(\PP(M))$ of the modular lattice $\PP(M)$ is the largest $n$ (if it exists) such that $M$ contains an internal direct sum $H_1\oplus \cdots \oplus H_n$ of $n$ non-zero p.p.~definable subgroups $H_1,\dots,H_n$ of~$M$. The dual Goldie dimension $\Gdim^*(\PP(M))$ of $\PP(M)$ is the largest $n$ (if it exists) such that $M$ contains proper p.p.~definable subgroups $H_1,\dots,H_n$ of~$M$ such that $H_i+\bigcap_{j\neq i} H_j=M$ for all $i\in [n]$.
We refer to $\Gdim(\PP(M))$ and $\Gdim^*(\PP(M))$ as the p.p.~Goldie dimension and the dual p.p.~Goldie dimension of $M$, respectively.
By Lemma~\ref{lem:Goldie dim vs breadth} we have 
$$\breadth(M)\geq\max\big\{\Gdim(\PP(M)),\Gdim^*(\PP(M))\big\}.$$
The following is the archetypical example of a module of infinite breadth:

\begin{example*}
Suppose $R=\bZ$ and $M=\bZ$, considered as a $\bZ$-module as usual. Then $\Gdim^*(\PP(M))=\infty$, since for each prime $p$ the subgroup $p\bZ$ is p.p.~definable, and for all pairwise distinct primes $p_1,\dots,p_d$, we have
$p_i\bZ+\bigcap_{j\neq i} p_j\bZ=\bZ$ for all $i\in [d]$.   Hence $\breadth(M)=\infty$ and so $\breadth(\PP_\bZ)=\infty$. (But note that $\Gdim(\PP(M))=1$.)
\end{example*}

If $M$ has only finitely many p.p.~definable subgroups (e.g., if $M$ is $\aleph_0$-categorical), then clearly $\breadth(M)$ is finite (and bounded by $\log_2\abs{\PP(M)}$). In fact, it is enough to require that $M$ has  finite height: $\breadth(M)<\height(M)$ by \eqref{eq:breadth vs height} (and this is strict since $\PP(M)$ has a largest element).
Note also that if $\height(M)$ is finite, then $\MRk(T) < \height(M)$, and if $M\equiv M^{\aleph_0}$ and $\MRk(T)<\omega$, then we have $\height(M) = \MRk(T) +1$ and hence $\breadth(M)\leq\MRk(T)$; here $T=\Th(M)$. 
(See \cite[Example~3.14]{ADHMS}.)

\begin{example*}
Prest \cite{PrestMR} showed that {\it all}\/ $R$-modules have finite Morley rank iff $R$ is of finite representation type, i.e.,  there are finitely many indecomposable $R$-modules $N_1,\dots,N_r$ such that every $R$-module is a direct sum of copies of the $N_i$. In this case there is a uniform (finite) bound $d$ on the Morley rank of $R$-modules,  hence $\breadth(\PP_R)\leq d$.
\end{example*}

The previous example raises the following question, for which we do not know the answer:

\begin{question}
For which rings $R$ do all $R$-modules have finite breadth, and for which rings~$R$ does $\PP_R$ have finite breadth?
\end{question}

If $\PP(M)$ is finite and distributive, then Proposition~\ref{prop:dilworth} allows us to compute the breadth of $\PP(M)$ as the width of the (often much simpler) ordered set $J(\PP(M))$ of join-irreducibles of $\PP(M)$.  It is well-known (cf.~\cite[Theorem~3.1]{EH}) that if $R$ is commutative, then for every $R$-module $M$ the lattice $\PP(M)$ is distributive iff~$R$ is a Pr\"ufer ring, i.e.,  the  lattice of all of its ideals is distributive.
(Equivalently, a commutative ring $R$ is Pr\"ufer iff the ideals of each localization $R_{\mathfrak m}$ of $R$ at a maximal ideal~$\mathfrak m$ of $R$ are totally ordered by inclusion.)

\subsection{Breadth and the $\VC{}d$ property.}
Let $M$ be an $R$-module and $T=\Th(M)$.
The following observation, connecting the breadth of $M$ with the $\VC{}d$~property, explains our interest in the concept of breadth:

\begin{proposition}\label{prop:breadth and vc for modules}
Suppose $M$ has \textup{(}finite\textup{)} breadth $d$. Then $T$ has the $\VC{}d$ property. In particular, if in addition $M$ is infinite, then $\vc^T(m)\leq dm$ for every~$m$.
\end{proposition}
\begin{proof}
Let $\Phi$ be the set of all p.p.~$\mathcal L_R$-formulas in the single object variable $x$. By Baur-Monk, every $\mathcal L_R$-formula $\varphi(x;y)$ is equivalent in $T$ to a Boolean combination of formulas from $\Phi$. 
So if $M$ has breadth $d$, then by  Lemma~\ref{lem:breadth of cosets}, the set system $\mathcal S_\Phi$  has breadth~$d$,
hence by Lemma~\ref{lem:breadth and UDTFS}, every finite subset of $\Phi$ has UDTFS with~$d$ parameters;
therefore $T$ has the $\VC{}d$ property by Lemma~\ref{lem:qe and VCd}.
\end{proof}

\begin{remark*}
A suitable modification of the argument in the proof above (using the natural multi-sorted version of Theorem~\ref{VCdensity, 2}, cf.~\cite[Corollary~5.8]{ADHMS}) shows more generally that if $\mathbf G$ is a $1$-based expansion of a group, and the meet-semilattice of $\acl^\eq(\emptyset)$-definable subgroups of $G$ has breadth $d$, then $(\mathbf G,G)$ has the $\VC{}d$ property, and hence
$\vc^{\Th(\mathbf G)}(m)\leq dm$ for every $m$. 
\end{remark*}

We record two immediate corollaries of the previous proposition. The $R$-module~$M$ is called \emph{p.p.-uniserial} if the ordered set $\PP(M)$ is a chain. 

\begin{corollary}\label{cor:pp uniserial}
If $M$ is p.p.-uniserial, then $T$ has the $\VC{}1$ property.
\end{corollary}

The following corollary offers a more precise result than Corollary~\ref{cor:finite MR groups} (but with an identical bound on $\vc^T$), under the additional assumption that $M^{\aleph_0}\equiv M$.

\begin{corollary}\label{cor:modules with finite MR}
Suppose $M^{\aleph_0}\equiv M$.
If $M$ has finite Morley rank $d$, then $M$ has the $\VC{}d$ property; in particular, if $\MRk(T)=1$,  then $T$ has the $\VC{}1$ property.
\end{corollary}

\subsection{Commensurability and VC~density.}
In the previous subsection, for the theory $T=\Th(M)$ of the infinite $R$-module $M$ we obtained an upper bound on the VC~density $\vc^T(m)$ in terms of $m$ and the breadth of the lattice $\PP(M)=\PP_1(M)$. In this subsection we show that $\vc^T(m)$ can indeed be computed precisely as the breadth of a suitable quotient of the lattice $\PP_m(M)$.
In fact, everything generalizes to $1$-based groups,  so we work in this wider setting. (In the last corollary of this subsection we additionally assume commutativity.) 

Let $G$ be a group (written multiplicatively). For subgroups $H_1$, $H_2$ of $G$ we write $H_1 \lesssim H_2$ if $H_1\cap H_2$ has finite index in $H_1$. It is easy to see that $\lesssim$ is a quasi-ordering on the set of subgroups of $G$ which extends the ordering by inclusion. The equivalence relation $\sim$ associated to $\lesssim$ is called commensurability: $H_1\sim H_2$ iff $H_1\cap H_2$ is of finite index in both $H_1$ and $H_2$. The $\sim$-class of a subgroup of $G$ is called its commensurability class; the quasi-ordering $\lesssim$ induces an ordering on the set of commensurability classes of subgroups of $G$.
One easily verifies:

\begin{lemma}
Let $H_1$ and $H_2$ be subgroups of $G$ with $H_1 \lesssim H_2$. Then
\begin{enumerate}
\item $H_1\cap H \lesssim H_2\cap H$ for every subgroup $H$ of $G$;
\item $H_1H\lesssim H_2H$ for every normal subgroup $H$ of $G$.
\end{enumerate}
\end{lemma}

In particular, by (1), $\cap$ induces a semilattice structure on the set of commensurability classes of subgroups of $G$; by (1) and (2), there is also a natural lattice structure on the set of commensurability classes of normal subgroups of $G$.

If $G$ is finite, there is only one commensurability class; from now on assume $G$ is infinite.
Let also $\mathbf G$ be an expansion of $G$ (viewed as a structure in the language of groups as usual). {\it Throughout the rest of this subsection we assume $T=\Th(\mathbf G)$ is $1$-based.}\/ We then have:

\begin{proposition}\label{prop:commensurable}
Suppose that the semilattice of commensurability classes of $\acl^\eq(\emptyset)$-de\-fin\-a\-ble subgroups of $G^m$ has finite breadth $d$. Then $\vc^T(m)\leq d$.
\end{proposition}
\begin{proof}   
By Lemma~\ref{lem:expansions by constants} we may assume that the language of $\mathbf G$ includes a constant symbol for every element of $G$.
Let $\Delta(x;y)$ be a finite set of partitioned
  formulas where $x$ is a tuple of variables of length $m$. 
By $1$-basedness (see \cite{hp}) and Lemma~\ref{lem:encoding finite sets of formulas} we may assume that each
instance of a formula $\varphi\in\Delta$ defines a coset of an $\acl^\eq(\emptyset)$-de\-fin\-a\-ble subgroup of 
$G^m$ that we denote by $H_\varphi$. (Here and in the rest of the proof,
  ``coset'' means ``left coset.'') 
Also  for every $\Psi\subseteq\Delta$
we   set $H_\Psi:=\bigcap_{\psi\in\Psi} H_\psi$. We may also assume
  $\abs{\Delta}\geq d$.

  For each non-empty subset $\Psi$ of $\Delta$ we choose
  $\Psi^{\#}\subseteq \Psi$ with $\abs{\Psi^\#}\leq d$ 
such that $H_{\Psi^\#}\sim
  H_\Psi$. Let $\varphi_{\Psi,1},\dots,\varphi_{\Psi,d}$ be the
  elements of $\Psi^\#$.  Also, choose representatives
  $c_{\Psi,1},\dots,c_{\Psi,n_\Psi}\in H_{\Psi^\#}$ for the cosets
  of $H_\Psi$ in $H_{\Psi^\#}$. Set $n:=\max_{\Psi\subseteq\Delta}
  n_\Psi$ where $n_\emptyset=1$.  For each $\Psi$ and a coset $E$ of
  $H_{\Psi^\#}$ in $G^m$ we also choose an arbitrary element $c_\Psi(E)$ of
  $E$.

  Let $B\subseteq G^{\abs{y}}$ be finite and non-empty. We define a map
\begin{equation}\label{eq:embedding into triples}
q\mapsto (\Psi_q,i_q,b_q)\colon S^{\Delta}(B)\to 2^{\Delta}\times [n] \times B^d
\end{equation}
as follows: given $q\in S^\Delta(B)$ 
let
$$\Psi = \Psi_q := \{ \psi\in\Delta:\text{$\psi(x;b)\in q$ for some $b\in B$}\},$$
so
$$D = D_q := \bigcap_{\psi(x;b)\in q} \psi(G;b)$$
is a coset of $H_\Psi$.  If $\Psi=\emptyset$ we let $i_q:=1$ and
$b_q\in B^d$ be arbitrary. Suppose $\Psi\neq\emptyset$, and take
$b=b_q=(b_1,\dots,b_d)\in B^d$ such that $\varphi_{\Psi,j}(x;b_j)\in
q$ for $j=1,\dots,d$. So~$D$ is contained in the coset
$E=\bigcap_{j=1}^d \varphi_{\Psi,j}(G;b_j)$ of $H_{\Psi^\#}$, hence
$c_\Psi(E)^{-1}D$ is a coset of $H_\Psi$ contained in $H_{\Psi^\#}$
and thus of the form $c_\Psi(E)^{-1}D = c_{\Psi,i} H_\Psi$ for a
unique $i=i_q\in [n_\Psi]$.  We have
$$D = c_\Psi(E) c_{\Psi,i}  H_\Psi \qquad\text{where $E=\varphi_{\Psi,1}(x;b_1)\cap\cdots\cap\varphi_{\Psi,d}(x;b_d)$,}$$
so the triple $(\Psi,i,b)$ uniquely determines $q$. Hence our map \eqref{eq:embedding into triples} is one-to-one. This shows that $\vc^*(\Delta)\leq d$ as required.
\end{proof}

It would be desirable if, strengthening Proposition~\ref{prop:breadth and vc for modules}, one could deduce linear growth of $\vc^T$ by simply knowing that the semilattice of commensurability classes of $\acl^\eq(\emptyset)$-definable subgroups of $G$ has finite breadth. An answer to the following natural question (which we do not know) would yield this:

\begin{question}
  Suppose the semilattice of commensurability classes of
  $\acl^\eq(\emptyset)$-definable subgroups of $G$ has finite breadth
  $d$.  Does the structure $\mathbf G$ have the $\VC{}d$ property?
  \textup{(}Of course, the choice of the coset element $c_\Psi(E)$ in the proof
  above is the main obstacle.\textup{)}
\end{question}

This proposition gives rise to a characterization of dp-minimal $1$-based groups (complementing the stability-theoretic characterization of stable dp-mini\-mal theories given in~\cite{OU}):

\begin{corollary}\label{cor:commensurable, 1}
The following are equivalent:
\begin{enumerate}
\item $T$ is dp-minimal;
\item $T$ is vc-minimal: $\vc^T(1)<2$;
\item $\vc^T(1)=1$;
\item the set of commensurability classes of $\acl^\eq(\emptyset)$-definable subgroups of $G$ is linearly ordered by $\lesssim$.
\end{enumerate} 
\end{corollary}
\begin{proof}
The implication (4)~$\Rightarrow$~(3) holds by Proposition~\ref{prop:commensurable}, (3)~$\Rightarrow$~(2) is trivial, and (2)~$\Rightarrow$~(1) holds by \cite[Corollary~5.13]{ADHMS}, so it only remains to show (1)~$\Rightarrow$~(4). Suppose $H_1$, $H_2$ are $\acl^\eq(\emptyset$)-definable subgroups of $G$ which are incomparable with respect to $\lesssim$, i.e., $[H_i:H_1\cap H_2]=\infty$ for each $i=1,2$. Let $H=H_1\cap H_2$ and
choose elements $b_{ij}\in H_{3-i}\setminus H$ ($i=1,2$, $j\in\bN$) with $b_{ij}b_{ik}^{-1}\notin H$ for $i=1,2$ and distinct $j,k\in\bN$.
For $j_1,j_2\in\bN$ set $a_{j_1,j_2} = b_{1j_1}\cdot b_{2j_2}$. 
We also let $\varphi_1(x;y)$ be an $\mathcal L^{\eq}$-formula stating that $y^{-1}x\in H_1$ and $\varphi_2(x;y)$ be an $\mathcal L^{\eq}$-formula stating that $xy^{-1}\in H_2$. Then
\begin{align*}
\mathbf G\models \varphi_1(a_{j_1,j_2};b_{1j})	&\qquad\Longleftrightarrow\qquad b_{1j}^{-1}b_{1j_1}b_{2j_2}\in H_1 \\
												&\qquad\Longleftrightarrow\qquad b_{1j}^{-1}b_{1j_1}\in H_1 \\
												&\qquad\Longleftrightarrow\qquad j=j_1,
\end{align*}
and similarly $\mathbf G\models \varphi_2(a_{j_1,j_2};b_{2j})$ iff $j=j_2$. Hence $\varphi_1(x,y)$, $\varphi_2(x,y)$ and the sequences $(b_{ij})_{j\in\bN}$ ($i=1,2$) form an ICT pattern in $(\mathbf G,G)$. This gives rise to an ICT pattern in $\mathbf G$, showing that $\mathbf G$ is not dp-minimal. 
\end{proof}

We also obtain the description of $\vc^T(m)$ advertised at the beginning of this subsection. 

\begin{corollary}\label{cor:commensurable, 2}
Suppose $G$ is abelian. Then $\vc^T(m)$ equals the breadth of the lattice of commensurability classes of $\acl^\eq(\emptyset)$-definable subgroups of $G^m$.
\end{corollary}

This is an immediate consequence of Proposition~\ref{prop:commensurable} and the following observation (in which we revert back to additive notation) generalizing the argument in the proof of (1)~$\Rightarrow$~(4) in Corollary~\ref{cor:commensurable, 1}:

\begin{lemma}\label{lem:breadth and vc for modules}
Let $A$ be an expansion of an infinite abelian group in a language~$\mathcal L$ expanding the language $\{0,{+}\}$ of abelian groups. Suppose $A^m$ contains definable subgroups $H_1,\dots,H_d$, where $d>1$,  such that for each $i\in [d]$, the definable subgroup
$H:=H_1\cap\cdots\cap H_d$ has infinite index in
$$H_{\neq i} := H_1\cap\cdots\cap H_{i-1}\cap H_{i+1}\cap\cdots \cap H_d.$$
Then  $\vc^{\Th(A)}(m)\geq d$.
\end{lemma}
\begin{proof}
Let $H_1,\dots, H_d$ be as in the hypothesis, and  let $i$ and $k$ range over $[d]$. Let $y$ be a tuple of new variables with $\abs{x}=\abs{y}$, and for each $i$ let 
$\varphi_i(x;y)$ be an $\mathcal L$-formula stating that $x-y \in H_i$, and set $\Delta:=\{\varphi_1,\dots,\varphi_d\}$. Since $[H_{\neq i}:H]$ is infinite,  we may choose elements $b_{ij}$ ($j\in\bN$) of $H_{\neq i}\setminus H$ which are pairwise inequivalent modulo~$H$. Fix  $t\in\bN$, and set
$$B := \big\{ b_{ij}: i\in [d], \ j\in [t]\big\},$$
so $\abs{B}=dt$. For each $\mathbf j=(j(1),\dots,j(d))\in [t]^d$ set
$$a_{\mathbf j} = b_{1j(1)} + \cdots + b_{dj(d)}.$$
Then for each $i$ we have
$$a_{\mathbf j}-b_{ij(i)} = \sum_{k\neq i} b_{kj(k)} \in \sum_{k\neq i} H_{\neq k} \subseteq H_i.$$
(Here we use that $A$ is abelian.)
Hence for each $i$, $k$ and  $l\in [t]$ we have 
\begin{align*}
A\models\varphi_i(a_{\mathbf j};b_{kl})	&\qquad\Longleftrightarrow\qquad a_{\mathbf j}-b_{kl} \in H_i \\
										&\qquad\Longleftrightarrow\qquad b_{ij(i)}-b_{kl}\in H_i \\
										&\qquad\Longleftrightarrow\qquad \text{$k=i$ and $l=j(i)$.}
\end{align*}
So the $a_{\mathbf j}$ realize distinct $\Delta(x;B)$-types in $A$. Thus $\pi^*_\Delta(dt)\geq\abs{S^\Delta(B)}\geq t^d$, and since $t$ was arbitrary, this yields $\vc^{\Th(A)}(m)\geq d$.
\end{proof}

\subsection{Commensurability and VC~density for modules.}
We now return to the setting of modules. Let $M$ be an infinite $R$-module and $T=\Th(M)$. 
We write $\widetilde{\PP}_m(M)$ for the lattice of commensurability classes of p.p.~definable subgroups of $M^m$. We have a natural surjective lattice morphism $\PP_m(M)\to\widetilde{\PP}_m(M)$, so $\widetilde{\PP}_m(M)$ is modular (and distributive if $\PP_m(M)$ is distributive) with
$$\breadth(\PP_m(M))\geq \breadth\big(\widetilde{\PP}_{m}(M)\big).$$ 
For all $m_1$, $m_2$ the natural embedding \eqref{eq:embedding of PPm's} shows that
$$\breadth\big(\PP_{m_1}(M)\big)+\breadth\big(\PP_{m_2}(M)\big)\leq\breadth\big(\PP_{m_1+m_2}(M)\big)$$
and
\begin{equation}\label{eq:m1m2}
\breadth\big(\widetilde{\PP}_{m_1}(M)\big)+\breadth\big(\widetilde{\PP}_{m_2}(M)\big)\leq\breadth\big(\widetilde{\PP}_{m_1+m_2}(M)\big).
\end{equation}
If $M\equiv M^{\aleph_0}$, then the morphism $\PP_m(M)\to\widetilde{\PP}_m(M)$ is bijective (see, e.g., \cite[Lemma~A.1.7]{Hodges}) and hence
$\PP_m(M)$ and $\widetilde{\PP}_{m}(M)$ have the same breadth, for each~$m$.

\begin{corollary}\label{cor:lower bounds on vc}
For each $m$,
$$m\breadth\big(\widetilde{\PP}(M)\big)\leq\breadth\big(\widetilde{\PP}_m(M)\big)=\vc^T(m)\leq m\breadth(\PP(M)),$$
with both inequalities being equalities if $M\equiv M^{\aleph_0}$.
\end{corollary}

\begin{proof}
The first inequality is a consequence of \eqref{eq:m1m2}, the equation is a special case of Corollary~\ref{cor:commensurable, 2}, and the last inequality is Proposition~\ref{prop:breadth and vc for modules}.
\end{proof}

\begin{question}
Is $\breadth\big(\widetilde{\PP}_m(M)\big)=m\breadth\big(\widetilde{\PP}(M)\big)$ for every $m$? \textup{(}In this case we would have $\vc^T(m) = m\vc^T(1)$ for each $m$, in line with the Question posed in Section~3.2 of~\textup{\cite{ADHMS}.)}
\end{question}

\begin{example*}
Let $K$ be an infinite field, $R=Ke_1\oplus\cdots\oplus Ke_d$ the ring-theoretic direct product of $d>0$ copies of $K$, with idempotents $e_1,\dots,e_d$, and $M=R$ as a module over itself. Then $M$ is the direct sum of its definable submodules 
$$Ke_i=\{a\in M:\text{$e_ja=0$ for all $j\neq i$}\} \qquad (i=1,\dots,d),$$ so 
$\PP(M)\cong 2^{[d]}$ by Lemma~\ref{lem:pp definable for direct sum} and hence $\breadth(M)=d$.
In particular, $M$ has the $\VC{}d$ property by Proposition~\ref{prop:breadth and vc for modules}. Thus $\vc^T(m)\leq dm$ for every $m$, and since $M\equiv M^{\aleph_0}$, we have in fact $\vc^T(m)=dm$ for every $m$, by the previous corollary.
\end{example*}

Since $\breadth\big(\widetilde{\PP}(M)\big)$ bounds both 
the Goldie dimension $\Gdim(\widetilde{\PP}(M))$ and the dual Goldie dimension $\Gdim^*(\widetilde{\PP}(M))$ of the lattice $\widetilde{\PP}(M)$, the previous corollary implies:

\begin{corollary}\label{cor:Gdim and vc for modules}
Let $H_1,\dots,H_d$ be p.p.~definable subgroups of $M$. If
\begin{enumerate}
\item $H_i$ is infinite and $H_i\cap\sum_{j\neq i} H_j$ is finite, for each $i$, or
\item $H_i$ has infinite index in $M$ and $H_i+\bigcap_{j\neq i} H_j$ has finite index in $M$, for each $i$, 
\end{enumerate}
then  $\vc^{T}(m)\geq md$ for every $m$.
\end{corollary}

One also verifies easily the following properties of the breadth of the lattice of p.p.~definable subgroups modulo commensurability (cf.~Section~\ref{sec:breadth and pp subgroups}):

\begin{lemma}\label{lem:basic properties of breadth, commensurability}
Let $M$, $N$ be $R$-modules. Then for each $m$:
\begin{enumerate}
\item If $N$ is p.p.~definable in $M$, then  
$$\breadth(\widetilde{\PP}_m(N)),\breadth(\widetilde{\PP}_m(M/N))\leq\breadth(\widetilde{\PP}_m(M));$$
\item if $N$ is a pure submodule of $M$, then 
$$\breadth(\widetilde{\PP}_m(N))\leq\breadth(\widetilde{\PP}_m(M));$$
\item if $I$ is a non-empty index set, then $$\breadth(\widetilde{\PP}_m(M^{(I)}))=\breadth(\widetilde{\PP}_m(M));$$
\item and
$$\breadth(\widetilde{\PP}_m(M\oplus N))\leq\breadth(\widetilde{\PP}_m(M))+\breadth(\widetilde{\PP}_m(N)),$$ with equality if
both $M$ and $N$ are p.p.~definable in $M\oplus N$, or if $N$ is finite \textup{(}in which case $\breadth(\widetilde{\PP}_m(N))=0$\textup{)}.
\end{enumerate}
\end{lemma}

In Section~\ref{sec:dp-min} we noted that in general, dp-minimality is not preserved under interpretability.
However, together with Corollary~\ref{cor:commensurable, 1} the case $m=1$ of Lemma~\ref{lem:basic properties of breadth, commensurability} immediately yields that dp-minimality is preserved under some natural constructions with modules:

\begin{corollary}\label{cor:operations on dp-minimal modules}
Let $M$ be a dp-minimal $R$-module. Then the following $R$-modules are also dp-minimal:
\begin{enumerate}
\item any p.p.~definable submodule $N$ of $M$ and its quotient $M/N$;
\item any pure submodule of $M$;
\item any power $M^{(I)}$, where $I$ is a non-empty index set; 
\item any direct sum $M\oplus N$ where $N$ is a finite $R$-module; and
\item any direct summand of $M$ of finite index.
\end{enumerate}
\end{corollary}

To finish this subsection we now relate the lattice $\widetilde{\PP}_m(M)$ to another modular lattice (of connected $\bigwedge$-definable subgroups of $M^m$) which plays an important role in the model theory of modules.
As usual we say that a subgroup $H$ of $M^m$ is (p.p.~definably) \emph{connected} if for each p.p.~definable subgroup $H'$ of $M^m$ we have $H\lesssim H'$ iff $H\leq H'$ (i.e., $[H:H\cap H']<\infty\Rightarrow H\subseteq H'$). 
In the following we assume that $M$ is $\abs{R}^+$-saturated.
We let $\PP_m^0(M)$ be the set of connected subgroups of $M^m$ which arise as the intersection of a (possibly infinite) family  of p.p.~definable subgroups of $M^m$. (In \cite{Prest} this is denoted by $\PP_0^{(m)}(M)$.) We equip $\PP_m^0(M)$ with the ordering by inclusion. One can show that then $\PP_m^0(M)$ is a modular lattice; moreover, for $H\in \PP_m(M)$ the subgroup
$$H^0 := \bigcap\{H'\in\PP_m(M): H\lesssim H'\}$$
of $H$ is an element of $\PP_m^0(M)$ (and called the connected component of $H$); see \cite[Corollary~2.3 and Lemma~2.6]{Prest}. The  map
$$\PP_m(M)\to \PP_m^0(M)\colon H\mapsto H^0,$$
which is clearly a morphism of ordered sets, is in general not a morphism of lattices; in fact, the intersection of two connected p.p.~definable subgroups of $M$ is not in general connected \cite[Example~2, p.~119]{Prest}.
Note however that two commensurable subgroups in $\PP_m(M)$ have identical connected component, so $H\mapsto H^0$ factors through the canonical surjective morphism $\PP_m(M)\to\widetilde{\PP}_m(M)$:
$$\xymatrixcolsep{3.5pc}\xymatrixrowsep{1.7pc}
\xymatrix{
\PP_m(M) \ar[d]	\ar[r]^{H\mapsto H^0} & \PP^0_m(M) \\
\widetilde{\PP}_m(M) \ar[ur]^\iota
}$$
The theory $T=\Th(M)$ is superstable iff $\PP^0(M)=\PP_1^0(M)$ is well-founded, and in this case $\URk(T)$ is the foundation rank of $\PP^0(M)$ \cite[Corollary~5.13]{Prest}; moreover:

\begin{proposition}\label{prop:foundation rank}
Suppose $T$ is superstable. Then $\widetilde{\PP}_m(M)$ has the same foundation rank $m\URk(T)$ as $\PP_m^0(M)$.
\end{proposition}

If $T$ is superstable we can employ the Shelah degree $R(-)$ (the ordinal-valued rank function denoted by $R^m(-,\mathcal L_R,\infty)$ in \cite{Shelah-book}), whose basic properties in this context we use freely (cf.~\cite[Section~5.2]{Prest}). 
In particular we recall that for $G\in\PP^0_m(M)$, the rank $R(G)$ is the foundation rank of $G$ in $\PP^0_m(M)$ \cite[Lemma~5.10, Theorems~5.12 and 5.18]{Prest}.
The proposition is a consequence of the following two lemmas, in both of which we assume that $T$ is superstable.
The first lemma is immediate from \cite[Lemma~5.10 and Corollary~5.15]{Prest}.

\begin{lemma}\label{lem:foundation rank, 1}
Let $G\in\PP^0_m(M)$. Then there is some $K\in\PP_m(M)$ with $G\leq K$ and $R(G)=R(K)$. 
\end{lemma}

\begin{lemma}\label{lem:foundation rank, 2}
The map $\iota$ is strictly increasing, i.e., for all $H,H'\in\PP_m(M)$ we have
$H\lnsim H' \Rightarrow H^0 \lneq (H')^0$.
\end{lemma}
\begin{proof}
Let $H,H'\in\PP_m(M)$; it is clear that $H\lesssim H' \Rightarrow H^0\leq (H')^0$ (and this implication doesn't need superstability). So suppose for a contradiction that $H\lnsim H'$ and $H^0 = (H')^0$. Since $H\sim H'\cap H$, after replacing $H$ by $H'\cap H$ we may assume $H\leq H'$.
Then $H$ has infinite index in $H'$, so $R(H)<R(H')$. Using Lemma~\ref{lem:foundation rank, 1}, pick $K'\in\PP_m(M)$ with $H'\lesssim K'$ and $R(K')=R((H')^0)$.
Then 
$$R(H)<R(H')=R(H'\cap K')\leq R(K')=R((H')^0)=R(H^0)\leq R(H),$$ 
a contradiction.
\end{proof}

\begin{proof}[Proof of Proposition~\ref{prop:foundation rank}] 
By Lemma~\ref{lem:foundation rank, 2}, the foundation rank of $\widetilde{\PP}_m(M)$ is not larger than that of $\PP_m^0(M)$. For the converse we show, by transfinite induction on the ordinal $\alpha$, that if $G\in\PP_m^0(M)$ satisfies $R(G)\geq\alpha$, then the foundation rank of each $K\in\PP_m(M)$ with $K\geq G$ and $R(G)=R(K)$, viewed as element of $\widetilde{\PP}_m(M)$, is no less than $\alpha$. Since the case $\alpha=0$ or $\alpha$ a limit ordinal are trivial, we only need to treat the case where $\alpha=\beta+1$ is a successor ordinal. So let $G\in\PP_m^0(M)$ with $R(G)\geq\alpha$ and let $K\in\PP_m(M)$ with $K\geq G$ and $R(G)=R(K)$. Take $G'\in\PP_m^0(M)$ with $G'\leq G$ and $R(G')=\beta$, and by Lemma~\ref{lem:foundation rank, 1} take $K'\in\PP_m(M)$ with $G'\leq K'$ and $R(G')=R(K')$; after replacing $K'$ by $K\cap K'$, we may assume $K'\leq K$. Now $R(K')<R(K)$ implies $[K:K']=\infty$, hence the foundation rank of $K$ in $\widetilde{\PP}_m(M)$ is strictly larger than that of $K'$, and by inductive hypothesis the latter is greater than or equal to $\beta$.
\end{proof}

\begin{example}\label{ex:Urank 1}
The module $M$ has $\URk$-rank $0$ iff $M$ is finite, and $M$ has $\URk$-rank $1$ iff for each $H\in\PP(M)$ we either have $H\sim 0$ or $H\sim M$.
\end{example}

\begin{corollary}\label{cor:finite U-rank for modules}
Suppose $T$ has finite $\URk$-rank $d$. Then $\vc^T(m)\leq dm$ for every $m$.
\end{corollary}

\begin{proof}
By Proposition~\ref{prop:foundation rank}, $\widetilde{\PP}_m(M)$ has height $dm$ and hence breadth at most~$dm$.
Thus by Proposition~\ref{prop:commensurable} we get $\vc^T(m)\leq dm$.
\end{proof}

\begin{remark*}
Every complete theory of a module is dimensional \cite[Corollary~6.21]{Prest} and hence does not have the finite cover property (see also \cite{Baudisch}); hence Corollary~\ref{cor:mainst, finite U-rank} yields another proof of Corollary~\ref{cor:finite U-rank for modules}.
\end{remark*}

The relationship between the ordered sets $\widetilde{\PP}_m(M)$ and $\PP_m^0(M)$ is particularly clean if $T$ is totally transcendental:

\begin{lemma}\label{lem:tt}
If $T$ is totally transcendental, then $\iota$ is an isomorphism of ordered sets.
\end{lemma}
\begin{proof}
In this case, $\PP_m(M)$ satisfies the descending chain condition \cite[Theorem~3.1]{Prest}; it follows that $H\sim H^0$ for all $H\in\PP_m(M)$ (showing that $\iota$ is an embedding of ordered sets) and $\iota$ is onto (using \cite[Lemma~5.10]{Prest}).
\end{proof}

\subsection{Direct sums of modules.}\label{sec:direct sum}
For use in the next section we establish two lemmas concerning the behavior of VC~density under the operation of taking the direct sum of two modules. (We also prove a corresponding fact for  the $\VC{}d$ property, which, however, will not be used later.)
We let $M$, $M'$ be  $R$-modules, $T=\Th(M)$, $T'=\Th(M')$, and $T^+=\Th(M\oplus M')$. 
We first observe that if $M$ is infinite and a pure submodule of $M'$, then $\vc^{T}\leq\vc^{T'}$.
(This is an immediate consequence of the Baur-Monk Theorem and the fact that
every Boolean combination of p.p.~$\mathcal L_R$-formulas is invariant under the extension $M'\subseteq M$.) This yields:

\begin{lemma}\label{lem:VC density under adding a finite module}
Suppose $M$ is infinite and $M'$ is finite. Then $\vc^{T^+}=\vc^T$.
\end{lemma}
\begin{proof}
By the remark preceding the lemma, applied to $M\oplus M'$ in place of $M'$, we obtain $\vc^{T^+}\geq \vc^T$. The other inequality $\vc^{T^+}\leq\vc^T$ follows from  the remark after Lemma~3.21 in \cite{ADHMS}.
\end{proof}

If both $M$ and $M'$ are infinite, by \cite[Lemma~3.21]{ADHMS} we can only conclude that
$$\max\big(\vc^{T},\vc^{T'}\big)\leq \vc^{T^+} \leq  \vc^T + \vc^{T'}.$$
{\it The standing assumption in both lemmas below is that $M$ and $M'$ are definable in $M\oplus M'$.}\/ Then the natural projections $\pi\colon M\oplus M'\to M$ and $\pi'\colon M\oplus M'\to M'$ are also definable in $M\oplus M'$. In this situation, the estimate on $\vc^{T^+}$ above can be improved:

\begin{lemma}\label{lem:VC density of sum of modules} 
Suppose $M$ and $M'$ are infinite.  Then
$\vc^{T^+} = \vc^T + \vc^{T'}$.
\end{lemma}
\begin{proof}
By Lemma~\ref{lem:expansions by constants} we may assume that $M$, $M'$ are $\emptyset$-definable in $M\oplus M'$.
We only need to show $\vc^{T^+} \geq \vc^T + \vc^{T'}$.
Let $\Delta(x;y)$, $\Delta'(x;y)$ be finite sets of p.p.~$\mathcal L_R$-formulas.
Set (slightly abusing syntax)
$$\Delta_1:=\big\{\varphi(\pi(x);y):\varphi(x;y)\in\Delta\big\}, \qquad
  \Delta_1':=\big\{\varphi'(\pi'(x);y):\varphi'(x;y)\in\Delta'\big\}.$$
Let $B\subseteq M^{\abs{y}}$, $B'\subseteq (M')^{\abs{y}}$ with $\abs{B}=\abs{B'}=t\in\bN$, and put 
$$C=B\cup B'\subseteq (M\oplus M')^{\abs{y}},$$ 
so $\abs{C}\leq 2t$. 
Given a pair $(q,q')\in S^\Delta(B)\times S^{\Delta'}(B')$ we let $p$ be the $(\Delta_1\cup\Delta_1')(x;C)$-type $\tp^{\Delta_1\cup\Delta_1'}(a_q + a_{q'}/C)$ of $a_q + a_{q'}\in (M\oplus M')^{\abs{x}}$ in $M\oplus M'$, where $a_q$, $a_{q'}$ are arbitrary realizations of $q$ in $M$ and of $q'$ in $M'$, respectively.
Then for each $\varphi\in\Delta$ and $b\in B$,
\begin{align*}
\varphi(x;b)\in q	&\quad\Longleftrightarrow\quad M\models\varphi(a_q;b) \\
					&\quad\Longleftrightarrow\quad M\oplus M'\models\varphi(a_q;b) \\
					&\quad\Longleftrightarrow\quad M\oplus M'\models\varphi(\pi(a_q + a_{q'});b) \\
					&\quad\Longleftrightarrow\quad \varphi(\pi(x);b)\in p,
\end{align*}
where in the second equivalence we used that $\varphi$ is p.p.
Similarly we see that $\varphi'(x;b')\in q'$ iff $\varphi'(\pi'(x);b')\in p$, for each $\varphi'(x;y)\in\Delta'$ and $b'\in B'$. Hence the map 
$$S^\Delta(B)\times S^{\Delta'}(B')\to S^{\Delta_1\cup\Delta_1'}(C)$$ 
given by $(q,q')\mapsto p$ is injective, so
$\pi^*_\Delta(t)\cdot\pi^*_{\Delta'}(t)\leq \pi^*_{\Delta_1\cup\Delta_1'}(2t)$. 
Here $\pi^*_\Delta$ and $\pi^*_{\Delta'}$ are computed in $M$ and $M'$, respectively, and
$\pi^*_{\Delta_1\cup\Delta_1'}$ is computed in $M\oplus M'$.
Since $t$ was arbitrary, this yields
$\vc^*(\Delta)+\vc^*(\Delta')\leq\vc^*(\Delta_1\cup\Delta_1')$, employing a similar convention for $\vc^*$ as for $\pi^*$.
By Lemma~\ref{lem:encoding finite sets of formulas} (and the Baur-Monk Theorem) we obtain $\vc^T(m)+\vc^{T'}(m)\leq\vc^{T^+}(m)$, where $m=\abs{x}$, as required.
\end{proof}

\begin{lemma}
Suppose $M$ and $M'$ are $\emptyset$-definable in $M\oplus M'$.
If $T$ has the  $\VC{}d$ property and $T'$ has the  $\VC{}d'$ property, then $T^+$ has the  $\VC{}(d+d')$ property. 
\end{lemma}

\begin{proof}
Suppose $T$ has the  $\VC{}d$ property and $T'$ has the  $\VC{}d'$ property.
Let $\Delta(x;y)$, where $\abs{x}=1$, be a finite set of p.p.~$\mathcal L_R$-formulas; by Lemma~\ref{lem:qe and VCd} and the Baur-Monk~Theorem it suffices to show that $\Delta$ has UDTFS with $d+d'$ parameters in $M\oplus M'$.
Let $\mathcal D=(\mathcal D_i)_{i\in I}$ and $\mathcal D'=(\mathcal D_{i'})_{i'\in I'}$ be uniform definitions of $\Delta(x;B)$-types over finite sets in $M$ and in $M'$, respectively, where
$$\mathcal D_i := \big\{ \operatorname{d}_{\varphi,i}(y;\overline{y}) :\varphi\in\Delta \big\}, \qquad
\mathcal D'_{i'} := \big\{ \operatorname{d}'_{\varphi,i'}(y;\overline{y'}) :\varphi\in\Delta \big\}
$$
with
$\overline{y}=(y_1,\dots,y_d)$ and $\overline{y'}=(y'_1,\dots,y'_{d'})$. By Baur-Monk we may assume that each $\operatorname{d}_{\varphi,i}$ and $\operatorname{d}'_{\varphi,i'}$ is a Boolean combination of p.p.~$\mathcal L_R$-formulas.

Let $C$ be a finite set of tuples from $(M\oplus M')^{\abs{y}}$. Take $B\subseteq M^{\abs{y}}$, $B'\subseteq (M')^{\abs{y}}$ minimal such that each $c\in C$ is of the form $c=b + b'$ for a pair $(b,b')\in B\times B'$. Let $a\in M$, $a'\in M'$ and $p=\tp^\Delta(a+a'/C)$. Put
$$q:=\tp^\Delta(a/B), \qquad q':=\tp^\Delta(a'/B'),$$
and take $\overline{b}=(b_1,\dots,b_d)\in B^d$, $\overline{b'}=(b_1',\dots,b_{d'}')\in (B')^{d'}$ and $i\in I$, $i'\in I'$ such that $\mathcal D_i(y;\overline{b})$ defines $q$ and $\mathcal D'_{i'}(y;\overline{b'})$ defines $q'$.
Then for all $(b,b')\in B\times B'$ we have
\begin{align*}
\varphi(x;b+ b')\in p	&\quad \Longleftrightarrow \quad M\oplus M'\models\varphi(a+a';b+ b') \\
							&\quad \Longleftrightarrow \quad M\models\varphi(a;b) \text{ and } M'\models\varphi(a';b')\\
							&\quad \Longleftrightarrow \quad M\models\operatorname{d}_{\varphi,i}(b;\overline{b})\text{ and } M'\models\operatorname{d}'_{\varphi,i'}(b';\overline{b'}),
\end{align*}
where in the second equivalence we used that $\varphi$ is p.p. 
Take $\overline{c}\in C^d$ and $\overline{c'}\in C^{d'}$ such that $\pi(\overline{c})=\overline{b}$ and $\pi'(\overline{c'})=\overline{b'}$.
Since $\operatorname{d}_{\varphi,i}$ and $\operatorname{d}'_{\varphi,i'}$ are Boolean combinations of p.p.~formulas, we have
$$M\models\operatorname{d}_{\varphi,i}(b;\overline{b})	\quad \Longleftrightarrow \quad 
M\oplus M'\models \operatorname{d}_{\varphi,i}(\pi(b+ b');\pi(\overline{c}))$$
and
$$M'\models\operatorname{d}'_{\varphi,i'}(b';\overline{b'})	\quad \Longleftrightarrow \quad 
M\oplus M'\models\operatorname{d}'_{\varphi,i'}(\pi'(b+ b');\pi'(\overline{c'})).$$
Hence $\mathcal D^+=(\mathcal D^+_{(i,i')})_{(i,i')\in I\times I'}$, where
$\mathcal D^+_{(i,i')} = \{\operatorname{d}^+_{\varphi,(i,i')}:\varphi\in\Delta\}$ with
$$\operatorname{d}^+_{\varphi,(i,i')}(y;\overline{y},\overline{y'}) := \operatorname{d}_{\varphi,i}(\pi(y);\pi(\overline{y})) \wedge \operatorname{d}'_{\varphi,i'}(\pi'(y);\pi'(\overline{y'})),$$
is a uniform definition of $\Delta(x;B)$-types over finite sets in $M\oplus M'$ with $d+d'$ parameters, as required.
\end{proof}

\section{Abelian Groups with Uniformly Bounded VC~Density}\label{sec:abelian}

\noindent
In this final section of the paper we focus on abelian groups. We write abelian groups additively and construe them as first-order structures in the language $\mathcal L_{\bZ}$ of $\bZ$-modules as usual. 
Throughout this section $A$ denotes an  abelian group and $T=\Th(A)$ its complete theory.
Determining the VC~density function of an arbitrary (infinite) abelian group is an interesting but probably intricate problem. In the first part of this section we obtain a satisfactory answer for those abelian groups whose lattice of p.p.~definable subgroups is finite; in particular, we prove Theorem~\ref{thm:aleph0-cat abelian} from the introduction.
After a preliminary subsection (Section~\ref{sec:homocyclic}) we then prove Theorem~\ref{thm:abelian} in Section~\ref{sec:abelian with uniform VC}, and in Section~\ref{sec:dp-minimal abelian} we characterize all dp-minimal abelian groups.
We finish with some remarks on (non-) dp-minimal expansions of the dp-minimal abelian group $\mathbb Z$ (Section~\ref{sec:dp-min expansions of Z}).

\subsection{A case study: abelian groups of finite exponent.}\label{sec:finite exponent}
For each prime~$p$ let 
$$U(p;A):=\big\{i\geq 0:U(p,i;A)\neq 1\big\}.$$ 
Then $A$ has finite exponent iff $A$ is torsion and each set $U(p;A)$ is finite with $U(p;A)=\emptyset$ for all but finitely many $p$.
We will show:

\begin{theorem}\label{thm:breadth of finite exponent group}
Suppose $A$ has finite exponent, and set
\begin{equation}\label{eq:d}
d:=\sum_p d(U(p;A)).
\end{equation}
Then $A$ has breadth $d$.
\end{theorem}

Here, given a finite non-empty set $I=\{i_1,\dots,i_n\}$ of integers $i_1<\cdots<i_n$, we define $d=d(I)$ as the maximal length of a sequence $1=j(1)<j(2)<\cdots<j(d)\leq n$ of indices such that for $k=1,\dots,d-1$,
$$j(k+1)-j(k)=\begin{cases} 
1&\text{ if $i_{j(k)+1}-i_{j(k)}>1$}, \\
2&\text{ otherwise.} 
\end{cases}$$
We also write $d(i_1,\dots,i_n)$ instead of $d(I)$, and set $d(\emptyset):=0$.

\begin{examples}\mbox{} \label{ex:d(I)}

\begin{enumerate}
\item For $I=\{2,3,5,7,8,9\}$ the sequence of $j(k)$'s is $1$, $3$, $4$, $6$, so $d(I)=4$.
\item For $I=\{1,2,3,\dots,n\}$ the sequence of $j(k)$'s is $1,3,5,\dots,n$ if $n$ is odd, and $1,3,5,\dots,n-1$ if $n$ is even, hence $d(I)=\lceil n/2\rceil$.
\end{enumerate}
\end{examples}

Clearly we have $d(I)\geq \lceil n/2\rceil$ for every finite set $I$ of $n$ positive integers.
We also have the following upper bound on $d$ (which, however, is strict in general):

\begin{lemma}\label{lem:upper bound on d(I)}
Let $I=\{i_1,\dots,i_n\}$  be as above, with $i_1>0$, and set $i_0:=0$. Then
$$d(I)\leq\displaystyle\min_{0\leq j\leq n} \big(n-j+\lceil i_j/2\rceil\big).$$
\end{lemma}
\begin{proof}
Set $d_k:=d(i_1,\dots,i_k)$ for $k\in [n]$, and let 
$$1=j(1)<j(2)<\dots<j(d_n)\leq n$$ be a sequence of indices as above; note that necessarily $j(d_k)=k$ or $j(d_k)=k-1$, and $d_k=d_{k-1}$ or $d_k=d_{k-1}+1$. We proceed by induction to show that for each $k\in [n]$, 
\begin{equation}\label{eq:inequ on dk}
d_k\leq\displaystyle\min_{0\leq j\leq k} \big(k-j+\lceil i_j/2\rceil\big).
\end{equation} 
For $k=1$ this is trivial, since $d_1=1$ and the right-hand side in this inequality is $\min\{1,\lceil i_1/2\rceil\}=1$. Suppose we have shown \eqref{eq:inequ on dk} for some value of $k<n$; then 
$$d_{k+1}\leq d_k+1\leq (k+1)-j+\lceil i_j/2\rceil\qquad\text{for $j=0,\dots,k$,}$$
so we only need to show that $d_{k+1}\leq \lceil i_{k+1}/2\rceil$. For this, we distinguish several cases.
Suppose first that $j(d_k)=k$ and $i_{k+1}-i_k>1$. Then $d_{k+1}=d_k+1$ (and $j(d_{k+1})=j(d_k)+1$), and
$\lceil i_{k+1}/2\rceil - \lceil i_k/2\rceil \geq 1$, so 
$$d_{k+1}\leq d_k+1\leq \lceil i_k/2\rceil+1\leq \lceil i_{k+1}/2\rceil.$$ 
If $j(d_k)=k$ and $i_{k+1}-i_k=1$, then clearly 
$$d_{k+1}=d_k\leq  \lceil i_k/2\rceil\leq \lceil i_{k+1}/2\rceil.$$
Finally, suppose $j(d_k)=k-1$; then necessarily $k>1$ and $i_k-i_{k-1}=1$, $d_k=d_{k-1}$. In this case we have $d_{k+1}=d_k+1$, $j(d_{k+1})=k+1$, and clearly $i_{k+1}-i_{k-1}\geq 2$, so 
$$d_{k+1}=d_{k-1}+1\leq \lceil i_{k-1}/2\rceil+1\leq \lceil i_{k+1}/2\rceil$$ 
as required.
\end{proof}

In particular we have
\begin{equation}\label{eq:bound on d(I)}
\lceil n/2\rceil \leq d(i_1,\dots,i_n) \leq \min\{ n, \lceil i_n/2 \rceil \}.
\end{equation}
For later use  also note the following easily verified observation:
\begin{equation}\label{eq:d(I)=1}
d(i_1,\dots,i_n)=1\quad\Longleftrightarrow\quad \text{$n=1$, or $n=2$ and $i_2=i_1+1$.}
\end{equation}

\subsubsection{Towards the proof of Theorem~\ref{thm:breadth of finite exponent group}.}
Recall that every abelian torsion group is an internal direct sum of its $p$-primary components, as $p$ ranges over the set of prime numbers.
We denote the $p$-primary component of $A$ by $A_p\leq A$; if $A_p$ has finite exponent, then $A_p$ is obviously
p.p.~definable in $A$.  We have
$$A_p \cong \left(\bigoplus_{i>0} \bZ(p^i)^{(\alpha_{p,i-1})}\right)\oplus  \bZ(p^\infty)^{(\beta_p)}$$
where $\alpha_{p,i}$ and $\beta_p$ are cardinals, with $U(p,i;A)=p^{\alpha_{p,i}}$. (Here and below, $\bZ(p^i)$ denotes the cyclic group of order $p^i$, and $\bZ(p^\infty)$ the Pr\"ufer $p$-group.)

Suppose now that $A$ has finite exponent. Then by Lemma~\ref{lem:basic properties of breadth},~(4) we have
$$\breadth(A) = \sum_{p} \breadth(A_p),$$ 
and in the representation of $A_p$ above we have $\alpha_{p,i}=0$ for all but finitely many pairs $(p,i)$, and $\beta_p=0$ for each $p$.
So for the proof of Theorem~\ref{thm:breadth of finite exponent group} we may assume that for some prime $p$, our group $A$ has the following form:
$$A=\bZ(p)^{(\alpha_0)}\oplus\bZ(p^2)^{(\alpha_1)}\oplus\cdots\oplus \bZ(p^m)^{(\alpha_{m-1})}\qquad\text{where the $\alpha_i$ are cardinals.}$$
Letting 
$$B=\bZ(p)^{(\beta_0)}\oplus\bZ(p^2)^{(\beta_1)}\oplus\cdots\oplus \bZ(p^m)^{(\beta_{m-1})}\qquad\text{where
$\beta_i=\begin{cases}
0 & \text{if $\alpha_i=0$} \\
1 & \text{otherwise,}
\end{cases}
$
}$$
for an arbitrary infinite cardinal $\kappa\geq\max_i \alpha_i$ we have $A^{(\kappa)}\cong B^{(\kappa)}$ and thus
$$\breadth(A)=\breadth(A^{(\kappa)})=\breadth(B^{(\kappa)})=\breadth(B)$$ 
by part~(3) of Lemma~\ref{lem:basic properties of breadth}. Also $U(p;A)=U(p;B)$. Thus, after replacing $A$ by $B$ we may further assume that for each $i$, either $\alpha_i=0$ or $\alpha_i=1$, so
$$A=\bZ(p^{i_1})\oplus\bZ(p^{i_2})\oplus\cdots\oplus\bZ(p^{i_n})$$
where $i_1<\dots<i_n$ are the indices $i\in [m]$ such that $\alpha_{i-1}\neq 0$.
Hence (switching notation from $i_k$ to $\lambda_k$) we are done once we have shown the following:

\begin{proposition} \label{prop:breadth of finite exponent group}
Suppose
$$A=\bZ(p^{\lambda_1})\oplus\bZ(p^{\lambda_2})\oplus\cdots\oplus\bZ(p^{\lambda_n})$$
is a finite abelian $p$-group of type $\lambda=(\lambda_1,\dots,\lambda_n)$, where $0<\lambda_1<\cdots<\lambda_n$.
Then $\breadth(A)=d(\lambda)$.
\end{proposition}

From now on assume that $A$ is as in this proposition. We first give an explicit description of the lattice of p.p.~definable subgroups of $A$.

\subsubsection{The lattice of p.p.~definable subgroups of a finite abelian $p$-group.}
The set $\bN^n$, equipped with the product ordering, is a distributive lattice. We say that an $n$-tuple $\mu=(\mu_1,\dots,\mu_n)\in\bN^n$ is \emph{slowly growing} (relative to $\lambda$) if 
$$0\leq \mu_{i}-\mu_{i-1}\leq\lambda_{i}-\lambda_{i-1}\qquad\text{for $i\in [n]$, where we set $\mu_0:=\lambda_0:=0$.}$$ 
Note that if $\mu$ is slowly growing, then $\mu_i\leq\lambda_i$ for each $i\in [n]$, and 
$\mu_j-\mu_i\leq\lambda_j-\lambda_i$ for $i,j\in [n]$ with $i<j$.
One  verifies easily that the set $\Lambda$ of slowly growing tuples is a (finite) sublattice of $\bN^n$ with smallest element $0$ and largest element $\lambda$.
If $\mu$ is slowly growing, then so is $\mu^*:=\lambda-\mu$, and the map $\mu\mapsto\mu^*$ is an order-reversing involution of $\Lambda$, so $\Lambda^*\cong\Lambda$.

For each integer $k\geq 0$ consider the $n$-tuple $\mu=(\mu_1,\dots,\mu_n)\in\bN^n$, given by
$\mu_i=\lambda_i$ if $\lambda_i\leq k$ and $\mu_i=k$ otherwise; we denote this $n$-tuple $\mu$ by $\langle k\rangle$. The map $k\mapsto\langle k\rangle$ is an embedding of the chain $K:=\{0,\dots,\lambda_n\}$ into $\Lambda$.

\begin{lemma}
The lattice $\Lambda$ is generated by the elements $\langle k\rangle$, $\langle k\rangle^*$ where $k$ ranges over~$K$.
\end{lemma}
\begin{proof}
Let $\Lambda'$ be the sublattice of $\Lambda$ generated by the elements $\langle k\rangle$, $\langle k\rangle^*$ where $k\in K$. 
Let an arbitrary $\mu\in\Lambda$ be given. For each $i\in [n]$ set
$$\mu_{\leq i} = (\mu_1,\ \mu_2,\ \dots,\ \mu_i,\ \mu_i,\ \dots,\ \mu_i)\in\Lambda.$$
We show by induction on $i$ that $\mu_{\leq i}\in\Lambda'$ for each $i\in [n]$.
For $i=1$ this is clear since $\mu_{\leq 1}=(\mu_1,\dots,\mu_1)=\langle \mu_1\rangle$, and for the inductive step notice that if $i\in [n-1]$, then
$$\mu_{\leq i+1} = \langle\mu_{i+1}\rangle \wedge \big(\mu_{\leq i}\vee \langle\lambda_{i+1}-\mu_{i+1}\rangle^*\big).$$
In particular $\mu=\mu_{\leq n}\in\Lambda'$. 
\end{proof}

We write the elements of $A$ as $n$-tuples $a=(a_1,\dots,a_n)$ where each $a_i$ is an element of $\bZ(p^{\lambda_i})$. 
For each $\mu\in\bN^n$ now define the subgroup
$$A_\mu := \big\{ a=(a_1,\dots,a_n)\in A: p^{\mu_i}a_i=0 \text{ for each $i\in [n]$} \big\}$$
of $A$. (So, e.g., $A_0=\{0\}$ and $A_\lambda=A$.)
It is clear that $A_{\mu\wedge\nu}=A_\mu\cap A_\nu$ and $A_{\mu\vee\nu}=A_\mu+A_\nu$ for all $\mu,\nu\in\bN^n$, i.e., $\mu\mapsto A_\mu$ is a  morphism of lattices.

\begin{lemma}
The map $\mu\mapsto A_\mu$ restricts to a lattice isomorphism $\Lambda\to\PP(A)$.
\end{lemma}
\begin{proof}
For each $k\in K$ the subgroup $A_{\langle k\rangle}$ of $A$ is defined by the p.p.~formula $p^k x=0$, 
and $A_{\langle k\rangle^*}$ is defined by $p^k|x$. Hence by the previous lemma, the image of the restriction of $\mu\mapsto A_\mu$ to $\Lambda$ is contained in $\PP(A)$.
Conversely, it is well-known (and easy to see) that the lattice $\PP_R$, where $R$ is the ring $\bZ/p^{\lambda_n}\bZ={\bZ(p^{\lambda_n})}$, is generated by the p.p.~formulas having the form $p^k x=0$ or $p^k|x$ where $k\in K$. Thus our morphism maps onto $\PP(A)$, and it is also clearly one-to-one.
\end{proof}

\subsubsection{The ordered set of join-irreducibles.}
Now that we have identified $\PP(A)$ as~$\Lambda$, we turn to giving an explicit description of the ordered set $J(\Lambda)$ of join-irreducibles of the distributive lattice $\Lambda$. 
(We are interested in such a description since $\breadth(L)=\width(J(L))$, as we recall from Proposition~\ref{prop:dilworth}.)
 Consider
$$P := \big\{ (i,j) : j\in [n],\ i\in [\lambda_j] \big\},$$
equipped with the ordering given by
$$(i,j)\leq (i',j') \qquad :\Longleftrightarrow \qquad i\leq i' \text{ and } \lambda_j-i \geq \lambda_{j'}-i'.$$
See Figure~\ref{fig:example of P} for an example. Note that $\abs{P}=\lambda_1+\cdots+\lambda_n$, and $P$ has smallest element $(1,n)$ and largest element $(\lambda_n,n)$.

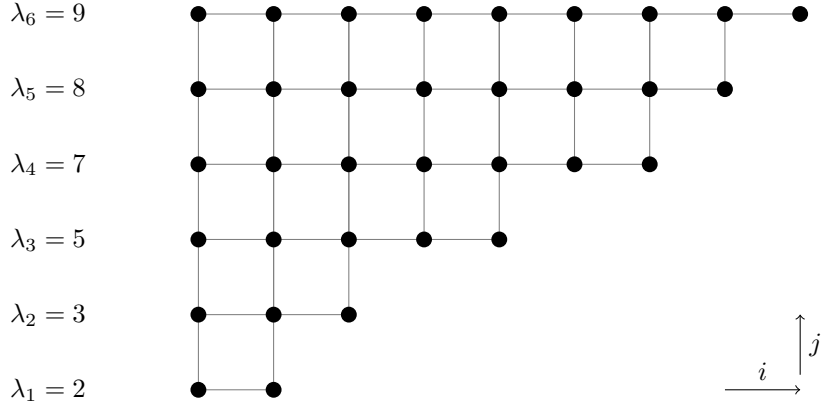
\begin{figure}

\begin{tikzpicture}

\draw[step=1,gray, thin] (1,1) grid (2,6);
\draw[step=1,gray, thin] (2,2) grid (3,6);
\draw[step=1,gray, thin] (3,3) grid (5,6);
\draw[step=1,gray, thin] (5,4) grid (7,6);
\draw[step=1,gray, thin] (7,5) grid (8,6);
\draw[gray,thin] (8,6) to (9,6);

\fill [black] (1,1) circle (3pt);
\fill [black] (2,1) circle (3pt);
\path (-1, 1) node {$\lambda_1=2$};

\fill [black] (1,2) circle (3pt);
\fill [black] (2,2) circle (3pt);
\fill [black] (3,2) circle (3pt);
\path (-1, 2) node {$\lambda_2=3$};

\fill [black] (1,3) circle (3pt);
\fill [black] (2,3) circle (3pt);
\fill [black] (3,3) circle (3pt);
\fill [black] (4,3) circle (3pt);
\fill [black] (5,3) circle (3pt);
\path (-1, 3) node {$\lambda_3=5$};

\fill [black] (1,4) circle (3pt);
\fill [black] (2,4) circle (3pt);
\fill [black] (3,4) circle (3pt);
\fill [black] (4,4) circle (3pt);
\fill [black] (5,4) circle (3pt);
\fill [black] (6,4) circle (3pt);
\fill [black] (7,4) circle (3pt);
\path (-1, 4) node {$\lambda_4=7$};

\fill [black] (1,5) circle (3pt);
\fill [black] (2,5) circle (3pt);
\fill [black] (3,5) circle (3pt);
\fill [black] (4,5) circle (3pt);
\fill [black] (5,5) circle (3pt);
\fill [black] (6,5) circle (3pt);
\fill [black] (7,5) circle (3pt);
\fill [black] (8,5) circle (3pt);
\path (-1, 5) node {$\lambda_5=8$};

\fill [black] (1,6) circle (3pt);
\fill [black] (2,6) circle (3pt);
\fill [black] (3,6) circle (3pt);
\fill [black] (4,6) circle (3pt);
\fill [black] (5,6) circle (3pt);
\fill [black] (6,6) circle (3pt);
\fill [black] (7,6) circle (3pt);
\fill [black] (8,6) circle (3pt);
\fill [black] (9,6) circle (3pt);
\path (-1, 6) node {$\lambda_6=9$};

\draw[->] (8,1) to node[above] {$i$} (9,1); 
\draw[->] (9,1.2) to node[right] {$j$} (9,2);

\end{tikzpicture}
\caption{An example for (the underlying set of) the ordered set $P$}
 \label{fig:example of P}

\end{figure}

For each $(i,j)\in P$ we define $\mu=\mu(i,j)\in\bN^n$ by
$$\mu_k := \begin{cases}
\lambda_k \dot- (\lambda_j-i) & \text{if $k<j$,} \\
i & \text{if $k\geq j$.} 
\end{cases}$$
Here $a\dot- b=\max\{a-b,0\}$  for integers $a$, $b$. It is easy to verify that for each $(i,j)\in P$, the tuple $\mu(i,j)$ is slowly growing. Figure~\ref{fig:mu(i,j)} illustrates $\mu(4,4)$ in the example from Figure~\ref{fig:example of P}.

\begin{lemma}\label{lem:join-irred, 0}
The map $(i,j)\mapsto \mu(i,j)\colon P\to\Lambda$ is an embedding of ordered sets.
\end{lemma}
\begin{proof}
Let $(i,j),(i',j')\in P$; we need to show that $(i,j)\leq (i',j')$ if and only if $\mu(i,j)\leq\mu(i',j')$.
Write  $\mu(i,j)=(\mu_1,\dots,\mu_n)$ and $\mu(i',j')=(\mu_1',\dots,\mu_n')$. It is convenient to distinguish two cases. 
First assume that $j\geq j'$. Then we have $(i,j)\leq (i',j')$ iff $i\leq i'$. If $\mu(i,j)\leq\mu(i',j')$, then $i=\mu_j\leq\mu'_{j'}=i'$. Conversely, suppose $i\leq i'$. Then for each $k\in [n]$, we have
$$\begin{cases}
\mu_k=i\leq i'=\mu'_k & \text{if $k\geq j$,} \\
\mu_k=\lambda_k \dot- (\lambda_j-i) < i \leq i'=\mu'_k & \text{if $j'\leq k<j$,} \\
\mu_k=\lambda_k \dot- (\lambda_j-i) \leq \lambda_k \dot- (\lambda_{j'}-i')=\mu'_k & \text{if $k<j'$.}
\end{cases}$$
This shows $\mu(i,j)\leq\mu(i',j')$.
On the other hand, if we assume that $j<j'$, then $(i,j)\leq (i',j')$ iff $\lambda_j-i\geq\lambda_{j'}-i'$, and in a similar way as in the previous case one sees that this is equivalent to $\mu(i,j)\leq\mu(i',j')$. 
\end{proof}

\begin{figure}

\begin{tikzpicture}[scale=0.8]
\draw[step=1,gray,thin] (1,1) grid (6,10);


\draw[very thick, dotted] (1,3) -- (2,4);
\draw[very thick, dotted] (2,4) -- (3,6);
\draw[very thick, dotted] (3,6) -- (4,8);
\draw[very thick, dotted] (4,8) -- (5,9);
\draw[very thick, dotted] (5,9) -- (6,10);

\draw[very thick] (1,1) -- (2,1);
\draw[very thick] (2,1) -- (3,2);
\draw[very thick] (3,2) -- (4,5);
\draw[very thick] (4,5) -- (6,5);

\path (1, 0.5) node  {$2$};
\path (2, 0.5) node  {$3$};
\path (3, 0.5) node  {$5$};
\path (4, 0.5) node  {$7$};
\path (5, 0.5) node  {$8$};
\path (6, 0.5) node  {$9$};

\path (0, 0.5) node  {$\lambda_k$};

\end{tikzpicture}
\caption{An example for $\mu(4,4)$}
 \label{fig:mu(i,j)}

\end{figure}
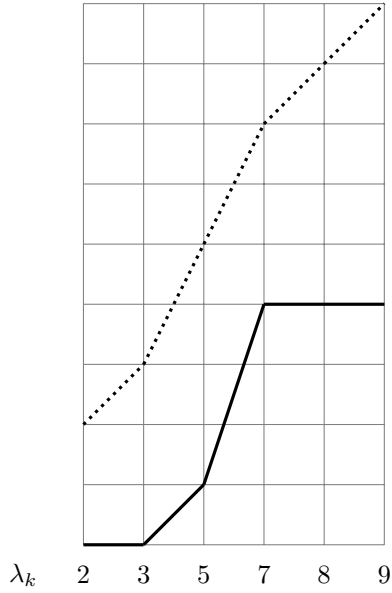

\begin{lemma}\label{lem:join-irred, 1}
Let  $\mu=\mu(i,j)$ where $(i,j)\in P$, and $\nu\in\Lambda$ with $\nu<\mu$. Then~$\nu_j<\mu_j$. 
\end{lemma}
\begin{proof}
Suppose for a contradiction that $\nu_j=\mu_j$. Then $\nu_k=\mu_k=i$ for all $k\geq j$,
so if we let $k_1\in [n]$ be the largest index such that $\nu_{k_1}<\mu_{k_1}$, then $k_1<j$.
Let also $k_0\in [n]$ be minimal such that $\lambda_{k_0}>\lambda_j-i$; note that $k_0\leq j$. 
Since $\nu_k=\mu_k=0$ for $k<k_0$, we have $k_0\leq k_1<j$.
Then 
$$\nu_{k_1+1}-\nu_{k_1}\leq \lambda_{k_1+1}-\lambda_{k_1}=\mu_{k_1+1}+\lambda_j-i-(\mu_{k_1}+\lambda_j-i)=\mu_{k_1+1}-\mu_{k_1}$$
and hence 
$$\nu_{k_1+1}\leq \mu_{k_1+1}+\nu_{k_1}-\mu_{k_1}<\mu_{k_1+1},$$ 
contradicting the maximality of~$k_1$.
\end{proof}

Let $\mu\in\Lambda$. Call $j\in [n]$ \emph{critical} (for $\mu$) if  $\mu_j>\mu_{j-1}$ and either $j=n$ or $j<n$ and $\mu_{j+1}-\mu_j<\lambda_{j+1}-\lambda_j$. Clearly if $j$ is critical, then $\mu-e_j\in\Lambda$. (Here and below, $e_1,\dots,e_n$ denote the standard basis vectors in $\bR^n$.)
Moreover, one easily shows:

\begin{lemma}\label{lem:join-irred, 2}
Let $\mu\in\Lambda$ and $j,j'\in [n]$ with $j'<j$, and suppose all indices $j',j'+1,\dots,j$ are non-critical.
\begin{enumerate}
\item If $j=n$ or $j<n$ and $\mu_j=\mu_{j+1}$, then $\mu_{j'-1}=\mu_{j'}=\cdots=\mu_j$.
\item If $\mu_{k-1}<\mu_k$ for $k=j',\dots,j$ and $j<n$, then $\mu_k=\lambda_k-(\lambda_{j+1}-\mu_{j+1})$ for $k=j',\dots,j$.
\end{enumerate}
\end{lemma}

Note that part (1) of the previous lemma shows in particular that $\mu$ has no critical index iff $\mu=0$. On the other hand, if $\mu$ is join-irreducible, then $\mu$ cannot have more than one critical index; for if $j\neq j'$ both are critical, then $\mu=(\mu-e_j)\vee (\mu-e_{j'})$ where $\mu-e_j,\mu-e_{j'}\in\Lambda\setminus\{\mu\}$. Hence if $\mu$ is join-irreducible, then $\mu$ has exactly one critical index.

\begin{lemma}
Let $\mu\in\Lambda$ have the unique critical index $j\in [n]$; then $\mu=\mu(i,j)$ where $i=\mu_j$.
\end{lemma}
\begin{proof}
By part (1) of the previous lemma we have $i=\mu_j=\mu_{j+1}=\cdots=\mu_n$; note that $i>0$ since $\mu\neq 0$.
On the other hand, if $\mu_{k-1}=\mu_k$ for some $k\in [j]$, then by the same part of Lemma~\ref{lem:join-irred, 2} we have $0=\mu_1=\cdots=\mu_k$. So if we let $k_0\in [n]$ be minimal such that $\mu_{k_0}>0$, then $k_0\leq j$ and $\mu_{k-1}<\mu_k$ for $k=k_0,\dots,j$ and hence $\mu_k=\lambda_k-(\lambda_j-\mu_j)$ for $k=k_0,\dots,j-1$. Also, 
$$\lambda_{k_0}-(\lambda_j-\mu_j) = \mu_{k_0} - \mu_{k_0-1} \leq \lambda_{k_0}-\lambda_{k_0-1}$$
and hence
$\lambda_{k_0-1}-(\lambda_j-\mu_j)\leq 0$. This yields the claim.
\end{proof}

\begin{corollary}\label{cor:join-irred}
$J(\Lambda) = \big\{ \mu(i,j) : (i,j)\in P \big\}$.
\end{corollary}
\begin{proof}
Lemma~\ref{lem:join-irred, 1} implies that the  $\mu=\mu(i,j)$ with $(i,j)\in P$ are join-irreducible:  if $\nu\in\Lambda$ satisfies  $\nu<\mu$, then $\nu\leq\mu-e_j$, in particular, there cannot exist $\nu,\nu'\in\Lambda$ with $\mu=\nu\vee\nu'$ and $\nu,\nu'<\mu$.
Conversely, by the previous lemma, every join-irreducible element $\mu$ of $\Lambda$ is of the form $\mu=\mu(i,j)$ with $(i,j)\in P$.
\end{proof}

\subsubsection{The width of $J(\Lambda)$.}
Since by Lemma~\ref{lem:join-irred, 0} and Corollary~\ref{cor:join-irred} we now have an explicit description of the ordered set $J(\Lambda)$, we can read off the various numerical invariants of $\Lambda$ from $P$ by means of Proposition~\ref{prop:dilworth}. For example, we have 
$$\height(\Lambda)=\lambda_1+\cdots+\lambda_n+1 \quad\text{and}\quad \Gdim(\Lambda)=\Gdim^*(\Lambda)=1.$$ 
To complete the proof of Proposition~\ref{prop:breadth of finite exponent group}, it remains to compute the width of $P$ (and hence of $J(\Lambda)$), using Dilworth's Theorem: we will specify a partition of $P$ into $d=d(\lambda_1,\dots,\lambda_n)$ many chains which, by virtue of Dilworth's Theorem, shows $\width(P)\leq d$, and, then we pick an element from each of these chains to form an antichain of size $d$, which shows $\width(P)\geq d$.

\begin{figure}

\begin{tikzpicture}

\draw[step=1,gray,thin] (1,1) grid (2,6);
\draw[step=1,gray,thin] (2,2) grid (3,6);
\draw[step=1,gray,thin] (3,3) grid (5,6);
\draw[step=1,gray,thin] (5,4) grid (7,6);
\draw[step=1,gray,thin] (7,5) grid (8,6);
\draw[gray,thin] (8,6) to (9,6);

\fill [black] (1,1) circle (3pt);
\fill [black] (2,1) circle (3pt);

\path (1, 6.5) node {$C_1$};
\draw[->,thick,shorten >=0.5em] (1,2) to (1,1); 
\draw[->,thick,shorten >=0.5em] (1,3) to (1,2); 
\draw[->,thick,shorten >=0.5em] (1,4) to (1,3); 
\draw[->,thick,shorten >=0.5em] (1,5) to (1,4); 
\draw[->,thick,shorten >=0.5em] (1,6) to (1,5); 
\draw[->,thick,shorten >=0.5em] (1,1) to (2,2); 
\draw[->,thick,shorten >=0.5em] (2,2) to (2,1); 
\draw[->,thick,shorten >=0.5em] (2,1) to (3,2); 
\draw[->,thick,shorten >=0.5em] (3,2) to (5,3); 
\draw[->,thick,shorten >=0.5em] (5,3) to (7,4); 
\draw[->,thick,shorten >=0.5em] (7,4) to (8,5); 
\draw[->,thick,shorten >=0.5em] (8,5) to (9,6); 

\path (2, 6.5) node {$C_2$};
\draw[->,thick,shorten >=0.5em] (2,6) to (2,5); 
\draw[->,thick,shorten >=0.5em] (2,5) to (2,4); 
\draw[->,thick,shorten >=0.5em] (2,4) to (2,3); 
\draw[->,thick,shorten >=0.5em] (2,3) to (3,3); 
\draw[->,thick,shorten >=0.5em] (3,3) to (4,3); 
\draw[->,thick,shorten >=0.5em] (4,3) to (6,4); 
\draw[->,thick,shorten >=0.5em] (6,4) to (7,5); 
\draw[->,thick,shorten >=0.5em] (7,5) to (8,6); 

\path (3, 6.5) node {$C_3$};
\draw[->,thick,shorten >=0.5em] (3,6) to (3,5); 
\draw[->,thick,shorten >=0.5em] (3,5) to (3,4); 
\draw[->,thick,shorten >=0.5em] (3,4) to (4,5); 
\draw[->,thick,shorten >=0.5em] (4,5) to (4,4); 
\draw[->,thick,shorten >=0.5em] (4,4) to (5,5); 
\draw[->,thick,shorten >=0.5em] (5,5) to (5,4); 
\draw[->,thick,shorten >=0.5em] (5,4) to (6,5); 
\draw[->,thick,shorten >=0.5em] (6,5) to (7,6); 

\path (4, 6.5) node {$C_4$};
\draw[->,thick,shorten >=0.5em] (4,6) to (5,6); 
\draw[->,thick,shorten >=0.5em] (5,6) to (6,6);

\fill [black] (1,2) circle (3pt);
\fill [black] (2,2) circle (3pt);
\fill [black] (3,2) circle (3pt);

\fill [black] (1,3) circle (3pt);
\fill [black] (2,3) circle (3pt);
\fill [black] (3,3) circle (3pt);
\fill [black] (4,3) circle (3pt);
\fill [black] (5,3) circle (3pt);

\fill [black] (1,4) circle (3pt);
\fill [black] (2,4) circle (3pt);
\fill [black] (3,4) circle (3pt);
\fill [black] (4,4) circle (3pt);
\fill [black] (5,4) circle (3pt);
\fill [black] (6,4) circle (3pt);
\fill [black] (7,4) circle (3pt);

\fill [black] (1,5) circle (3pt);
\fill [black] (2,5) circle (3pt);
\fill [black] (3,5) circle (3pt);
\fill [black] (4,5) circle (3pt);
\fill [black] (5,5) circle (3pt);
\fill [black] (6,5) circle (3pt);
\fill [black] (7,5) circle (3pt);
\fill [black] (8,5) circle (3pt);

\fill [black] (1,6) circle (3pt);
\fill [black] (2,6) circle (3pt);
\fill [black] (3,6) circle (3pt);
\fill [black] (4,6) circle (3pt);
\fill [black] (5,6) circle (3pt);
\fill [black] (6,6) circle (3pt);
\fill [black] (7,6) circle (3pt);
\fill [black] (8,6) circle (3pt);
\fill [black] (9,6) circle (3pt);

\draw [black] (1,1) circle (4pt);
\draw [black] (2,3) circle (4pt);
\draw [black] (3,4) circle (4pt);
\draw [black] (4,6) circle (4pt);

\end{tikzpicture}
\caption{Partition of $P$ into the chains $C_k$}
 \label{fig:example of P, 2}

\end{figure}
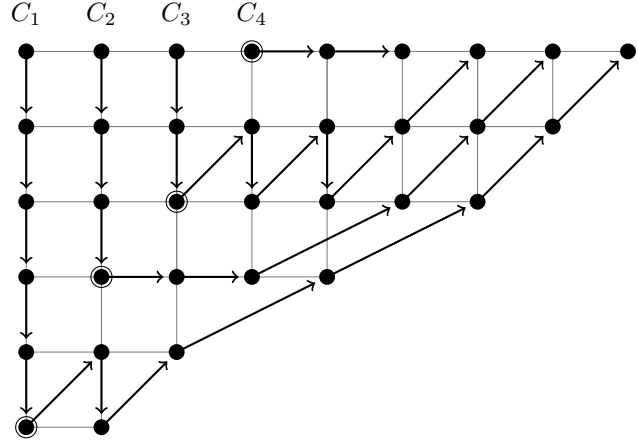

We consider the strictly increasing sequence 
$$1=j(1)<\cdots<j(k)<\cdots<j(d)$$
defined according to the definition of $d=d(\lambda_1,\dots,\lambda_n)$, i.e., for $k\in [d-1]$, we have
$j(k+1)=j(k)+1$ if $\lambda_{j(k)+1}-\lambda_{j(k)}\geq 2$ and 
$j(k+1)=j(k)+2$ otherwise.
Now for $k\in [d]$ define the subsets
\begin{align*}
C_k^{\downarrow}		&:= \big\{ (k,j) : j=j(k),\dots,n \big\}  \\
C_k^{\rightsquigarrow}	&:= \big\{ (i,j) : i=k,\dots,\lambda_{j(k)}-k+1,\ j(k)\leq j<j(k+1) \big\} \\
C_k^{\nearrow}		&:= \big\{ (\lambda_j-k+1,j) : j=j(k),\dots,n \big\}
\end{align*} 
of $\bN^2$, where we set $j(d+1)=n+1$. Using Lemma~\ref{lem:upper bound on d(I)} it is not difficult to verify that 
each of the sets $C_k^{\downarrow}$, $C_k^{\rightsquigarrow}$ and $C_k^{\nearrow}$ is contained in $P$.
Note that in $P$, we have
$$(k,n)<(k,n-1)<\cdots<(k,j(k)),$$
so $C_k^{\downarrow}$ is a chain in $P$. Also, $C_k^{\rightsquigarrow}$ is a chain, since
$$(k,j(k))<(k+1,j(k))<\cdots<(\lambda_{j(k)}-k+1,j(k))\quad
\text{if $j(k+1)-j(k)=1$}$$
and
\begin{multline*}
(k,j(k)+1)<(k,j(k))<(k+1,j(k)+1)<(k+1,j(k))<\cdots < \\ (\lambda_{j(k)}-k+1,j(k))  \qquad
\text{if $j(k+1)-j(k)=2$.}
\end{multline*}
Similarly we see that $C_k^{\nearrow}$ is a chain, since
$$(\lambda_{j(k)}-k+1,j(k))<(\lambda_{j(k)+1}-k+1,j(k)+1)<\cdots<(\lambda_n-k+1,n).$$
Hence
$$C_k:=C_k^{\downarrow} \cup C_k^{\rightsquigarrow} \cup C_k^{\nearrow}$$
is a chain, with smallest element $(k,n)$ and largest element $(\lambda_n-k+1,n)$. Figure~\ref{fig:example of P, 2} shows the $C_k$ in the example introduced in Figure~\ref{fig:example of P}.

As with $d=d(\lambda_1,\dots,\lambda_n)$, the definition of the ordered set $P$ and the chains~$C_k$ depend on $\lambda=(\lambda_1,\dots,\lambda_n)$; we write $P=P(\lambda)$ and $C_k=C_k(\lambda)$, respectively, to make this dependence explicit.

\begin{lemma}\label{lem:partition into chains}
The chains $C_1,\dots,C_d$ form a partition of $P$.
\end{lemma}
\begin{proof}
We proceed by induction on $n$.
Note that 
$$C_k\cap ([\lambda_n]\times\{n\}) = 
\begin{cases}
\{k,\lambda_n-k+1\}\times\{n\} & \text{if $j(k)<n-1$} \\
\{k,\dots,\lambda_n-k+1\}\times\{n\} & \text{if $j(k)=n$ or $j(k)=n-1$,}
\end{cases}
$$ and hence
the sets $C_k\cap ([\lambda_n]\times\{n\})$ ($k=1,\dots,d$) partition  $[\lambda_n]\times\{n\}$.
Suppose $n>1$ and set $\lambda'=(\lambda_1,\dots,\lambda_{n-1})$; then we clearly have 
$$P(\lambda') = P(\lambda)\cap ([\lambda_{n-1}]\times [n-1])$$ 
as ordered sets, and
$$C_k(\lambda)\cap ([\lambda_{n-1}]\times [n-1]) = 
\begin{cases}
C_k(\lambda') & \text{if $j(k)\leq n-1$ } \\
\emptyset     & \text{if $j(k)=n$.}
\end{cases}$$
By inductive hypothesis therefore, if $j(d)<n$, then
the sets $C_k\cap ([\lambda_{n-1}]\times [n-1])$ ($k=1,\dots,d$) partition  $P\cap ([\lambda_{n-1}]\times [n-1])$, if $j(d)=n$, then the sets $C_k\cap ([\lambda_{n-1}]\times [n-1])$  ($k=1,\dots,d-1$)  partition  $P\cap ([\lambda_{n-1}]\times [n-1])$, with $C_d\subseteq [\lambda_n]\times\{n\}$. Hence the $C_k$ ($k=1,\dots,d$) partition $P=\bigcup_{k\in [n]} [\lambda_k]\times\{k\}$.
\end{proof}

By Lemma~\ref{lem:partition into chains} and Dilworth's~Theorem we have $\width(P)\leq d$; the next lemma thus shows that $\width(P) = d$. (In Figure~\ref{fig:example of P, 2} the elements of this antichain are circled.)

\begin{lemma}
The elements $(1,j(1)),\dots,(d,j(d))$  form an antichain of $P$.
\end{lemma}
\begin{proof}
For all $k\in [d-1]$ we have $\lambda_{j(k+1)}-\lambda_{j(k)}\geq 2$, hence the sequence $\lambda_{j(k)}-k$ ($k\in [d]$) is strictly increasing.
Since $(k,j(k)) \leq (l,j(l))$ iff $k\leq l$ and $\lambda_{j(k)}-k\geq\lambda_{j(l)}-l$ for all $k,l\in [d]$, this immediately yields the claim.
\end{proof}

This finishes the proof of Proposition~\ref{prop:breadth of finite exponent group}, and hence of Theorem~\ref{thm:breadth of finite exponent group}. \qed 

\medskip
\noindent
From Theorem~\ref{thm:breadth of finite exponent group} and Proposition~\ref{prop:breadth and vc for modules} we see immediately:

\begin{corollary}\label{cor:finite exp, cor 1}
Suppose $A$ has finite exponent, and let $d$ be as in \eqref{eq:d};
then $A$ has the $\VC{}d$ property.  Hence if $A$ is infinite, then $\vc^{T}(m)\leq dm$ for each $m$. 
\end{corollary}

For the following examples also consult Examples~\ref{ex:d(I)} above:

\begin{example*}
Suppose
$$A = \bZ(p^2)^{(\alpha_2)}\oplus \bZ(p^3)^{(\alpha_3)}\oplus \bZ(p^5)^{(\alpha_5)}\oplus \bZ(p^7)^{(\alpha_7)}\oplus \bZ(p^8)^{(\alpha_8)}\oplus  \bZ(p^9)^{(\alpha_9)},$$
where the $\alpha_i$ are non-zero cardinals. Then $A$ has breadth~$4$ and hence $A$ has the $\VC{}4$ property; so if one of the $\alpha_i$ is infinite, then $\vc^{T}(m)\leq 4m$ for each $m$.
\end{example*}

\begin{example*}
Suppose 
$$A=\bZ(p)^{(\alpha_1)}\oplus\bZ(p^2)^{(\alpha_2)}\oplus\cdots\oplus\bZ(p^n)^{(\alpha_n)}\qquad (n>0)$$ 
where the $\alpha_i$ are  non-zero cardinals. 
Then $A$ has the $\VC{}d$ property where $d=\lceil n/2\rceil$; so if $A$ is infinite, then $\vc^{T}(m)\leq \lceil n/2\rceil m$ for each $m$. 
\end{example*}

In fact, we can now precisely determine the VC~density function of an infinite abelian group of finite exponent.
Recall that for each $p$, in the introduction we defined
$$U_{\geq\aleph_0}(p;A)=\big\{i\geq 0:U(p,i;A)\geq\aleph_0\big\},$$
a subset of $U(p;A)$. 
With this notation we have:

\begin{corollary}\label{cor:finite exp, cor 2}
Suppose $A$ is $\aleph_0$-categorical, and let 
$$d_{0}:= \sum_p d(U_{\geq\aleph_0}(p;A)).$$
Then $\vc^T(m)=d_0 m$ for each $m$.
\end{corollary}

\begin{proof}
We may write $A=A_{\geq\aleph_0}\oplus B$ where $U_{\geq\aleph_0}(p;A)=U(p;A_{\geq\aleph_0})$ for each $p$ and $B$ is finite: if we decompose each $p$-primary component as
$$A_p \cong \bigoplus_{i>0} \bZ(p^i)^{(\alpha_{p,i-1})},$$
then
$$A_{\geq\aleph_0} \cong \bigoplus_{\substack{\text{$p$ prime, $i>0$} \\ \text{with $\alpha_{p,i-1}\geq\aleph_0$}}} \bZ(p^i)^{(\alpha_{p,i-1})}, \qquad
B \cong \bigoplus_{\substack{\text{$p$ prime, $i>0$} \\ \text{with $\alpha_{p,i-1}<\aleph_0$}}} \bZ(p^i)^{(\alpha_{p,i-1})}.
$$
Replacing $A$ by $A_{\geq\aleph_0}$ if necessary we may assume that $U_{\geq\aleph_0}(p;A)=U(p;A)$ for every $p$ and so $A^{\aleph_0}\equiv A$ and $d_0=d=\breadth(A)$. Now $\vc^{T}(m) = dm$ follows from Corollary~\ref{cor:lower bounds on vc}.
\end{proof}

Theorem~\ref{thm:aleph0-cat abelian} from the introduction is a consequence of Corollary~\ref{cor:finite exp, cor 2} and the upper bound in \eqref{eq:bound on d(I)}. We also obtain:

\begin{corollary}\label{cor:finite exp dp-minimal}
An  $\aleph_0$-categorical abelian group is dp-minimal iff it is isomorphic to one of the form
$$\bZ(p^k)^{(\alpha)}\oplus\bZ(p^{k+1})^{(\beta)}\oplus B,$$ 
where $p$ is a prime, $k>0$, $\alpha$, $\beta$ are cardinals, at least one of which is infinite, and~$B$ is a finite abelian group.
\end{corollary}
\begin{proof}
Suppose $A$ is $\aleph_0$-categorical. Then by the previous corollary, 
$$\vc^T(1)=\sum_p d(U_{\geq\aleph_0}(p;A)).$$
Thus by Corollary~\ref{cor:commensurable, 1}, $A$ is dp-minimal iff $U_{\geq\aleph_0}(p;A)\neq\emptyset$ for exactly one prime~$p$ and for this $p$, $d(U_{\geq\aleph_0}(p;A))=1$. By \eqref{eq:d(I)=1}, the latter condition is equivalent to
$U_{\geq\aleph_0}(p;A)=\{k\}$ or $U_{\geq\aleph_0}(p;A)=\{k,k+1\}$, for some $k>0$.
\end{proof}

\subsection{Asymptotics of p.p.~definable subgroups in  homocyclic groups.}\label{sec:homocyclic}
In this subsection we go beyond  finite-exponent groups and consider abelian groups of the form
$$A = \bigoplus_{i>0} \bZ(p^i)^{(\alpha_{i-1})}$$
where the $\alpha_{i-1}$ are natural numbers; we assume that $A$ is infinite, i.e., $\alpha_{i-1}>0$ for infinitely many $i>0$. Our goal is to show:

\begin{proposition}\label{prop:homocyclic}
For every $m$ we have $\breadth(\widetilde{\PP}_m(A))\leq m$.
\end{proposition}

Note that by Corollary~\ref{cor:lower bounds on vc} (and since always $\vc^T(m)\geq m$), this proposition implies $\vc^{T}(m)=m$ for every $m$. Before we give a proof of Proposition~\ref{prop:homocyclic}, we first look at the instructive case $m=1$, in which we  also have a sharper result.  Below, for each~$n$ we let $A[n]:=\{a\in A:na=0\}$ (a~p.p.~definable subgroup of~$A$).

\begin{lemma}\label{lem:asymp pp, chain}
The ordered set $\widetilde{\PP}(A)=\widetilde{\PP}_1(A)$ is a chain of order type $\omega+\omega^*$, with representatives given by the p.p.~definable subgroups $A[p^d]$ and $p^e A$:
$$0\lnsim A[p] \lnsim \cdots \lnsim A[p^d]\lnsim A[p^{d+1}]\lnsim  \cdots \lnsim p^{e+1}A\lnsim p^eA\lnsim\cdots \lnsim  pA \lnsim A.$$
\end{lemma}
\begin{proof}
For all $d<d'$ the p.p.~formula $\delta^p_{d,d'}(x)=\delta_{d,d'}(x):=(p^{d'}|p^d x)$ defines the subgroup
$$\delta_{d,d'}(\bZ(p^i))=\begin{cases}
\bZ(p^i)				& \text{if $i\leq d$} \\
p^{i-d}\bZ(p^i)		& \text{if $d<i<d'$} \\
p^{d'-d}\bZ(p^i)		& \text{if $i\geq d'$}
\end{cases}$$
of $\bZ(p^i)$, hence
$$\delta_{d,d'}(A) = \bigoplus_{i < d'} p^{\max\{0,i-d\}}\bZ(p^i)^{(\alpha_{i-1})}\oplus	
\bigoplus_{i \geq d'} p^{d'-d} \bZ(p^i)^{(\alpha_{i-1})}.$$
Note that setting $e:=d'-d$, the group $\delta_{d,d'}(A)$ is commensurable with 
$$p^{e}A=\bigoplus_{i\geq e} p^{e} \bZ(p^i)^{(\alpha_{i-1})}.$$
For each $d$ the p.p.~formula $\tau^p_d(x)=\tau_d(x):=(p^d x=0)$ defines the subgroup
$\tau_d(\bZ(p^i))=p^{\max\{i-d,0\}}\bZ(p^i)$ in $\bZ(p^i)$, so
$$A[p^d]=\tau_d(A) \sim \bigoplus_{i\geq d} p^{i-d} \bZ(p^i)^{(\alpha_{i-1})}.$$
This description of the subgroups $A[p^d]$ and $p^eA$ makes it clear that $A[p^d]\lesssim p^eA$ for all $d$, $e$. 
Since every p.p.~definable subgroup of $A$ is a finite intersection of groups definable by $\tau_d$'s or $\delta_{d,d'}$'s, the lemma follows.
\end{proof}

The key in this proof was that for any choice of $d$ and $e$, for large enough $i$ we have an inclusion $p^e\bZ(p^i) \geq p^{i-d}\bZ(p^i)$ among the subgroups of $\bZ(p^i)$ defined by the p.p.~formulas $p^e|x$ and $p^dx=0$, respectively.
For the case $m>1$ we proceed in a similar way, by first investigating the intersection behavior, as $i\to\infty$, of subgroups of the homocyclic $p$-group $\bZ(p^i)^m$ defined by a fixed collection of p.p.~formulas:

\begin{lemma}\label{lem:asymp pp}
Let $\varphi_1(x),\dots,\varphi_n(x)$ be p.p.~$\mathcal L_{\bZ}$-formulas where $\abs{x}=m$. There are $i_1,\dots,i_m\in [n]$ such that $\varphi_1\wedge\cdots\wedge\varphi_n$ and $\varphi_{i_1}\wedge\cdots\wedge\varphi_{i_m}$ define the same subgroup of $\bZ(p^i)^m$, for all sufficiently large $i$.
\end{lemma}

In the proof of this lemma, which we give after some preliminary observations, we use the following basic properties of finitely generated modules over discrete valuation rings~(DVRs):

\begin{lemma}
Let $R$ be a DVR with maximal ideal $\mathfrak m$ and residue field $k=R/\mathfrak m$,  let $M$ be a finitely generated $R$-module, and let $\overline{M}=M/\mathfrak m M$, viewed as a $k$-linear space, with natural surjection $x\mapsto\overline{x}\colon M\to\overline{M}$.
\begin{enumerate}
\item Elements $a_1,\dots,a_n$ of $M$ form a minimal generating set for the $R$-module $M$ iff $\overline{a_1},\dots,\overline{a_n}$ form a basis for the $k$-linear space $\overline{M}$. 
\item If $M$ is torsion-free, then $M$ is free, and any minimal generating set for $M$ is a basis for $M$.
\item If $M$ is a submodule of $R^m$, then $M$ can be generated by $m$ elements.
\end{enumerate}
In particular, if $M$ is a submodule of $R^m$ and $a_1,\dots,a_n\in M$ generate $M$, then there are $i_1,\dots,i_r\in [n]$, where $r=\dim_k \overline{M}\leq m$, such that $a_{i_1},\dots,a_{i_r}$ is a basis for~$M$.
\end{lemma}

We omit the proof of these facts, which can be found in any standard text on commutative algebra. (Part~(1) is a consequence of Nakayama's Lemma and holds more generally if $R$ is a local ring.)
We use this lemma to show:

\begin{lemma}\label{lem:DVR}
Let $R$ be a DVR and $\pi$ a generator of the maximal ideal of $R$, and let $a_1,\dots,a_n,b_1,\dots,b_n\in R^m$. There are $i_1,\dots,i_r,j_1,\dots,j_s\in [n]$ with $r+s\leq m$ and some integer $e\geq 0$ such that 
\begin{enumerate}
\item $a_{i_1},\dots,a_{i_r}$ generate the $R$-submodule $N=Ra_1+\cdots+Ra_n$ of $R^m$;
\item for each $i\in [n]$ we  have $\pi^e b_i\in N+R\pi^eb_{j_1}+\cdots+R\pi^eb_{j_s}$.
\end{enumerate}
\end{lemma}
\begin{proof}
By the lemma above, first choose $i_1,\dots,i_r\in [n]$ such that $a_{i_1},\dots,a_{i_r}$ is a basis for $N$. 
For each $k$ let $M_k$ be the $R$-submodule $N+R\pi^kb_1+\cdots+R\pi^kb_n$ of~$R^m$, so 
$$N\subseteq\cdots\subseteq M_{k+1}\subseteq M_k\subseteq\cdots\subseteq M_0=:M.$$ 
Now the torsion submodule of $M/N$ is finitely generated, hence we may take some $e\in\bN$ such that if $a\in M$ satisfies $\pi^ka\in N$ for some $k\geq 0$, then $\pi^ea\in N$. Then $M_e/N$ is torsion-free: if $a\in M_e$ is torsion in $M_e/N$, then $\pi^e a\in N$, so
writing $a=b+\pi^e c$ ($b\in N$, $c\in M$) we have $\pi^{2e}c=\pi^ea-\pi^eb\in N$ and hence $\pi^ec\in N$, i.e., $a\in N$.
Now apply the lemma again and choose $j_1,\dots,j_s\in [n]$ such that $\pi^e b_{j_1},\dots,\pi^e b_{j_s}$ map onto a basis for $M_e/N$ under the natural surjection $M_e\to M_e/N$. Then $M_e=\bigoplus_{k} Ra_{i_k} \oplus \bigoplus_l R\pi^e b_{j_l}$ (internal direct sum of $R$-submodules of $M_e$), so $r+s\leq m$ (again by the lemma above), and the $i_k$, $j_l$ and~$e$ have the desired properties.
\end{proof}

We now prove Lemma~\ref{lem:asymp pp}:

\begin{proof}[Proof of Lemma~\ref{lem:asymp pp}]
Let $a_1,\dots,a_n,b_1,\dots,b_n\in\bZ^m$ and $d_1,\dots,d_n\in\bN$, and consider the p.p.~formulas
$$\chi_j(x) = (a_jx = 0),\qquad \psi_j(x) = (p^{d_j}|b_j x) \qquad (j\in [n]).$$
Here, for $a=(a_1,\dots,a_m)\in\bZ^m$ and $x=(x_1,\dots,x_m)$, we denote by $ax$ the $\mathcal L_{\bZ}$-term $a_1x_1+\cdots+a_mx_m$. It suffices to show (cf.~\cite[Lemma~A.2.1]{Hodges} or \cite[Theorem~2.$\bZ$1]{Prest}): there are $i_1,\dots,i_r,j_1,\dots,j_s\in [n]$ with $r+s\leq m$ such that for sufficiently large $i$, the formulas $\chi_{1}\wedge\cdots\wedge\chi_{n}\wedge\psi_{1}\wedge\cdots\wedge\psi_{n}$ and
$\chi_{i_1}\wedge\cdots\wedge\chi_{i_r}\wedge\psi_{j_1}\wedge\cdots\wedge\psi_{j_s}$ define the same subgroup of 
$\bZ(p^i)^m$. Now
$\psi_j(x)$ and $p^{\max\{0,i-d_j\}}b_jx=0$ define the same subgroup of
$\bZ(p^i)^m$, and if $i$ and $d$ are such that $i\geq d\geq d_j$ for each $j$, then $p^{\max\{0,i-d_j\}}b_j=p^{i-d}(p^{d-d_j}b_j)$, hence after replacing $b_j$ by $p^{d-d_j}b_j$ we may assume that $d=d_j$ for each $j$.
Set $R=\bZ_{(p)}$ (the localization of the ring of integers at its prime ideal $(p)=p\bZ$) and $\pi=p$, and choose $i_1,\dots,i_r,j_1,\dots,j_s\in [n]$ and $e\in\bN$ as in Lemma~\ref{lem:DVR}.
Then by condition (1) in this lemma,
$\chi:=\chi_{1}\wedge\cdots\wedge\chi_{n}$ and $\chi_{i_1}\wedge\cdots\wedge\chi_{i_r}$ define the same subgroup of 
$\bZ(p^i)^m$; similarly, by condition (2), if $i\geq d+e$, then $\chi\wedge \psi_{1}\wedge\cdots\wedge\psi_{n}$ and $\chi\wedge\psi_{j_1}\wedge\cdots\wedge\psi_{j_s}$  define the same subgroup of $\bZ(p^i)^m$.
\end{proof}

Lemma~\ref{lem:asymp pp} immediately implies Proposition~\ref{prop:homocyclic}. \qed

\medskip
\noindent
We note that if the question posed after the proof of Proposition~\ref{prop:commensurable} had a positive answer, then Proposition~\ref{prop:homocyclic} would simply follow from Lemma~\ref{lem:asymp pp, chain}.

\subsection{Abelian groups with uniform bounds on VC~density.}\label{sec:abelian with uniform VC}
Throughout this subsection we assume that $A$ is infinite.
A list of abelian groups $A$ with $\vc^T(m)=m$ for all $m$ (so that  in particular, $T=\Th(A)$ is dp-minimal) includes: 
$\bZ^n$ and $\bZ^n\oplus\bQ$, for each $n$ (in each case, an expansion has the $\VC{}1$ property, by 
\cite[Corollary~6.5]{ADHMS}); and any p.p.~uniserial abelian group such as 
$$\bZ(p^d)^{(\alpha)}\oplus\bZ_{(p)}^{(\beta)}\oplus\bQ^{(\gamma)}$$ 
where $d$ is a positive integer or $\infty$, and $\alpha$, $\beta$, $\gamma$ are cardinals (by Cor\-ollary~\ref{cor:pp uniserial}). (Here
$\bZ_{(p)}$ is the additive group of the localization of the ring of integers at its prime ideal $(p)=p\bZ$.)
On the other hand, the following lemma shows that for many~$A$ we have $\vc^{T}(1)=\infty$.

\begin{lemma}\label{lem:lower bound, abelian gps}
Suppose $p_1,\dots,p_d$ are pairwise distinct primes. If 
\begin{enumerate}
\item $A[p_i]$ is infinite for each $i$, or
\item $A/p_iA$ is infinite for each $i$, 
\end{enumerate}
then $\vc^T(m)\geq dm$ for each $m$ \textup{(}and so if $d>1$, then $T$ is not dp-minimal\textup{)}.
\end{lemma}
\begin{proof}
In the first case, this follows from Corollary~\ref{cor:Gdim and vc for modules},~(1) applied to $H_i=A[p_i]$, and in the second case, from Corollary~\ref{cor:Gdim and vc for modules},~(2) applied to $H_i=p_iA$.
\end{proof}

In this subsection we extend this observation to a characterization of all $A$ satisfying $\vc^T(m)<\infty$ for all $m$. We say that a prime $p$ is \emph{non-singular} for $A$ if both $A[p]$ and $A/pA$ are finite, and \emph{singular} for $A$ otherwise. We also say that an abelian group is \emph{non-singular} if it is either finite, or  infinite and each prime is non-singular for it, and \emph{singular} otherwise. For example, the additive group $\bZ$ of integers is non-singular.
If an infinite abelian group is non-singular, then it has $\URk$-rank~$1$  (cf.~Example~\ref{ex:Urank 1}); but of course, the converse of this implication does not hold, as witnessed, e.g., by infinite elementary abelian $p$-groups.


\begin{theorem}\label{thm:abelian groups}
The following are equivalent:
\begin{enumerate}
\item $\vc^T(1)<\infty$;
\item $\vc^T(m)<\infty$ for every $m$;
\item there is some $d$ such that $\vc^T(m)\leq dm$ for every $m$;
\item there are only finitely many  $p$ which are singular for $A$, and for all $p$ the set
$U_{\geq\aleph_0}(p;A)$ is finite.
\end{enumerate}
\end{theorem}

For the proof of this theorem,  recall that 
a \emph{Szmielew group} is a countable abelian group of the form 
$$\bigoplus_{\text{$p$ prime}} \left(
\bigoplus_{n>0} \bZ(p^n)^{(\alpha_{p,n-1})}\oplus  \bZ(p^\infty)^{(\beta_p)} \oplus \bZ_{(p)}^{(\gamma_p)}
\right)\oplus\bQ^{(\delta)}$$
where the $\alpha_{p,n-1}$, $\beta_p$, $\gamma_p$ and $\delta$ are cardinals (finite or $\aleph_0$).
Such a Szmielew group is \emph{strict} if
\begin{enumerate}
\item $\delta$ is either $0$ or $\aleph_0$;
\item if $\beta_p\neq 0$ or $\gamma_p\neq 0$ for some $p$ or $\alpha_{p,n-1}\neq 0$ for infinitely many pairs $(p,n)$, then $\delta=0$; and
\item for each  $p$, if there is no finite upper bound on the order of the elements which have $p$-power order and are not divisible by $p$, then $\beta_p=\gamma_p=0$.
\end{enumerate}
Any abelian group is elementarily equivalent to a unique strict Szmie\-lew group (cf.~\cite[A.2]{Hodges}). 
Also, suppose $B$ is an abelian group, elementarily equivalent to $A$, and $p$ is a prime. Then for each $N\in\bN$ we have
$\abs{A[p]}=N\Longleftrightarrow \abs{B[p]}=N$ and $\abs{A/pA}=N\Longleftrightarrow\abs{B/pB}=N$, so $p$ is non-singular for $A$ iff $p$ is non-singular for~$B$. Similarly one sees that $U_{\geq\aleph_0}(p;A)=U_{\geq\aleph_0}(p;B)$. Hence $A$ satisfies condition~(4) in Theorem~\ref{thm:abelian groups} iff $B$ does.
So for the proof of the theorem we may assume that~$A$ is a strict Szmielew group as above. 
We then have 
$$U(p,n;A) = \abs{(p^nA)[p]/(p^{n+1}A)[p]} =  p^{\alpha_{p,n}}\qquad\text{for all $n$,}$$ 
and
$$\dim_{\mathbb F_p} A[p] = \sum_{i\geq 0} \alpha_{p,i} + \beta_p, \quad
\dim_{\mathbb F_p} A/pA = \sum_{i\geq 0} \alpha_{p,i} + \gamma_p.$$
We see from this that $p$ is non-singular for $A$ iff $\alpha_{p,n}<\aleph_0$ for every $n$, with $\alpha_{p,n}=0$ for all but finitely many~$n$, and $\beta_p,\gamma_p<\aleph_0$. Hence if $U_{\geq\aleph_0}(p;A)$ is non-empty, then $p$ is singular for $A$. 

\begin{proof}[Proof of Theorem~\ref{thm:abelian groups}.]
Since the implications (3)~$\Rightarrow$~(2)~$\Rightarrow$~(1) are trivial, we only need to prove (1)~$\Rightarrow$~(4)~$\Rightarrow$~(3).
Suppose that $\vc^T(1)<\infty$. By Lemma~\ref{lem:lower bound, abelian gps} there are only finitely many singular primes for $A$. 
Suppose $p$ is such that $U(p,i;A)\geq\aleph_0$ for infinitely many $i$.  
Let $(i_k)_{k>0}$ be a strictly increasing sequence of non-negative integers such that $U(p,i_k;A)\geq\aleph_0$ for each $k$; then
$$A_n:=\bigoplus_{k\in [n]} \bZ(p^{i_k})^{(\alpha_{p,i_k-1})}$$
is a pure subgroup of $A$, so $\vc^T\geq\vc^{T_n}$ for $T_n=\Th(A_n)$; 
see the discussion at the beginning of Section~\ref{sec:direct sum}. 
Since $A_n$ has finite exponent $p^{i_n}$ and $U_{\geq\aleph_0}(p;A_n)=\{i_1,\dots,i_n\}$,
by 
Corollary~\ref{cor:finite exp, cor 2} and \eqref{eq:bound on d(I)} we have $\vc^{T_n}(1)\geq \lceil n/2\rceil $. Since this holds for each $n$, we obtain  $\vc^T(1)=\infty$, a contradiction. This shows (1)~$\Rightarrow$~(4).

For the proof of the remaining implication (4)~$\Rightarrow$~(3), suppose condition (4) holds. Then the set 
$$\mathcal P = \big\{ p: \text{$\dim_{\mathbb F_p} A[p]<\aleph_0$ and $\dim_{\mathbb F_p} A/pA<\aleph_0$} \big\}$$
of non-singular primes for $A$ contains all but finitely many primes. The group
$$A^{[\mathcal P]} := \bigoplus_{p\in\mathcal P} \left(
\bigoplus_{n>0} \bZ(p^n)^{(\alpha_{p,n-1})}\oplus  \bZ(p^\infty)^{(\beta_p)} \oplus \bZ_{(p)}^{(\gamma_p)}
\right)\oplus\bQ^{(\delta)}$$
is non-singular. Thus if $A^{[\mathcal P]}$ is infinite, then, with $T^{[\mathcal P]}=\Th(A^{[\mathcal P]})$, we have
\begin{equation}\label{eq:vc-1}
\vc^{T^{[\mathcal P]}}(m)=m\qquad\text{for every $m$,}
\end{equation}
by Corollary~\ref{cor:finite U-rank for modules}.
Now let $p\notin\mathcal P$ and set
$$A^{[p]}:=\bigoplus_{n>0} \bZ(p^n)^{(\alpha_{p,n-1})}\oplus  \bZ(p^\infty)^{(\beta_p)} \oplus \bZ_{(p)}^{(\gamma_p)}$$
and $T^{[p]}=\Th(A^{[p]})$, and suppose $A^{[p]}$ is infinite. If $\alpha_{p,n}=0$ for all but finitely many~$n$, then $A^{[p]}$ is the direct sum of a finite exponent
group and the p.p.~uniserial abelian group $\bZ(p^\infty)^{(\beta_p)} \oplus \bZ_{(p)}^{(\gamma_p)}$, and hence satisfies 
\begin{equation}\label{eq:vc-2}
\vc^{T^{[p]}}(m)\leq m\cdot\left(d(U_{\geq\aleph_0}(p;A))+1\right)\qquad\text{for each $m$,}
\end{equation}
by Corollaries~\ref{cor:finite exp, cor 2} and \ref{cor:pp uniserial}, respectively. Suppose
$\alpha_{p,n}>0$ for infinitely many~$n$; then $\beta_p=\gamma_p=0$ since $A$ is strict Szmielew, and by assumption (4) the set $U_{\geq\aleph_0}(p;A)$ is finite.
Decomposing
$A^{[p]}= A^{[p]}_{\aleph_0}\oplus A^{[p]}_{<\aleph_0} $
with
$$A^{[p]}_{\aleph_0} = \bigoplus_{\alpha_{p,n-1}=\aleph_0} \bZ(p^n)^{(\alpha_{p,n-1})}$$
and 
$$A^{[p]}_{<\aleph_0} = \bigoplus_{\alpha_{p,n-1}<\aleph_0} \bZ(p^n)^{(\alpha_{p,n-1})},$$
we see that Corollary~\ref{cor:finite exp, cor 2} applies to $A^{[p]}_{\aleph_0}$ and
Proposition~\ref{prop:homocyclic}  to $A^{[p]}_{<\aleph_0}$; thus
\begin{equation}\label{eq:vc-3}
\vc^{T^{[p]}}(m)\leq m\cdot\left(d(U_{\geq\aleph_0}(p;A))+1\right)\qquad\text{for each $m$.}
\end{equation}
Hence from \eqref{eq:vc-1}--\eqref{eq:vc-3} and
$$A = A^{[\mathcal P]} \oplus \bigoplus_{p\notin\mathcal P} A^{[p]}$$
we see that for all $m$ we have
$$\vc^T(m) \leq \big((\abs{\mathcal P^{\operatorname{c}}}+1)+d\big)\cdot m\qquad\text{where $d=\sum_{p} d(U_{\geq\aleph_0}(p;A))$.}$$ 
Here and below $\mathcal P^{\operatorname{c}}$ denotes the complement of $\mathcal P$ in the set of prime numbers; i.e., $\mathcal P^{\operatorname{c}}$ is the (finite) set of primes singular for $A$.
Thus (3) holds (with $(\abs{\mathcal P^{\operatorname{c}}}+1)+d$ in place of $d$).
This finishes the proof of (4)~$\Rightarrow$~(3), and hence of Theorem~\ref{thm:abelian groups}. 
\end{proof}

Implicit in the proof of this theorem are upper and lower bounds on~$\vc^T$:

\begin{corollary}\label{cor:abelian groups}
Suppose $\vc^T(1)<\infty$. Then the set $\mathcal P^{\operatorname{c}}$ of primes singular for~$A$ is finite,
$U_{\geq\aleph_0}(p;A)$ is finite for all $p$, and
$U_{\geq\aleph_0}(p;A)=\emptyset$ for all but finitely many~$p$. For every $m$ we have
$$\max\{d,\abs{\mathcal P^{\operatorname{c}}}\}\cdot m\leq\vc^T(m) \leq \big(d+(\abs{\mathcal P^{\operatorname{c}}}+1)\big)\cdot m
 \quad\text{where $d=\sum_p d(U_{\geq\aleph_0}(p;A))$.}$$
\end{corollary}

\begin{proof}
Again we can assume that $A$ is strict Szmielew as before.
In the proof of Theorem~\ref{thm:abelian groups} we already observed that $\mathcal P^{\operatorname{c}}$ is finite and $U_{\geq\aleph_0}(p;A)$ is finite for all~$p$ and non-empty only for finitely many $p$, and we deduced the upper bound 
on~$\vc^T$.
For each prime $p$ set
$$A^{\langle p\rangle}:=\bigoplus_{n>0} \bZ(p^n)^{(\alpha_{p,n-1})},$$
and let
$$A' := \bigoplus_{p} A^{\langle p\rangle}[p^{m_p+1}]\qquad\text{where $m_p=\max U_{\geq\aleph_0}(p;A)$ with $\max\emptyset=-1$.}$$
Then $U_{\geq\aleph_0}(p;A')=U_{\geq\aleph_0}(p;A)$ for each $p$, and $A'$ is a p.p.~definable subgroup of the pure subgroup $\bigoplus_{p} A^{\langle p\rangle}$ of $A$, hence $\vc^T\geq\vc^{T'}$ where $T'=\Th(A')$.
Corollary~\ref{cor:finite exp, cor 2} thus yields $\vc^T(m)\geq dm$ for every $m$.
Similarly, set
$$A'' := \bigoplus_{p\in\mathcal P^{\operatorname{c}}} A^{\langle p\rangle}.$$
Then $A''$ is a pure subgroup of $A$, and hence Lemma~\ref{lem:lower bound, abelian gps} applied to $A''$ shows that
$\vc^T(m)\geq \abs{\mathcal P^{\operatorname{c}}}m$ for each $m$.
\end{proof}

Note that there are totally transcendental abelian groups with infinite VC~density, e.g., $A=\bigoplus_p \bZ(p^\infty)^{(\aleph_0)}$. 
Theorem~\ref{thm:abelian groups} allows us to give a simple characterization of superstable abelian groups with finite VC~density:

\begin{corollary}
If $T$ is superstable, then $\vc^T(1)<\infty$ iff there are only finitely many primes $p$ such that $A[p]$ is infinite. 
\end{corollary}
\begin{proof}
We may assume again that $A$ is strict Szmielew.
Now   $T$ is superstable iff there are only finitely many pairs $(p,n)$ such that 
$\sum_{i\geq n} \alpha_{p,i}+\gamma_p=\aleph_0$ (see \cite[Theorem~A.2.13]{Hodges}). In particular, if $T$ is superstable then there are only finitely many $p$ such that $A/pA$ is infinite, and for each $p$ there are only finitely many $n$ such that $\alpha_{p,n}\neq 0$. Thus the claim follows from Theorem~\ref{thm:abelian groups}.
\end{proof}

It would be interesting to know the precise values of the VC~density function of a given abelian group satisfying one of the equivalent conditions in Theorem~\ref{thm:abelian groups}. 
This was accomplished earlier in this paper in some special cases, see, e.g., 
Corollary~\ref{cor:finite exp, cor 2}.
We finish this subsection with treating another simple special case:

\begin{corollary}\label{cor:tt abelian groups}
Suppose $T$ is $\aleph_1$-categorical. Then $\vc^T(m)=m$ for each $m$. \textup{(}In particular, $T$ is dp-minimal.\textup{)}
\end{corollary}
\begin{proof}
By a theorem of Macintyre (see \cite{mac-abelian} or \cite[Theorem~A.2.12]{Hodges}), $A=B\oplus C$ where $B$ is finite and either $C$ is divisible and $C[p]$ is finite for each prime $p$; or $C=\bZ(p^n)^{(\alpha)}$ for some prime $p$, some $n>0$, and some (infinite) cardinal $\alpha$.
In the first case, $A$ is non-singular and hence of $\URk$-rank~$1$, and the claim follows from Corollary~\ref{cor:finite U-rank for modules}. In the second case, $C$ is p.p.~uniserial, hence has the $\VC{}1$~property, by Corollary~\ref{cor:pp uniserial}. So $\vc^T(m)=m$ for each $m$, by Lemma~\ref{lem:VC density under adding a finite module}.
\end{proof}

\subsection{Dp-minimal abelian groups.} \label{sec:dp-minimal abelian}
As in the previous subsection we assume here that $A$ is infinite.
For $m=1$ we can improve on Corollary~\ref{cor:abelian groups}:

\begin{proposition}\label{prop:dp-minimal abelian}
$T$ is dp-minimal iff $A$ is elementarily equivalent to one of the following abelian groups:
\begin{enumerate}
\item a direct sum of a non-singular abelian group with a group of the form
$$\left(\textstyle\bigoplus_{n>0} \bZ(p^n)^{(\alpha_{n-1})}\right)\oplus\bZ(p^\infty)^{(\beta)}\oplus\bZ_{(p)}^{(\gamma)},$$ for some prime~$p$ and cardinals $\alpha_{n-1}$, $\beta$, $\gamma$, with each $\alpha_{n-1}$ finite;
\item $\bZ(p^k)^{(\alpha)}\oplus \bZ(p^{k+1})^{(\beta)} \oplus B$ where $p$ is a prime, $k>0$, $\alpha$, $\beta$ are cardinals, at least one of which is infinite, and $B$ is a finite abelian group.
\end{enumerate}
\end{proposition}

Note that by Corollary~\ref{cor:finite exp dp-minimal}, the groups in (2) are precisely the $\aleph_0$-categorical dp-minimal abelian groups. 

\medskip
\noindent
We first show that each of the groups described in~(1) is dp-minimal. 
Recall the notation introduced in the proof of Lemma~\ref{lem:asymp pp, chain}: for $d<d'$, $\delta^p_{d,d'}(x)$ denotes the p.p.~formula $p^{d'}|p^dx$, and $\tau^p_d(x)$ the p.p.~formula $p^dx=0$. (If~$p$ is understood from the context we drop the superscript $p$ in this notation.)
Note that if $B$ is a non-singular abelian group and $\varphi(x)$ is a p.p.~formula of the form $ax=0$ where $a\in\bZ$, $a\neq 0$, then $\varphi(B)\sim 0$, whereas if $\varphi=\delta^p_{d,d'}$ for some prime $p$ and $d<d'$, then $\varphi(B)\sim B$.

\begin{lemma}\label{lem:dp-minimal abelian, 1}
Suppose $A=B\oplus C$ where $B$ is non-singular and
$$C= \bigoplus_{n>0} \bZ(p^n)^{(\alpha_{n-1})}\oplus\bZ(p^\infty)^{(\beta)}\oplus\bZ_{(p)}^{(\gamma)},$$
where $p$ is a prime and  $\alpha_{n-1}$, $\beta$, $\gamma$ are cardinals, with each $\alpha_{n-1}$ finite. Then $\widetilde{\PP}(A)$ is a chain.
\end{lemma}
\begin{proof}
By the remark preceding this lemma it is enough to show that 
$$\tau^p_e(C)\lesssim\delta^p_{d,d'}(C)\qquad\text{ for all $d<d'$ and all~$e$.}$$
Since
$$\delta^p_{d,d'}(\bZ(p^\infty))=\bZ(p^\infty),\qquad \delta^p_{d,d'}(\bZ_{(p)})=p^{d'-d}\bZ_{(p)}$$ 
as well as
$$\tau^p_e(\bZ(p^\infty))=\bZ(p^\infty)[p^e],\qquad \tau^p_e(\bZ_{(p)})=0,$$ 
this follows from (the proof of) Lemma~\ref{lem:asymp pp, chain}.
\end{proof}

Before we show that conversely, if $T$ is dp-minimal, then it is the complete theory of one of the groups described in (1) and  (2) of Proposition~\ref{prop:dp-minimal abelian}, we note:

\begin{lemma}\label{lem:dp-minimal abelian, 2}
Let $\beta$, $\gamma$ be cardinals and $k>0$ be an integer, and suppose
$$A=\bZ(p^k)^{(\aleph_0)}\oplus \bZ(p^\infty)^{(\beta)} \oplus \bZ_{(p)}^{(\gamma)}\oplus B$$
where $B$ is finite.
Then $A$ is dp-minimal iff $\beta=\gamma=0$. 
\end{lemma}
\begin{proof}
Since $A$ is dp-minimal iff $A/B$ is dp-minimal, by parts (4) and (5) of Corollary~\ref{cor:operations on dp-minimal modules}, we may assume that $B=0$. So if $\beta=\gamma=0$, then $A=\bZ(p^k)^{(\aleph_0)}$ is p.p.~uniserial and hence dp-minimal (by Corollaries~\ref{cor:pp uniserial} and \ref{cor:commensurable, 1}). Conversely, suppose $A$ is dp-minimal.
We have
$$A[p^{k-1}] = p\bZ(p^k)^{(\aleph_0)}\oplus (\bZ(p^\infty)[p^{k-1}])^{(\beta)}$$
and
$$p^kA = \bZ(p^\infty)^{(\beta)} \oplus p^k\bZ_{(p)}^{(\gamma)}$$
with intersection
$$A[p^{k-1}] \cap p^kA = (\bZ(p^\infty)[p^{k-1}])^{(\beta)},$$
so 
$$A[p^{k-1}]/ (A[p^{k-1}] \cap p^kA)\cong p\bZ(p^k)^{(\aleph_0)}$$ 
and
$$p^kA / (A[p^{k-1}] \cap p^kA)\cong (\bZ(p^\infty)/\bZ(p^\infty)[p^{k-1}])^{(\beta)}\oplus p^k\bZ_{(p)}^{(\gamma)}.$$
Thus $A[p^{k-1}] \cap p^kA$ has infinite index in both $A[p^{k-1}]$ and $p^kA$  iff $k>1$ and $\beta$ or $\gamma$ are non-zero.
By the equivalence of (1) and (4) in Corollary~\ref{cor:commensurable, 1},
this shows that, as $A$ is dp-minimal, we have $k=1$ or $\beta=\gamma=0$.
Suppose $k=1$; then
$$A[p] = \bZ(p)^{(\aleph_0)} \oplus \bZ(p^\infty)^{(\beta)}[p],\quad
pA = \bZ(p^\infty)^{(\beta)} \oplus p\bZ^{(\gamma)}_{(p)},\quad  
  A[p]\cap pA=\bZ(p^\infty)^{(\beta)}[p],$$
so $A[p]\cap pA$ has infinite index in $A[p]$, and $A[p]\cap pA$ has infinite index
in $pA$ iff $\beta>0$ or $\gamma>0$. Again, dp-minimality of $A$ and Corollary~\ref{cor:commensurable, 1}
yield $\beta=\gamma=0$.
\end{proof}

Now suppose $T$ is dp-minimal; so in particular $\vc^T(1)=1$ by Corollary~\ref{cor:commensurable, 1}. By Corollary~\ref{cor:abelian groups}, this implies that there is at most one prime singular for $A$, and $d(U_{\geq\aleph_0}(p;A))\leq 1$ for all $p$, with  $d(U_{\geq\aleph_0}(p;A)) = 1$ for at most one $p$. (Recall the definition of $d(\,\cdot\,)$ from Section~\ref{sec:finite exponent}.)
If no prime is singular for $A$, then $A$ is of type~(1) (with $p$ arbitrary and $\alpha_{n-1}=\beta=\gamma=0$ for all $n>0$), so we may assume that $p$ is the unique prime which is singular for~$A$.

We may and shall assume that $A$ is strict Szmielew as described before the proof of Theorem~\ref{thm:abelian groups}.
Employing the notation introduced in that proof, we then have 
$$A=A^{[\mathcal P]}\oplus A^{[p]}_{\aleph_0} \oplus A^{[p]}_{<\aleph_0}\oplus\bZ(p^\infty)^{(\beta_p)}\oplus\bZ_{(p)}^{(\gamma_p)}$$ 
where $A^{[\mathcal P]}$ is non-singular. 
Note that by Corollary~\ref{cor:operations on dp-minimal modules},~(2), each of the direct summands of $A$ is also dp-minimal.
If $U_{\geq\aleph_0}(p;A)=\emptyset$, then $A^{[p]}_{\aleph_0}=0$, and $A$ is as described in~(1). So from now on suppose 
that $U_{\geq\aleph_0}(p;A)$ is non-empty, hence $d(U_{\geq\aleph_0}(p;A)) = 1$, that is, $U_{\geq\aleph_0}(p;A)=\{k\}$ or $U_{\geq\aleph_0}(p;A)=\{k,k+1\}$, for some $k>0$.  Then, applying Lemma~\ref{lem:dp-minimal abelian, 2} to the direct summand 
$$\bZ(p^k)^{(\aleph_0)}\oplus\bZ(p^\infty)^{(\beta_p)}\oplus\bZ_{(p)}^{(\gamma_p)}$$ 
of $A$, we see that $\beta_p=\gamma_p=0$. Thus, to show that $A$ is of type (2), it is enough to prove that $B:=A^{[\mathcal P]}\oplus A^{[p]}_{<\aleph_0}$ is finite, and for this, after passing to a direct summand, we may assume that 
$$A=\bZ(p^k)^{(\aleph_0)}\oplus B.$$
Then by (the proof of) Lemma~\ref{lem:asymp pp, chain} and the remark before Lemma~\ref{lem:dp-minimal abelian, 1} we have
$$\delta^p_{k-1,k}(A) \sim p\bZ(p^k)^{(\aleph_0)} \oplus A^{[\mathcal P]}\oplus  p A^{[p]}_{<\aleph_0}$$
whereas
$$\tau^p_{k}(A)=A[p^k]\sim \bZ(p^k)^{(\aleph_0)} \oplus (A^{[p]}_{<\aleph_0})[p^k].$$
Hence 
$$\delta^p_{k-1,k}(A)\cap \tau^p_{k}(A) \sim p\bZ(p^k)^{(\aleph_0)} \oplus (A^{[p]}_{<\aleph_0})[p^k]$$
has infinite index in $\tau^p_{k}(A)$.
Now suppose for a contradiction that $B$ is infinite.
Then $\delta^p_{k-1,k}(A)\cap \tau^p_{k}(A)$ also has infinite index in  $\delta^p_{k-1,k}(A)$, contradicting dp-minimality of $A$.
This shows that indeed, $B$ is finite.

This finishes the proof of Proposition~\ref{prop:dp-minimal abelian}. \qed

\begin{corollary}
Suppose $T$ is totally transcendental. Then $T$ is dp-minimal iff~$A$ is elementarily equivalent to the direct sum of a finite abelian group with one of the following:
\begin{enumerate}
\item an infinite divisible group $D$ such that $D[p]$ infinite for at most one $p$;
\item $\bZ(p^k)^{(\aleph_0)}$ where $p$ is a prime and $k>0$; or
\item $\bZ(p^k)^{(\aleph_0)}\oplus \bZ(p^{k+1})^{(\aleph_0)}$ where $p$ is a prime and $k>0$.
\end{enumerate}
\end{corollary}
\begin{proof}
By a theorem of Macintyre (see \cite{mac-abelian} or \cite[Theorem~A.2.11]{Hodges}), $A=B\oplus D$ where $B$ has finite exponent and $D$ is divisible. The claim now follows from
Proposition~\ref{prop:dp-minimal abelian} and the observation that every non-singular abelian group of finite exponent is finite. 
\end{proof}

By \cite{OU}, a stable theory is dp-minimal iff all $1$-types have weight $1$. It would be interesting to obtain a proof of Proposition~\ref{prop:dp-minimal abelian} using only the methods of geometric stability theory.
It is instructive to see how Proposition~\ref{prop:dp-minimal abelian} entails the well-known fact (cf.~\cite[Fact~3.2]{OU}) that every $\URk$-rank $1$ abelian group is dp-minimal. The following lemma shows, in fact, that if a singular abelian group has $\URk$-rank~$1$, then it is of type~(2) in Proposition~\ref{prop:dp-minimal abelian}.

\begin{lemma}
$A$ has  $\URk$-rank $1$ iff $A$ is non-singular or of the form $A=\bZ(p)^{(\aleph_0)}\oplus B$ where $p$ is a prime and $B$ is finite.
\end{lemma}
\begin{proof}
Clearly if $A$ is non-singular, or if $A=\bZ(p)^{(\aleph_0)}\oplus B$ where $p$ is a prime and~$B$ is finite, then $A$ has $\URk$-rank~$1$. So assume conversely that $A$ has $\URk$-rank~$1$, and let~$p$ be a prime singular for $A$.
We are allowed to assume that $A$ is strict Szmielew as before. Suppose first that  $A[p]$ is infinite.
Then $A/A[p]$ is finite, since~$A$ has $\URk$-rank~$1$. This yields $\alpha_{p,n-1}<\aleph_0$ for all $n>1$, with $\alpha_{p,n-1}=0$ for all but finitely many $n>1$,  $\beta_p=\gamma_p=\nu=0$, and
$\bigoplus_{q\neq p} A^{[q]}$ is finite.
Hence $A$ is a direct sum of $\bZ(p)^{(\alpha_{p,0})}$ with a finite abelian group, as required.
Similarly, if $A/pA$ is infinite, then $pA$ is finite, and again we obtain the same conclusion.
\end{proof}

\subsection{Dp-minimal expansions of the group of integers.} \label{sec:dp-min expansions of Z}
The abelian group $(\bZ,{+})$ of integers, having $\URk$-rank $1$, is dp-minimal.
In fact, in \cite{ADHMS} we showed that the ordered abelian group $(\bZ,{<},{+})$  of integers is also dp-minimal, and has no proper dp-minimal expansions.
One might wonder whether the non-dp-minimality of the structures in question is caused by the presence of the ordering (due to our use, in the proof of this fact given in \cite{ADHMS}, of a lemma from \cite{Simon}, which specifically deals with ordered dp-minimal structures). We do not know the answer to the following:

\begin{question}
Is every dp-minimal expansion of $(\mathbb Z,{+})$ definable in $(\bZ,{<},{+})$?
\end{question}

We have a positive answer in several examples (which are well-known to have otherwise good model-theoretic properties). This is based on the following observation. (Below we will apply this lemma in the case $U=V$.)

\begin{lemma}\label{lem:not dp-min}
Let $(A,{+})$ be a commutative semigroup, let  $U\subseteq V$ be infinite subsets of~$A$, and $C>0$. Suppose for each $u\in U$ there are at most $C$ elements~$u'$ of $U$ such that $u+u'=v+v'$ for some $v,v'\in V$ with $\{u,u'\}\neq\{v,v'\}$.
Then $\mathbf A = (A,{+},V)$ is not dp-minimal.
\end{lemma}
\begin{proof}
Let $I=\{u,u'\}$ range over the $2$-element subsets of $U$. 
Let us say that  $I$ is \emph{good} if the only representation of $u+u'$ as a sum of two elements of $V$ is the given one, that is, if $v,v'\in V$ with $u+u'=v+v'$, then  $\{u,u'\}=\{v,v'\}$.
If $I$ is good, then $u+u'\notin v+V$, for each $v\in V\setminus I$. Consider the partitioned formula $$\varphi(x;y)=\exists z(z\in V\wedge  x=y+z)$$ 
in the language of $\mathbf A$. By what we have shown, for each choice of good $I=\{u,u'\}$ the element $u+u'$ satisfies
$$\{\varphi(x;u),\varphi(x;u')\}\cup\{\neg\varphi(x;v):v\in V\setminus I\}.$$
Each $u\in U$ is contained in no more than $C$ $2$-element subsets of $U$ which are not good. Hence
given $U'\subseteq U$ with $\abs{U'}=n$, at least ${n\choose 2}-Cn$ $2$-element subsets of~$U'$ are good. Since for each $r>0$ and $C'$ we have ${n\choose 2}-Cn>C'n^{2-1/r}$ for $n\gg 0$, we see by Lemma~\ref{lem:dp-min} that $\mathbf A$ is not dp-minimal.
\end{proof}

Recall that a set $U$ of elements of a commutative semigroup, written additively, is said to be a {\em Sidon set} if every element of $U+U$ can be written as a sum of two elements of $U$ in essentially a unique way, that is: if $u+v=u'+v'$ where $u,v,u',v'\in U$, then $\{u,v\}=\{u',v'\}$. (In the terminology of the proof above, all two-element subsets of $U$ are good.) We say that a set $U$ is a {\em weak Sidon set} if for all $u,v,u',v'\in U$ with $u+v=u'+v'$ we have $u=v$ or $u'=v'$ or $\{u,v\}=\{u',v'\}$.
From the previous lemma we immediately obtain:

\begin{corollary}\label{cor:weak Sidon}
Let $U$ be an infinite weak Sidon set in a commutative semigroup $(A,{+})$. Then $(A,{+},U)$ is not dp-minimal.
\end{corollary}

\begin{example*}
If $R$ is a commutative ring, $M$ is an $R$-module, and $U$ an infinite $R$-linearly independent subset of $M$, then $U$ is a weak Sidon set (and a Sidon set if $R$ is not of characteristic $2$); hence $(M,{+},U)$ is not dp-minimal. (This applies, e.g., to the structures introduced in the proof of Proposition~4.10 in \cite{ADHMS}.)
\end{example*}

Fast growing integer sequences form Sidon sets. More precisely, let $U$  be an infinite subset of $\bZ$, enumerated as $u_0<u_1<\cdots<u_n<u_{n+1}<\cdots$, and suppose $u_0>0$ and $u_{n+1}\geq 2u_n$ for each $n$; then $U$ is a Sidon set.

\begin{corollary}
The structure $(\mathbb Z,{+},U)$ is not dp-minimal if $U$ is one of the following:
\begin{enumerate}
\item $U=\{b^n:n\geq 0\}$ or $U=\{n+b^n:n\geq 0\}$, where $b\in\bN$, $b>1$;
\item $U=\{n!:n\geq 0\}$; or
\item $U=\{F_n:n\geq 0\}$ where $F_n$ is the $n$th Fibonacci number: $F_0=F_1=1$, and $F_{n+2}=F_{n+1}+F_n$ for each $n$.
\end{enumerate}
\end{corollary}
\begin{proof}
For parts (1) and (2) use Corollary~\ref{cor:weak Sidon} and the fact that the set of powers~$b^n$ and the set of integers of the form $n+b^n$, where  $b\in\bN$, $b>1$, as well as the set of factorials, 
are examples of Sidon sets. The set of Fibonacci numbers is not a Sidon set,
but it is a weak Sidon set, and hence Corollary~\ref{cor:weak Sidon} also applies. In fact, using Zeckendorf's~Theorem (every positive integer can be represented in a unique way as the sum of distinct Fibonacci numbers, no two of which are consecutive)
one shows easily that $a\in\bN$ has two representations as the sum of two Fibonacci numbers iff $a=2F_n$ for some $n\geq 2$, and in this case $a=F_{n+1}+F_{n-2}$ is the only other representation of $a$ as a sum of two Fibonacci numbers.
\end{proof}

In the examples considered in the previous corollary, the corresponding expansion of the ordered abelian group $(\bZ,{<},{+})$ of integers by a predicate symbol for $U$ has quantifier elimination in a natural expansion of the language $\{<,{+},U\}$ (see \cite{Cherlin-Point, Point}), and one can show that these expansions of $(\bZ,{<},{+})$ are NIP.

\medskip
\noindent
We finish by noting that Lemma~\ref{lem:not dp-min} also yields a result about expansions of fields by a subgroup of their multiplicative group.
For this, let $K$ be a field and $G$ be an infinite subgroup of~$K^\times$.

\begin{corollary}\label{cor:Mann}
If $G$ has the Mann property, then $(K,{+},G)$ is not dp-minimal.
\end{corollary}

Recall that $G$ is said the have the \emph{Mann property} if every equation
$$a_1x_1+\cdots+a_nx_n=1$$
where $n\geq 2$ and $a_1,\dots,a_n$ are non-zero coefficients in the prime field of $K$,
has only finitely many non-degenerate solutions in $G$, i.e., solutions $(g_1,\dots,g_n)\in G^n$ such that $\sum_{i\in I} a_ig_i\neq 0$ for each
non-empty $I\subseteq [n]$. (Examples for multiplicative groups with the Mann property include all finite-rank subgroups of $K^\times$ if~$K$ is algebraically closed of characteristic zero, by \cite{ESS}.)
In \cite{vdDG}, van den Dries and G{\"u}nayd{\i}n study the model theory of pairs $(K,G)$ (in the language of fields expanded by a unary predicate symbol) where $K$ is algebraically closed or real closed and $G$ has the Mann property; in particular, they show that if $K$ is algebraically closed and~$G$ has the Mann property, then $(K,G)$ is stable \cite[Corollary 6.2]{vdDG}. 

To prove Corollary~\ref{cor:Mann}, suppose $G$ has the Mann property  and let $C$ be a bound on the number of non-degenerate solutions of the equation $x_1+x_2-x_3=1$ in $G$.
Then for each $u,v,u',v'\in G$ with $u+v=u'+v'$ and $\{u,v\}\neq\{u',v'\}$, either $u+v=0$ or the triple $(u'/u,v'/u,v/u)$ is a non-degenerate solution of $x_1+x_2-x_3=1$.
Hence taking $(K,{+})$, $G$, and $C+1$ for $(A,{+})$, $U$ and $C$, respectively, the hypothesis of Lemma~\ref{lem:not dp-min} is satisfied. \qed

\medskip
\noindent
Note that Corollary~\ref{cor:Mann} is mainly interesting if $K$ is not real closed, since Simon \cite[Corollary~3.7]{Simon} has shown that an expansion of a (linearly) ordered group is o-minimal if  it is dp-minimal and definably complete and its underlying group is divisible. (These conditions are obviously necessary for o-minimality.) In particular, the only dp-minimal expansions of the ordered group $(\mathbb R,{+},{<})$ of real numbers are the o-minimal ones.

\medskip
\noindent
If $K$ is algebraically closed and $G$ has the Mann property, then $(K,G)$ is not dp-minimal; this
consequence of Corollary~\ref{cor:Mann} may be strengthened as follows. (That a sufficiently saturated structure $\mathbf M=(K,G)$ where $K$ is algebraically closed satisfies the hypothesis of the lemma below with $X=G$ follows from \cite[Lemmas~2.2~(3)~and~6.1]{vdDG}.)

\begin{lemma}  Let $\mathbf M$ be a structure which expands a field.  Suppose that $X$ is an infinite definable subset of $M$
so that the transcendence degree of $M$ over its subfield generated by $X$ is at least $n$.  Then there is a definable  one-to-one function $X^{n+1} \to M$, hence the dp-rank of $\Th(\mathbf M)$ is at least $n+1$.
\end{lemma}

\begin{proof}  Pick  elements $e_1, \dots, e_n\in M$ which are algebraically independent over the subfield of $M$ generated by $X$.
Consider the function $X^{n+1} \to M$ given by $(x_1, \dots, x_{n+1})\mapsto x_1e_1 + \cdots +x_ne_n + x_{n+1}$.  It is easy
to see that this function is one-to-one. The claim thus follows from Lemma~\ref{lem:dp-rank}.
\end{proof}

We plan to systematically investigate the VC~density functions of pairs of structures (including the expansions of fields by groups with the Mann property studied in \cite{vdDG}) at another occasion.

\bibliographystyle{jflnat}

\end{document}